\documentclass[11pt,a4paper]{article}
\usepackage[margin=1in]{geometry}
\usepackage{amsfonts}
\usepackage{bbding}
\usepackage{graphicx,float}
\usepackage{amsfonts,amssymb,amsmath,amsthm,amscd}
\usepackage{mathrsfs}
\usepackage{xcolor}
\usepackage{empheq}
\usepackage{adforn}
\usepackage{cancel}
\usepackage{xr-hyper}
\usepackage{xr}
\usepackage{mdframed}
\usepackage{mathdots}
\usepackage{enumitem}
\usepackage{lmodern}
\usepackage{mdframed}
\usepackage{stmaryrd}
\usepackage{blindtext}
\usepackage{leftidx}
\usepackage{accents}
\usepackage{appendix}
\usepackage{cases}
\usepackage{bbm}
\usepackage{authblk}
\usepackage{subfigure}

\usepackage[colorlinks=true,  linkcolor=, citecolor=,  urlcolor=]{hyperref}

\usepackage[T1]{fontenc}

\newtheorem{theorem}{Theorem}[section] \newtheorem{lemma}[theorem]{Lemma}
\newtheorem{proposition}[theorem]{Proposition}

\theoremstyle{definition} 
\newtheorem{definition}[theorem]{Definition}

\newtheorem{remark}[theorem]{Remark} \numberwithin{equation}{section}
\numberwithin{figure}{section}

\newcommand{\Cb}{\mathbb{C}}

\newcommand{\Eb}{\mathbb{E}}

\newcommand{\Pb}{\mathbb{P}}
\newcommand{\Qb}{\mathbb{Q}}
\newcommand{\Rb}{\mathbb{R}}

\newcommand{\Zb}{\mathbb{Z}}

\newcommand{\Pbf}{\mathbf{P}}

\newcommand{\Kbf}{\mathbf{K}}
\newcommand{\Pbx}{\mathbb{P}_{\vec{x}}}
\newcommand{\unionW}{\mathcal{W}}

\newcommand{\Pf}{\mathbf{P}}

\newcommand{\Ac}{\mathcal{A}}
\newcommand{\Bc}{\mathcal{B}}
\newcommand{\Cc}{\mathcal{C}}
\newcommand{\Dc}{\mathcal{D}}

\newcommand{\Nc}{\mathcal{N}}

\newcommand{\Uc}{\mathcal{U}}

\newcommand{\Wc}{\mathcal{W}}

\newcommand{\Yc}{\mathcal{Y}}

\newcommand{\Acc}{\mathcal{A}}
\newcommand{\Jcc}{\mathcal{J}}

\usepackage[section]{placeins}
\usepackage{wrapfig}
\usepackage{lpic}

\newcommand{\ol}{\overline}

\def \eps {\varepsilon}

\newcommand{\extendSubSet}{T}
\newcommand{\Ball}{\widetilde{B}}
\newcommand{\extendNearSubSet}{V}
\newcommand{\extendSquare}{S}
\newcommand{\extendSet}{T}
\newcommand{\Tube}{T}
\newcommand{\TubeComplement}{\widetilde\Tube}
\newcommand{\TubeGroup}{\mathscr{T}}
\newcommand{\Cone}{\mathsf{C}}
\newcommand{\Cylinder}{\mathsf{D}}
\newcommand{\dist}{\mathrm{dist}}
\newcommand{\B}{B}
\newcommand{\CAP}{\mathrm{Cap}}
\newcommand{\hpro}{X}
\newcommand{\T}[2]{T_{#2}(#1)}
\newcommand{\sausage}[2]{\B(#1,#2)}
\newcommand{\con}{\,\Big |\,} 
\newcommand{\tgp}{\widetilde\gamma '{}}
\newcommand{\hgp}{\widehat\gamma '{}}
\newcommand{\olgp}{\overline\gamma '{}}
\newcommand{\tg}{\widetilde\gamma}
\newcommand{\hg}{\widehat\gamma}
\newcommand{\olg}{\overline\gamma}
\newcommand{\hA}{\widehat A}
\newcommand{\rQ}{Q}
\newcommand{\ind}[1]{{\mathbf{1}{\{#1\}}}}

\title{A boundary Harnack principle and its application to analyticity of 3D Brownian intersection exponents}
\author[1]{Yifan Gao\thanks{yifangao@cityu.edu.hk}}
\author[2]{Xinyi Li\thanks{xinyili@bicmr.pku.edu.cn 
Research supported by National Key R\&D Program of China (No.\ 2021YFA1002700 and No.\ 2020YFA0712900) and NSFC (No.\ 12071012).
}}
\author[2]{Yifan Li\thanks{2100015801@stu.pku.edu.cn}}
\author[2]{Runsheng Liu\thanks{liurunsheng@pku.edu.cn}}
\author[2]{Xiangyi Liu\thanks{liuxiangyi@stu.pku.edu.cn}}

\affil[1]{City University of Hong Kong}
\affil[2]{Peking University}

\date{}
\begin{document}

\maketitle
\begin{abstract}
    We show that a domain in $\mathbb{R}^3$ with the trace of a 3D Brownian motion removed almost surely satisfies the boundary Harnack principle (BHP). Then,
    we use it to prove that the intersection exponents for 3D Brownian motion are analytic. 
    
    \bigskip
    
    \textit{Key words and phrases: Brownian motion, analyticity, intersection exponent, boundary Harnack principle, twisted H\"older domain, separation.}
\end{abstract}

\tableofcontents

\section{Introduction}
In this paper, we prove that the intersection exponents for 3D Brownian motions  are analytic. This property has been shown to hold in 2D by Lawler, Schramm and Werner in \cite{LSW02a}. Due to the absence of a rich conformal symmetry in 3D, the strategy used in the 2D case cannot be applied directly, and thus making the proof of analyticity in 3D a long-standing open problem. 

By analyzing the fine geometric structure of the random slit domain in $\Rb^3$, whose boundary is given by the trace of a Brownian motion, we establish a conditional version of separation lemma. As we will show in this paper, it is indeed equivalent to a boundary Harnack principle (BHP) for the slit domain. Surprisingly, this BHP in 3D serves as an effective substitute for the conformal invariance in 2D to derive the analyticity of Brownian intersection exponents.

Section~\ref{subsec:analyticity} is devoted to the result of analyticity of Brownian intersection exponents. In Section~\ref{subsec:csl}, we describe the conditional separation lemma (CSL). In Section~\ref{subsec:bhp}, we state a particular type of BHP we obtained. We outline the proof strategy in Section~\ref{subsec:strategy}, and give the structure of the paper in Section~\ref{subsec:structure}.

\subsection{Analyticity in dimension three}\label{subsec:analyticity}
We begin by reviewing the Brownian intersection exponents in $\Rb^d$, $d=2,3$. Let $k$ be a positive integer, and let $W_1, \dots, W_k$ denote $k$ independent Brownian motions, each started at the origin. Additionally, let $W_0$ be another independent Brownian motion, this time started at $(1,0)$ in 2D or $(1,0,0)$ in 3D. We use $\Pb$ and $\Eb$ to denote the respective probability measure and expectation for these processes. Let $\Bc_n$ be the ball of radius $e^n$ about $0$, and $\tau^i_n$ be the first hitting time of $\partial\Bc_n$ by $W_i$ for all $0\le i\le k$. For $n\ge 0$, consider the following conditional probability
\begin{equation*}
Z_n:=\Pb\Big(W_0[0,\tau^0_n]\cap \big(W_1[0,\tau^1_n]\cup\ldots\cup W_k[0,\tau^k_n]\big)=\varnothing \ \Big| \ W_1[0,\tau^1_n],\ldots,W_k[0,\tau^k_n]\Big).  
\end{equation*}
For all real $\lambda>0$, the intersection exponent in $d=2,3$ is defined by 
\begin{equation}\label{eq:exponent}
    \xi_d(k,\lambda):=-\lim_{n\rightarrow\infty}\frac{\log \Eb[Z_n^\lambda]}{n},
\end{equation}
where the existence of this limit follows from a standard subadditivity argument.

The intersection exponents are closely linked to other critical exponents that arise in statistical physics, such as those in critical planar percolation and self-avoiding walks (see \cite{MR1796962, MR1905353, LSW01a, LSW02a, MR2112127} for detailed discussions). In particular, the values of $\xi_2(k,\lambda)$ were originally predicted by Duplantier and Kwon \cite{DK88} using methods from conformal field theory. Alternative, non-rigorous derivations based on quantum gravity also supported these predictions (see \cite{Du98, duplantier1999two}). Through the connection between planar Brownian motion and Schramm-Loewner evolution (SLE$_6$), along with the conformal invariance of Brownian motion, Lawler, Schramm and Werner provided a rigorous proof of this conjecture in a series of celebrated works \cite{LSW01a,LSW01b,LSW02,LSW02a}. In particular, in \cite{LSW02a}, they established that $\xi_2(k,\lambda)$ is a real analytic function of $\lambda$ on the interval $(0, \infty)$, a result which plays a crucial role in determining the values of these exponents across the entire range of $\lambda$.

However, little is known about precise values of $\xi_3(k,\lambda)$, except the particular case $\xi_{3}(1,2)=1$; see \cite{MR998666,BL90}. Although determining all values of $\xi_3(k,\lambda)$ is still out of reach, Lawler \cite{G98} showed that $\xi_3(k,\lambda)$ has a continuous, negative second derivative (in $\lambda$), based on basic properties of Brownian motions. Nevertheless, the analyticity of $\xi_3(k,\lambda)$ in $\lambda$ remains open to this day. Several preparations have been made; see \cite{vermesi2008intersection,LV12}.
But as pointed out in these works, there are still essential gaps. In this paper, we resolve this long-standing open problem with the following theorem, which confirms the analyticity of $\xi_3(k, \lambda)$ for all positive integers $k$.

\begin{theorem}[Analyticity]\label{thm:analyticity}
    For all integers $k\geq 1$, the function $\lambda \mapsto \xi_3(k,\lambda)$ is real analytic on $(0,\infty)$.
\end{theorem}

The proof of Theorem~\ref{thm:analyticity} relies on the BHP of the slit domain ``$\Rb^3$ with the trace of a BM removed'' (see Section~\ref{subsec:bhp} below), which is a consequence of the CSL, as we illustrate in Section~\ref{subsec:csl} below.

\subsection{Conditional separation lemma}\label{subsec:csl}
In this subsection, we state the CSL in more detail. 
Intuitively, if we ``freeze'' several Brownian motions and then consider another independent Brownian motion conditioned to avoid the frozen ones, the conditioned Brownian motion will be well-separated from the frozen paths with positive probability, uniformly in its starting point. As mentioned, it can be viewed as an equivalent probabilistic interpretation of the BHP (see Section~\ref{subsec:equi}). We believe it is of independent interest and has potential applications in many other problems.  

We introduce some notation first, which will be used throughout the paper.
For $k\ge1$ and any $k$ points $x_1,\ldots,x_k$ in $\Rb^3$, let $(W_i)_{1\le i\le k}$ be $k$ independent Brownian motions started from $x_i$, respectively. We use $\Pb_{\vec x}:=\prod_{i=1}^k\mathbb P_{x_i}$ to denote the product measure of these $k$ Brownian motions. Let $\unionW:=\cup_{i=1}^k W_i[0,\infty)$ be the union of their traces. For $x\in\Rb^3$, let $W$ be another independent Brownian motion started from $x$, and we denote its law by $\Pbf_x$.
For all $y\in\Rb^3$ and $r>0$, let $B(y,r)$ be the closed ball of radius $r$ around $y$.
For $A\subseteq \Rb^3$, define $\dist(y,A):=\inf\{|y-z|:z\in A\}$. 

\begin{theorem}[CSL] \label{thm:sep}
    For $\Pb_{\vec x}$-almost  every $\unionW$, there exist $\delta_1(\unionW)$, $\delta_2(\unionW)>0$ such that for any closed subset $\Ac\subseteq\unionW$ and all $x\in B(0,1/2)\setminus \Ac$, we have 
    \begin{equation}\label{eq:csl}
\Pbf_x\left(\dist(W(\tau),\Ac)>\delta_1 \con  W[0,\tau]\cap \Ac=\varnothing\right)>\delta_2,
    \end{equation}
    where $\tau$ is the first hitting time of $\partial \B(0,1)$ by $W$.
\end{theorem}

\begin{remark}
    In fact, the proof of Theorem~\ref{thm:analyticity} only needs the result of Theorem~\ref{thm:sep} with $\Ac=\unionW$. However, we provide a stronger result for arbitrary closed $\Ac\subseteq\unionW$ since such generalization incurs no extra difficulty. More importantly, we can use it to study related problems associated with subsets of $3$D Brownian motions, such as the ``loop-erasure'' of Brownian motion (which we denote by $\Kbf$) given by Kozma's \cite{MR2350070} scaling limit of loop-erased random walk in three dimensions (see \cite{MR3877544} for a characterization). More precisely, taking $\Ac=\Kbf$ in Theorem~\ref{thm:sep}, we see that the CSL holds for $\Kbf$, and hence the BHP (Theorem~\ref{theorem:bhp}) also holds true for $\Kbf$. If one defines the
    intersection exponents $\eta(k,\lambda)$ associated with $\Kbf$, as in \eqref{eq:exponent} by keeping the Brownian motion $W_0$ unchanged and replacing the other $(W_i)_{1\le i\le k}$ with $k$ independent copies of $\Kbf$, then we can also show that $\lambda \mapsto \eta(k,\lambda)$ is real analytic on $(0,\infty)$. In particular, $\beta=2-\eta(1,1)$, also known as the growth exponent, can be used to describe the growth rate of the length of a $3$D loop-erased random walk, and has been extensively studied in \cite{LERW3exp,MR4017129}.
\end{remark}

\begin{remark}
A corresponding version of the CSL can be formulated in two dimensions, where it follows directly from conformal invariance. In contrast, the CSL in three dimensions is significantly more challenging to establish, and this difficulty is also reflected in the BHP. Notably, a similar separation idea was implicitly used by Lawler, Schramm, and Werner \cite{LSW02a} when they coupled two $h$-processes in appropriate slit domains. A detailed explanation of this coupling strategy is provided in Section~\ref{subsec:strategy}.

We also emphasize that the above \emph{conditional} separation lemma is different from the so-called separation lemmas (i.e., non-conditional ones; see Lemma~\ref{lem:sep} for a version) in the literature \cite{Law96,G98,LV12}. They are not implied by each other. 
\end{remark}

For completeness of exposition, in the following, we briefly summarize the development of classical (non-conditional) separation lemmas.
Broadly speaking, the separation lemma asserts that if two random processes are conditioned not to intersect, then they have a positive probability of being ``well-separated'' at their endpoints. This separation phenomenon has been observed and rigorously established in various statistical physics models, where it is instrumental in deriving quasi-multiplicativity properties of certain probabilities and in determining the Hausdorff dimension of specific random fractals. Given the extensive literature on this topic, we provide only a selective overview of separation lemmas across different models.

This type of separation result was first employed by Kesten \cite{Ke1987a} in studying arm events in two-dimensional percolation (see, e.g., \cite{BN2021,SS2010,DS11} for subsequent alternative proofs and generalizations). In a distinct context, Lawler \cite{Law96,Law96a,G98} established the separation lemma for non-intersecting Brownian motions, which served as a key ingredient for determining the Hausdorff dimension of cut points and the Brownian frontier. Lawler's approach was later streamlined and applied to show strong separation properties for other models in two and three dimensions, including random walks \cite{RWcuttimes}, loop-erased random walks \cite{Ma2009,LERW3exp,La2020}, and Brownian loop soups \cite{GLQ2022}. 

A unified approach that synthesizes ideas from Kesten and Lawler was presented in the appendix of \cite{GPS2013}. Recently, this refined approach has been adapted to prove stronger versions of separation lemmas in critical planar percolation \cite{du2022sharp}, as well as variations for random walks within a random walk loop soup \cite{GNQ2024a}.

\subsection{Boundary Harnack principle }\label{subsec:bhp}
As an application of Theorem \ref{thm:sep}, we obtain the following version of BHP. 

\begin{theorem}[BHP]\label{theorem:bhp}
    Let $U$ be a domain in $\Rb^3$ and $K$ a compact set such that $K\subseteq U$. For $\Pb_{\vec x}$-almost every $\unionW$, there exists $C(U,K,\unionW)\in (0,\infty)$ such that for all closed subset $\Ac\subseteq\unionW$, if $u$ and $v$ are bounded positive harmonic functions in $U\setminus \Ac$ that vanish continuously on regular points of $\Ac\cap U$, then
    \begin{equation}\label{eq:bhp}
        u(x_1)\,v(x_2)\leq C\,u(x_2)\,v(x_1) \quad \text{for all $x_1,x_2\in K\setminus \Ac$}.
    \end{equation}
\end{theorem}

In fact, one can also derive Theorem~\ref{thm:sep} from Theorem~\ref{theorem:bhp}; see Section~\ref{subsec:equi}.

\begin{remark}
    We mention that Theorems \ref{thm:sep} and~\ref{theorem:bhp} also hold for the 3D Brownian fabric of any positive intensity (introduced along with Brownian interlacements in \cite{sznitman2013scaling})  in place of $\Wc$, since a.s.\ there are only finitely many trajectories that intersect a bounded domain. 
\end{remark}

\begin{remark}
    As shown in \cite{MR2464701}, the BHP and the Carleson estimate are equivalent for arbitrary domains. Hence, we can obtain the Carleson estimate as follows. Under the setting of Theorem~\ref{theorem:bhp} and letting $x_0$ be a point in $K\setminus \Ac$,
    \begin{equation}
    \mbox{there exists $C(U,K,\Ac,x_0)\in (0,\infty)$ such that $u(x)\le C\, u(x_0)$ for all $x\in K\setminus \Ac$.}    
    \end{equation}
    We believe that with more efforts one can also make $C$ only depend on $U,K,\Wc,x_0$ (but independent of $\Ac$) in this Carleson estimate.
\end{remark}

\begin{remark}
    A \emph{uniform} BHP was established for uniform domains in \cite{MR1800526} (see e.g.\ \cite{MR2204573,MR4482110} for related considerations). However, we expect that $U\setminus\Wc$ does not satisfy the uniform BHP.
\end{remark}

The BHP has a long history. It is hard to be exhaustive, and we only mention some related literature. For Lipschitz domains, the BHP was first proved in \cite{Da77,An78} in the late 1970's (see \cite{MR513884,MR716504} for alternate proofs and \cite{MR1042338} in particular for a probabilistic proof). It was extended to nontangentially accessible (NTA) domains in \cite{MR676988,MR946350}. In the seminal work of Bass and Burdzy \cite{MR1127476}, the twisted H\"older domains of order $\alpha$, $\alpha\in(0,1]$, were introduced and the BHP was shown to hold in such domains with $\alpha\in(1/2,1]$, based on probabilistic methods (where it was also showed that $1/2$ is in some sense critical). The twisted H\"older domains contain a large class of domains. When $\alpha=1$, it is equivalent to the class of John domains, hence it contains NTA domains and uniform domains \cite{MR1131398}. In particular, \cite{MR1127476} implies that the BHP holds for H\"older domains of order $\alpha\in(1/2,1]$, as a subclass of twisted H\"older domains. Based on a similar probabilistic approach, the BHP was extended to all H\"older domains, $\alpha\in(0,1]$, in \cite{MR1131398} subsequently. An analytic proof was later given in \cite{MR1658624} (see also Remark 3 in \cite{MR1800526}). We also refer to \cite{MR4093736} for a recent short unified analytic proof to various domains.

We will derive the following result in Section~\ref{sec:twisted}, which is of independent interest.
\begin{theorem}\label{thm:twisted}
    Let $U$ be a domain. $\Pb_{\vec x}$-almost surely, for every closed subset $\Ac$ of $\unionW$, $U\setminus \Ac$ is a twisted H\"older domain of any order $\alpha<1$. 
\end{theorem}

Since a twisted H\"older domain of order $\alpha>1/2$ satisfies the BHP by Theorem 4.4 of \cite{MR1127476},  
combined with Theorem~\ref{thm:twisted}, we can also obtain Theorem~\ref{theorem:bhp}, but the constant $C$ in \eqref{eq:bhp} derived in this way further depends on $\Ac$. 
This provides an alternative way to show Theorem~\ref{thm:analyticity} (see Figure~\ref{fig:label1}),
which in fact only needs to use Theorem~\ref{theorem:bhp} in the case $\Ac=\Wc$. Moreover, we believe that $U\setminus \Wc$ is \emph{not} a  twisted Lipschitz domain, i.e., twisted H\"older domain of order $\alpha=1$, and leave its proof to the interested reader.

\subsection{Outline of the proof}\label{subsec:strategy}
Our proof will follow the structure outlined in the diagram in Figure~\ref{fig:label1}. 
\begin{figure}[ht]
    \centering
\includegraphics[width=0.6\linewidth]{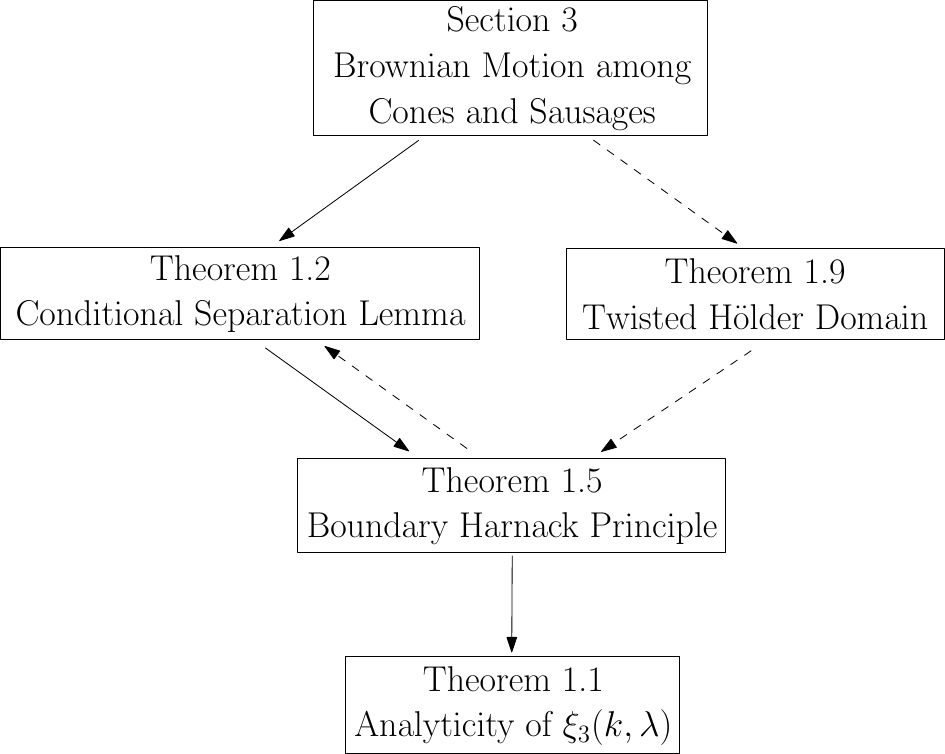} 
    \caption{The workflow diagram. We mainly follow the solid arrows to arrive at Theorem~\ref{thm:analyticity}. An alternate route via twisted H\"older domain is illustrated by dashed arrows. The direction from Theorem~\ref{theorem:bhp} to Theorem~\ref{thm:sep} is less relevant, and hence also dashed.}
    \label{fig:label1}
\end{figure}

We first explain the proof of Theorem~\ref{thm:sep}, which mainly relies on the following estimate established in Section~\ref{sec:cone}. Recall the setup above Theorem~\ref{thm:sep}. Except for a negligible event for $\Wc$ (by negligible we mean its probability under $\Pb_{\vec x}$ is $O(9^{-n})$), for Brownian motion \( W \) starting from any point \( x \) at a distance \( 2^{-n-1} \) from \( \Wc \), we establish the following two estimates, where \( a_1 \) and \( a_2 \) are universal constants in \( (0,1) \).
\begin{enumerate}[label=(\roman*)]
    \item $W[0,\tau]$ does not intersect $\Wc$ with probability at least $a_1^{1.1^n}$.
    \item $W$ stays in the $2^{-n+1}$-sausage of $\Wc$ before exiting $B(x,(3/4)^n)$ with probability bounded from above by $a_2^{1.2^n}$.
\end{enumerate}

For (i), we start by finding a cone with vertex $x$ and radius $u_1^{-n}$ (with some adjustable constant $u_1>1$) avoiding $\Wc$ in the unit ball (which will be called an uncovered cone) and thus the proof boils down to a cover time argument; see Proposition~\ref{prop:cone1} (a more detailed account is provided just below it). We then obtain (i) by forcing $W$ to stay in the uncovered cone (Proposition~\ref{prop:lowEstimate}). The second point (ii) is much easier to get by using a standard large deviation argument; see Lemma \ref{lem:1}.

Combining (i) and (ii), we conclude that conditioned on $W[0,\tau]\cap\Wc=\varnothing$, $W$ will get $2^{-n+1}$ away from $\Wc$ before it travels a distance of $(3/4)^n$, with probability at least $(1-a_2^{1.2^n}/a_1^{1.1^n})$. Since both the spacings $(3/4)^n$ and the probability errors $a_2^{1.2^n}/a_1^{1.1^n}$ are summable sequences, with probability bounded from below, the distance between the ending point of $W$ and the whole range $\Wc$ can improve from any ``microscopic'' $2^{-n}$ to some ``macroscopic'' $\delta_1>0$ before the time $\tau$. This concludes Theorem~\ref{thm:sep}. In fact,  we cheated a bit here in exchange of clarity of explanation, that is, the negligible event can depend on the starting point of $W$. To handle this issue, we discretize, approximating the unit ball by \( O(8^n) \) lattice points in \( \Zb^3/2^{n+2} \). The strategy works because a union bound is valid here for the $O(8^n)$ associated negligible events, which results in a (still summable) error of $O((8/9)^n)$. 

Theorem~\ref{thm:twisted} can be also derived from a combination of (i) and (ii).
The equivalence between Theorem~\ref{thm:sep} and Theorem~\ref{theorem:bhp} is not hard to establish by considering the Dirichlet problem. Hence, we will focus on the proof of Theorem~\ref{thm:analyticity} below, with the aid of Theorem~\ref{theorem:bhp}.

For clarity of explanation, we will only discuss the case $k=1$. Let $m'\ge m\ge 1$. Let $\gamma$ and $\gamma'$ be two continuous curves started from some points in $\Bc_{-m'}^\circ\setminus\{0\}$ and ended at their first visits of $\partial\Bc_0$, respectively, such that they agree with each other after their first visits of $\partial\Bc_{-m'}$. Consider two $h$-processes $X$ and $X'$, given by two Brownian motions started from $0$ and conditioned to avoid $\gamma$ and $\gamma'$ respectively until exiting $\Bc_0$. 
Let $D_i:=\Bc^\circ_{-i}\setminus(\Bc_{-i-1}\cup\gamma)$, $0\le i\le m'-1$ be $m'$ layers. 
A key step in proving analyticity is to couple \( X \) and \( X' \) so that they coincide after hitting \( \partial\Bc_{-m/2} \), with probability at least \( (1 - O(e^{-cm})) \), given at least \( m \) ``good'' layers.
In fact, the main task is to find a suitable definition of good layers such that the coupling will succeed with positive probability at each good layer. In $2$D, Lawler, Schramm and Werner in \cite{LSW02a} used ``down-crossings'' to measure the layers, and constructed the desired coupling by exploiting the conformal invariance in the plane (see Proposition~4.1 in \cite{LSW02a}). However, it does not work in the three-dimensional case. The main novelty of this part of our work is to characterize a good layer by its associated optimal comparison constant of BHP (i.e., the infimum $C$ such that \eqref{eq:bhp} holds).

Before getting to the definition of good layers, we will first explain how the BHP comes into play. Let $p=p_i(x,y)$ be the Poisson kernel in some layer $D_i$ restricted to $\partial\Bc_{-i-1/2}\setminus\gamma\times\partial\Bc_{-i}\setminus\gamma$ (see \eqref{eq:pb}). Then, it is not hard to see that (Lemma \ref{lem:5.17}) the optimal comparison constant of BHP in $D_i$ is equal to $K(p)$, the ``switching constant'' of $p$ (see \eqref{eq:K} for the definition). Another important observation is that $K(p)$ can be related to the ``extremal total variation distance'' $R(p)$ with respect to $p$ (see \eqref{eq:R}) through the following formula (see Lemma~\ref{lemma:RK})
\begin{equation}\label{eq:RK}
        R(p)=1-\frac{2}{1+\sqrt{K(p)}}.
\end{equation}
The quantity $R(p)$ can be used to bound the total variation distance between two measures produced by any positive finite measure $\mu$ on $\partial\Bc_{-i}\setminus\gamma$ reweighted by $p(x_1,\cdot)$ and $p(x_2,\cdot)$, respectively, for any pair  $x_1,x_2 \in \partial\Bc_{-i-1/2}\setminus\gamma$. Note that the law of the part of $h$-process $X$ after its first visit of $\partial\Bc_{-i}$ can be represented as some measure reweighted by $p(x,y)$. Hence, $X$ and $X'$ can be coupled to coincide with each other after their first visits of $\partial\Bc_{-i}$ with probability at least $2/(1+\sqrt{K(p)})$ by the maximal coupling.
Therefore, our definition of good layers can be made as follows.
Let $M$ be some large constant.
We say the layer $D_i$ is {\it good} if $K(p_i)< M$. Roughly speaking, if there are $m$ good layers, and the coupling has a positive chance to succeed at each good layer, then it will succeed before reaching $\partial\Bc_{-m/2}$ with probability at least $(1-O(e^{-cm}))$. This completes the construction of the desired coupling (Proposition~\ref{pro:1}). 

We would like to mention that the two functionals $K(\cdot)$ and $R(\cdot)$ can be defined for general non-negative Borel measurable functions on any product measurable space (although we only need to take $p$ as a Poisson kernel for the proof of Theorem~\ref{thm:analyticity}). The formula \eqref{eq:RK} still holds in this general setting, as well as some other useful properties. We believe they are of independent interest and refer to Section~\ref{subsection:RK} for details.

As another important intermediate step, we need to show that except for an event of probability $O(e^{-cm})$, if we take $\gamma$ as a Brownian motion started from $\partial\Bc_{-m'}$ and stopped upon reaching $\partial\Bc_0$, then it contains at least $m/2$ good layers. Now, the Poisson kernel $p_i$ in the layer $D_i$ associated with $\gamma$ is a random function. Since $K(p_i)$ is equal to the optimal comparison constant of BHP in $D_i$ as mentioned,   according to Theorem~\ref{theorem:bhp}, the probability $\Pb(K(p_i)<M)$ can be made arbitrarily close to $1$ by choosing $M$ sufficiently large. Finally, we obtain the desired estimate by decoupling the scales and using a large deviation argument (Proposition~\ref{pro:2}).

Together, these two estimates and additional inputs from Section~\ref{sec:analyticity} complete the proof of Theorem~\ref{thm:analyticity}, using a similar strategy as in \cite{LSW02a}. We do not repeat it here.

\subsection{Organization of the paper}\label{subsec:structure}
The paper is organized in the following way. We introduce the notation and review some elementary results for Brownian motion in Section~\ref{sec:notation}. In Section~\ref{sec:cone}, we derive various estimates on Brownian motions. In particular, we show that with high probability, there exists a cone not intersected by Brownian motions, which leads to a useful lower bound on the conditional non-intersection probability. In Section~\ref{sec:sep_bhp}, we use the estimates in Section~\ref{sec:cone} to prove the CSL (Theorem~\ref{thm:sep}), which leads to the BHP (Theorem~\ref{theorem:bhp}) by a deterministic equivalence result. In Section~\ref{sec:analyticity}, we use Theorem~\ref{theorem:bhp} to prove that the intersection exponents are analytic (Theorem~\ref{thm:analyticity}). Finally, in Section~\ref{sec:twisted}, we prove Theorem~\ref{thm:twisted} by using some estimates in Section~\ref{sec:cone}.

\bigskip

\noindent {\bf Acknowledgments:} The authors thank Greg Lawler for suggesting the problem and very helpful suggestions. Yifan Li wishes to thank Yuanzheng Wang for fruitful discussions. Yifan Gao is supported by a GRF grant from the Research Grants Council of the Hong Kong SAR (project CityU11307320). Xinyi Li acknowledges the support of National Key R\&D Program of China (No.\ 2021YFA1002700 and No.\ 2020YFA0712900) and NSFC (No.\ 12071012).

\section{Notation and preliminary results}\label{sec:notation}

\subsection{Notation}\label{subsec:notation}

\paragraph{Sets in $\mathbb R^3$.}

Let $B(x,r):=\{y:|x-y|\leq r\}$ denote a closed ball in $\Rb^3$, and $\mathcal{B}_{r}:=\B(0,e^{r})$ the exponential ball around $0$. Let $\sigma_l(\cdot)$ be the surface measure on $\partial\mathcal B_l$. For a set $A$ in $\Rb^3$, we use $A^\circ$ to denote its interior and $\ol A$ to denote its closure. The $r$-sausage of $A$ is denoted by 
\begin{equation}\label{eq:sausage}
    \sausage{A}{r}:=\{x\in\mathbb R^3:\dist(x,A)\leq r\}.
\end{equation}
Let $v\in\mathbb R^3$, $|v|=1$, and $r>0$.
Denote the cylinder with {\it direction} $v$ and {\it radius} $r$ by 
\begin{equation}\label{eq:cylinder}
    \Cylinder(v,r):=\bigcup_{a\in\Rb}\sausage{av}{r},
\end{equation}
and the cone with {\it direction} $v$ and {\it radius} $r$ by 
\begin{equation}\label{eq:cone}
    \Cone(v,r):=\{x\in\mathbb R^3\setminus\{0\}:|x/|x|-v|\le r\}
\end{equation}
(note that such cones do not contain the origin). For $y\in\Rb^3$, let $\Cone_y(v,r):=y+\Cone(v,r)$ be the cone with vertex $y$.
\paragraph{Curves.}
Let $\mathcal P$ denote the set of all finite time continuous curves $\gamma:[0,t_\gamma]\to\mathbb R^3$, endowed with the following metric
$$d(\gamma,\gamma')=|t_\gamma-t_{\gamma'}|+\sup_{s\in [0,1]}\left|\gamma(st_\gamma)-\gamma'(st_{\gamma'})\right|.$$ 

\paragraph{Hitting time.}
For any set $A\subseteq\mathbb R^3$ and a continuous curve $\gamma$, let
\begin{equation}\label{eq:hitting}
    \tau_A=\tau_A(\gamma):=\inf\{t\ge0:\gamma(t)\in A\}\text{\ \ and\ \ }L_A=L_A(\gamma):=\sup\{t\ge0:\gamma(t)\in A\}
\end{equation}
denote the first-hitting and last-exit time of $A$ by $\gamma$, respectively. We will always use the convention that 
\begin{equation*}
\tau:=\tau_{\partial B(0,1)}.
\end{equation*}

Let $\mathscr{S}=\{S_i\}_{i\in I}$ be a countable collection of disjoint Borel sets in $\mathbb R^3$. Define the sequence of stopping times $\tau_{\mathscr{S}}^j$ $(j\in\mathbb N)$ associated with $\mathscr{S}$ and $\gamma$ recursively as follows
\begin{align}\notag
    &\tau_{\mathscr{S}}^0:=\tau_{\cup_{i\in I} S_i}(\gamma),\\ \label{eq:oht}
    &\tau_{\mathscr{S}}^j:=\inf\{t>\tau_{\mathscr{S}}^{j-1}:\gamma(t)\in \cup_{i\in I\setminus\{l\}} S_i\}\quad \text{ if } j\ge 1 \text{ and } \gamma(\tau_{\mathscr{S}}^{j-1})\in S_l.
\end{align}

\paragraph{Constants.}
Constants without numeric subscripts will change from time to time, while constants with numeric subscripts remain the same in the same section. Unless otherwise
specified, all constants are positive. Given two positive sequences $a_n$ and $b_n$, we write $a_n=O(b_n)$ if there exists a constant $C$ such that $a_n\le Cb_n$. If $a_n=O(b_n)$ and $b_n=O(a_n)$, we write $a_n=\Theta(b_n)$.

\subsection{Basic facts on Brownian motion}

We now review some basic facts about potential theory for later use, and refer the reader to  \cite{BM} for more details.
For $x\in\Rb^3$, denote by $\Pb_x$ the probability measure of a Brownian motion $W$ starting from $x$.
Let $p(t,z):=(2\pi t)^{-3/2}e^{-|z|^2/(2t)}$ be the heat kernel of a three-dimensional Brownian motion. 
For $t>0$, $x,y\in\mathbb R^3$ and $D$ a closed set, define $$r_D(t,x,y):=\Eb_x\big[p(t-\tau_D,W(\tau_D)-y);\tau_D<t\big],$$
and 
$$g_D(x,y):=\int_0^{\infty}\big( p(t,x-y)-r_D(t,x,y)\big)\,dt.$$
Note that $g_D(x,y)=G_{D^c}(x,y)$ if $x,y\notin D$, where $G_{F}(\cdot,\cdot)$ stands for the Green's function in a domain $F\subset\mathbb{R}^3$. 

Suppose $D^c$ is a domain. The harmonic measure in $D^c$ from $x\notin D$ is given by
\begin{equation}\label{eq:h}
h_D(x,V):=\Pb_x\big(\tau_D<\infty,W(\tau_D)\in V\big) \quad \text{ for all } V\subseteq \partial D.
\end{equation}
If $\partial D$ is smooth, then $h_D(x,\cdot)$ has a density with respect to the surface measure $\sigma_D$ on $\partial D$, which is called the Poisson kernel and will be denoted by $H_{D^c}(x,y)$. Let $U$ be a bounded domain with smooth boundary and $A$ be a closed set. Then, for all $x\in U\setminus A$ and $y\in \partial U\setminus A$, the derivative $(\partial/\partial n_y)G_{U\setminus A}(x,y)$ exists (see e.g.\  \cite[Theorem 3]{Da77}), where $n_y$ is the unit inward normal to $U$ at $y$, and it coincides with the Poisson kernel $H_{U\setminus A}(x,y)$. When $U=\Bc_b^\circ$ is an open ball, we write
\begin{equation}\label{eq:pb}
    p_{b}^A(x,y):=H_{\Bc_b^\circ\setminus A}(x,y).
\end{equation}
For $a<b$, we use $p_{a,b}^A$ to denote $p_b^A$ restricted to $(\partial\mathcal B_a\setminus A)\times(\partial\mathcal B_b\setminus A)$. By the strong Markov property of Brownian motion, for any $x\in \Bc_a^\circ\setminus A$, we have
\begin{equation}
    \label{eq:329}
    p_{b}^A(x,y)=\int p_a^{A}(x,z)\,p_{a,b}^A(z,y)\,\sigma_a(dz),
\end{equation}
where we recall that $\sigma_a$ denotes the surface measure on $\partial\Bc_a$. 

Finally, we recall the solution of Dirichlet problem; see \cite[Theorem 2.11]{BM}.
\begin{lemma}[Solution of Dirichlet problem]
    \label{lemma:harnack}
    Let $D$ be a bounded domain. Let $\phi$ be a bounded continuous function on $\partial D$. Let $f$ be a bounded continuous function on $D\cup\partial D$ that is harmonic in $D$ and converges to $\phi$ on regular points of $D^c\cap\partial D$. Then for all $x\in D$,
    \begin{equation*}
        f(x)=\int \phi(y)\, h_{D^c}(x,dy).
    \end{equation*}
\end{lemma}

\section{Brownian motion among cones and sausages}\label{sec:cone}

In this section, we will give various estimates on  Brownian motion moving among cones and Wiener sausages. In Section~\ref{subsec:ucone}, we show that with high probability, there is a cone (called an {\it uncovered cone} below) that is not intersected by $\Wc$ (the union of $k$ independent Brownian motions), using an argument in the framework of cover times. In Section~\ref{subsec:lbc}, we obtain a lower bound on the conditional non-intersection probability by forcing the other Brownian motion to stay in an uncovered cone. In Section~\ref{subsec:ubs}, we derive upper bounds for the probability that an independent Brownian motion stays close to $\Wc$. The estimates in this section together serve as building blocks for the proof of Theorems~\ref{thm:sep} and~\ref{thm:twisted}.

Throughout this section, we work under the setting of Theorem~\ref{thm:sep}, and refer to the paragraph above it for relevant notation. 

\subsection{Finding uncovered cones}\label{subsec:ucone}
In this subsection, we show that with high probability, there exists a cone that is not intersected by $\Wc$. The following result is reminiscent of large deviation results of cover times of Markov chain on 2D torus (see e.g.\ \cite{MR2123929,MR3126579}).
\begin{proposition}[Uncovered cone]
\label{prop:cone1}
For any $u_1>1$ and $0<u_2<1$, there exists 
$n_0(u_1,u_2,k)>0$ such that if $n>n_0$ and $|x_i|\geq 2^{-n}$ for all $1\le i\le k$, then 
\begin{equation}\label{eq:488}
\Pbx\left(\Wc\cap\left(\Cone(v,u_1^{-n})\cap \B(0,2)\right)=\varnothing \text{ for some unit vector $v$}\right)>1-u_2^n.
\end{equation}
\end{proposition}

Basically, our strategy for the proof of Proposition~\ref{prop:cone1} is to construct $n$ groups of disjoint cones of radius $u_1^{-n}$ such that the following properties hold.
\begin{enumerate}[label=(\roman*)]
    \item Each group contains $\Theta(n)$ disjoint cones. They are aligned in a way that the hitting probability of any sub-group $G$ by a Brownian motion started from some remote point before reaching distance $n^{-1/2}$ is bounded from above by $c\, (\log (n/|G|)+c')\, |G|/n$.
    \item The cones are dense enough such that the transition probability from any cone to other cones within distance $n^{-1/2}$ is uniformly bounded from below; 
    \item The cones are sparse enough such that a Brownian motion started from $B(0,2^{-n})^c$ will visit the cones at most $O(n^2)$ times before its last exit of $B(0,2)$ with high probability. 
\end{enumerate}
We refer to Lemma~\ref{Lemma:3.6} for (i) and (ii), and Lemma~\ref{Lemma:3.9} for (iii). Using these estimates, one can derive Proposition~\ref{prop:cone1} from a large deviation result on cover times under a general setting (see Lemma \ref{Lemma:3.8}).

The next three further subsections are devoted to the construction of cones, the transition between cones, and the conclusion of proof, respectively.

\subsubsection{Construction of cones}

We start with the construction of cones that will satisfy the previously mentioned  properties. Recall \eqref{eq:cone} for the definition of cones.
\begin{definition}[Cones]\label{def:cone}
    For any $u_1>1$, we choose a small constant $d_0>0$ and a large constant $n_1>64\, d_0^2$ such that the following construction is possible for all $n>n_1$ and $1\le m\le n$. 
    \begin{enumerate}[label=$\bullet$]
        \item On the unit sphere, we pick $n$ unit vectors $(v_i)_{1\le i\le n}$ such that $|v_i-v_{i'}|>12\,d_0\,n^{-1/2}$ for any $1\le i<i'\le n$. Let $\extendSquare_i:=\Cone(v_i,d_0/\sqrt n)$, $1\le i\le n$, be $n$ disjoint cones.
        \item On each surface $S_i\cap\partial \Bc_0$, we further pick $m$ points $v_{i,j}$ along an arc such that the following properties hold.
        Let $c_0$ and $c_1$ be two small positive constants. Consider two smaller cones
        \[
        \Tube_{i,j}:=\Cone(v_{i,j},u_1^{-n}) \quad \text{and} \quad \extendNearSubSet_{i,j}:=\Cone(v_{i,j},c_0m^{-1}n^{-1/2}).
        \]
        For any $i,j$, we have $\Tube_{i,j}\subseteq \extendNearSubSet_{i,j}\subseteq \extendSquare_i$, $\dist(v_{i,j},\partial S_i)\ge c_1\, n^{-1/2}$, and $V_{i,j}\cap V_{i,l}=\varnothing$ if $j\neq l$. Moreover, if $1\le j<l\le m$, then $ |z/|z|-v_{i,l}|\geq c_1|j-l|/(m\sqrt n)$  for every $z\in \extendNearSubSet_{i,j}$.
    \end{enumerate} 
\end{definition}
\begin{figure}[H]
    \centering
    \includegraphics[width=0.75\linewidth]{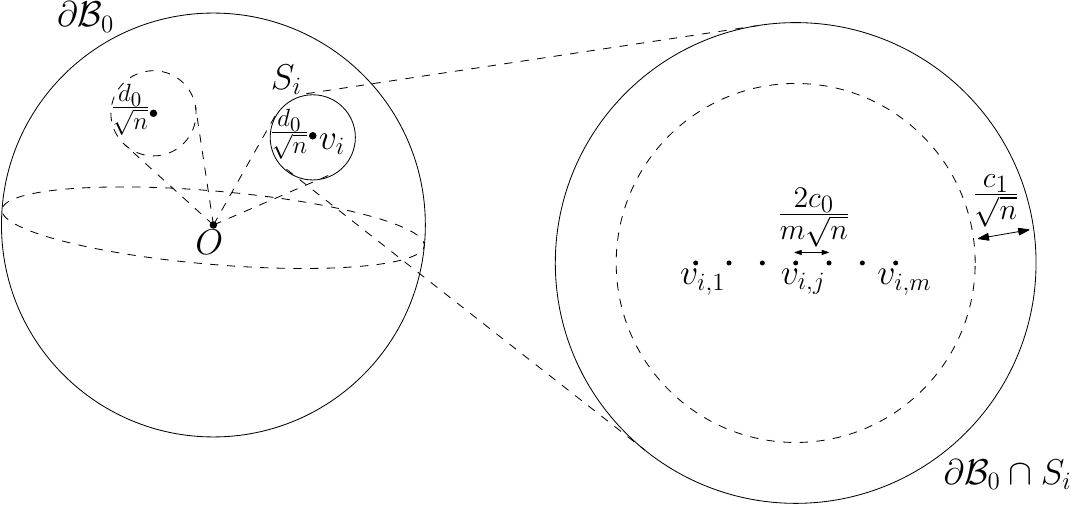}
    \caption{\textit{Left:} Each group of cones $(T_{i,j})_{1\le j\le m}$ is contained in some large cone $S_i$, and there are $n$ groups. \textit{Right:} Each group contains $m$ cones distributed uniformly in a line. We will set $m$ to be a constant fraction of $n$; see \eqref{eq:m}.}
    \label{fig:Cones Constructions}
\end{figure}
We refer to Figure~\ref{fig:Cones Constructions} for an illustration.
For later use, we set $r_0:=4d_0$ and note  $n_1>4r_0^2$. For $z\in \mathbb R^3$, denote 
\begin{equation}\label{eq:rz}
r(z) := r_0|z|n^{-1/2}<|z|/2 \quad \mathrm{and}\quad\Ball(z):=B(z,r(z)).
\end{equation}
We call $\Ball(z)$ the local ball around $z$.
Then, for all $z\in\extendSquare_i$ (see Figure~\ref{fig:Pic_Cones settings}),
\begin{equation}\label{6-1}
    \text{$\partial \B(0,|z|)\cap\extendSquare_i\subseteq \B(z,r(z)/2)$ and $\Ball(z)\cap\extendSquare_{i'}=\varnothing$ if $i'\neq i$,}
\end{equation}
and furthermore (see Figure~\ref{fig:Pic_Lemma 3.6}),
\begin{equation}\label{eq:440}
    \Cylinder(v_{i,j},u_1^{-n}|z|/2)\cap \Ball(z)\subseteq \Tube_{i,j}\cap \Ball(z)\subseteq\Cylinder(v_{i,j},2u_1^{-n}|z|)\cap \Ball(z),
\end{equation}
where $\Cylinder$ refers to cylinders; see \eqref{eq:cylinder}.
\begin{figure}[H]
    \centering
\includegraphics[width=0.55\linewidth]{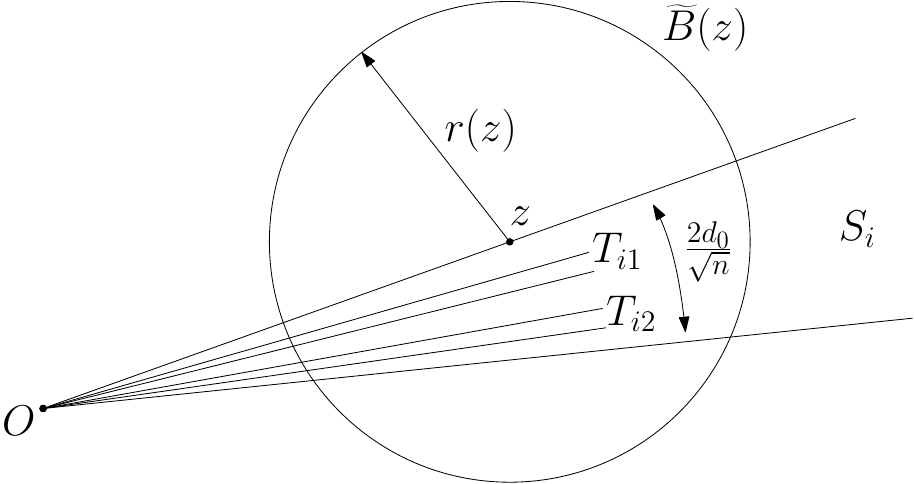}
    \caption{The local ball $\Ball(z)$.}
    \label{fig:Pic_Cones settings}
\end{figure}
\begin{figure}[H]
    \centering
\includegraphics[width=0.55\linewidth]{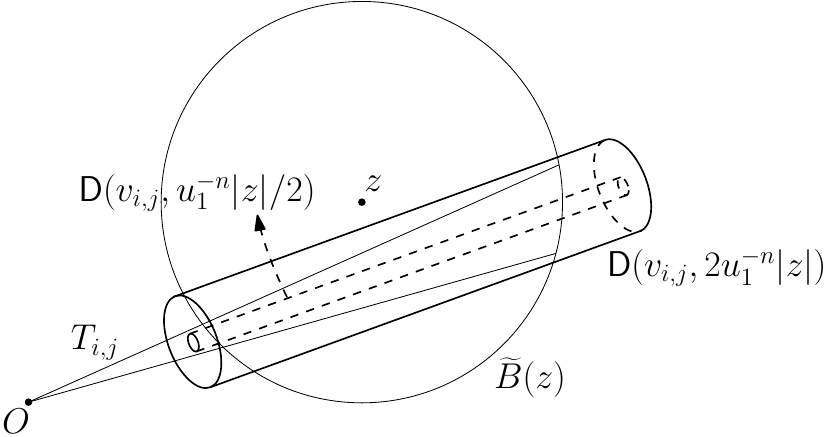}
    \caption{Bound cones by cylinders within $\Ball(z)$.}
    \label{fig:Pic_Lemma 3.6}
\end{figure}

\subsubsection{Transitions between the cones}
We first introduce some notation. Consider the union of tubes
$$
\extendSet_i=\bigcup_{j=1}^m\Tube_{i,j},\quad \Tube=\bigcup_{i=1}^n \extendSet_i, \quad \TubeComplement_{i,j}=\Tube\setminus \Tube_{i,j},
$$
and for $J\subseteq\{1,2,\ldots,m\}$, let
$$
\extendSubSet_{i,J}=\bigcup_{j\in J}\Tube_{i,j}, \quad 
\extendNearSubSet_{i,J}=\bigcup_{j\in J}\extendNearSubSet_{i,j}.
$$
We view $T_i$ as the $i$-th group and $T_{i,J}$ as a sub-group of $T_i$.
The following lemma provides upper and lower bounds on the hitting probability of cones. We leave its proof to Appendix~\ref{sec:hpe}. 

\begin{lemma}\label{Lemma:3.6}
\textnormal{(i) (Upper bound)} There exist positive constants $c_2(u_1)$ and $n_2(u_1)$ such that for all $n>n_2(u_1)$, $1\le m\le n$, $1\le i, i'\le n$, $1\le j'\le m$, and nonempty set $J\subseteq\{1,\ldots,m\}$, 
    \begin{equation}\label{eq:924}
        \Pb_{z}\Big(\tau_{\TubeComplement_{i',j'}} = \tau_ {\extendSubSet_{i,J}},\tau_{\TubeComplement_{i',j'}}<\tau_{\partial \Ball(z)}\Big)<c_2\Big(\log(m/|J|)+1\Big)|J|/n, \quad \text{for all } z\in\partial\extendNearSubSet_{i,J},
    \end{equation}
    and 
    \begin{equation}\label{eq:930}
        \Pb_{z}\Big(\tau_{\TubeComplement_{i, j}}=\tau_{\extendSubSet_{i,J}} ,\tau_{\TubeComplement_{i, j}}<\tau_{\partial \Ball(z)}\Big)<c_2\Big(\log(m/|J|)+1\Big)|J|/n,
        \quad \text{for all } z\in\Tube_{i,j}.
    \end{equation}
    \textnormal{(ii) (Lower bound)} There exist positive constants $c_3(u_1)\le \min\{1,1/(2c_2)\}$, $c_4(u_1)$ and $n_3(u_1)$ such that for all $n>n_3$, $m=[c_3n]$, $1\le i, i'\le n$, $1\le j'\le m$, and $z\in\partial \extendSquare_i$,
    \begin{equation}\label{eq:930-1}
    \Pb_{z}\Big(\tau_{\TubeComplement_{i', j'}}=\tau_{\extendSet_i} ,\tau_{\TubeComplement_{i', j'}}<\tau_{\partial \Ball(z)}\Big)>c_4.
    \end{equation}
\end{lemma}
We used the notation $[a]$ to denote the largest integer that is smaller than or equal to $a$.
From now on, we will fix
\begin{equation}\label{eq:m}
    m=[c_3n].
\end{equation}
As a direct consequence of Lemma~\ref{Lemma:3.6}, the following lemma shows that the cones are effectively ``sparse'' in $\Ball(z)$. More precisely, a Brownian motion started from $z\in\Tube_{i, j}$ exits $\Ball(z)$ before it hits other cones with probability greater than 1/2. 
\begin{lemma}\label{lem:731}
    For all $1\le i\le n$, $1\le j\le m$, $n>n_2$, and $z\in\Tube_{i,j}$, we have\begin{equation}\label{eq:722}
        \Pb_{z}\Big(\tau_{\TubeComplement_{i, j}}\ge\tau_{\partial \Ball(z)}\Big)\ge\frac{1}{2}.
    \end{equation}
\end{lemma}
\begin{proof}
    Note that $c_3\leq1/(2c_2)$. Using \eqref{eq:930} and letting $J=\{1,\ldots,m\}$, we have 
    \begin{equation*}
        \Pb_{z}\Big(\tau_{\TubeComplement_{i, j}} = \tau_{\extendSet_{i}},\tau_{\TubeComplement_{i, j}}<\tau_{\partial\Ball(z)}\Big)< c_2m/n\leq1/2.
    \end{equation*}
    Since $\Ball(z)$ does not intersect $\extendSet_{l}$ for any $l\neq i$, we conclude this lemma. 
\end{proof}

Next, we give an upper bound on the hitting probability of a sub-group (of cones) conditioned on hitting a given group (that contains this sub-group). 
\begin{lemma}\label{lem:transition}
    There exist constants $n_4(u_1)$ and $c_5(u_1)$ such that for all $1\le i, i'\le n$, $1\le j\le m$, $J\subseteq\{1,\ldots,m\}$, $z\notin \extendNearSubSet_{i,J}$, and $n>n_4$, we have
    \begin{equation}\label{eq:1062}
    \Pb_{z}\left(W(\tau_{\TubeComplement_{i', j'}})\in \extendSubSet_{i,J}\con W(\tau_{\TubeComplement_{i', j'}})\in \extendSet_i\right)<c_5\Big(\log(m/|J|)+1\Big)|J|/m.
    \end{equation}
    In particular, the above inequality holds for $z\in\Tube_{i',j'}$, and either $i\neq i'$ or $j'\notin J$.
\end{lemma}

\begin{proof}
    Choose $n_4>\max\{n_2,n_3\}$ such that all the results in Lemma \ref{Lemma:3.6} hold for $n>n_4$. We will first show \eqref{eq:1062} for $z\in S_i$.
    By (ii) of Lemma \ref{Lemma:3.6}, for every $z\in\extendSquare_i$, we have 
    \begin{equation*}
     \Pb_{z}\Big(W(\tau_{\TubeComplement_{i', j'}})\in \extendSet_i\Big)>c_4.
    \end{equation*}
    Then, it is sufficient to show that for all $z\notin \extendNearSubSet_{i,J}$ and large $n$,
    \begin{equation}\label{eq:3.7}
        \Pb_{z}\Big(W(\tau_{\TubeComplement_{i', j'}})\in \extendSubSet_{i,J}\Big)<c\Big(\log({m}/{|J|})+1\Big){|J|}/{m}.
    \end{equation}
    Define $$\overline p:=\sup_{z\in\partial \extendNearSubSet_{i,J}}\Pb_{z}\Big(W(\tau_{\TubeComplement_{i', j'}})\in \extendSubSet_{i,J}\Big),$$
    and let $z^*\in\partial \extendNearSubSet_{i,J}$ be such that $\Pb_{z^*}(W(\tau_{\TubeComplement_{i', j'}})\in \extendSubSet_{i,J})=\overline p$. We split $\overline p$ into three terms.
    \begin{align*}
    \overline p =\ & \Pb_{z^*}\Big(W(\tau_{\TubeComplement_{i', j'}})\in \extendSubSet_{i,J},\tau_{\TubeComplement_{i', j'}}<\tau_{\partial \Ball(z^*)}\Big) \\
&\quad +\Pb_{z^*}\Big(W(\tau_{\TubeComplement_{i', j'}})\in \extendSubSet_{i,J}, \tau_{\TubeComplement_{i', j'}}\geq\tau_{\partial \Ball(z^*)}, W(\tau_{\partial \Ball(z^*)})\notin \extendNearSubSet_{i,J}\Big) \\
&\quad \quad +\Pb_{z^*}\Big(W(\tau_{\TubeComplement_{i', j'}})\in \extendSubSet_{i,J}, \tau_{\TubeComplement_{i', j'}}\geq\tau_{\partial \Ball(z^*)}, W(\tau_{\partial \Ball(z^*)})\in \extendNearSubSet_{i,J}\Big)\\
:=\ &I_1 + I_2 + I_3.
    \end{align*}
    By (i) of Lemma~\ref{Lemma:3.6},
    \[
    I_1\le c_2\,\Big(\log({m}/{|J|})+1\Big){|J|}/{m}.
    \]
    By (ii) of Lemma~\ref{Lemma:3.6} and the fact $\extendSubSet_{i,J}\subset \extendNearSubSet_{i,J}$ (combined with the strong Markov property and the definition of $\overline p$),
    \[
    I_2\le (1-c_4)\, \overline p.
    \]
    Moreover, for some constant $c'(u_1)>0$ and sufficiently large $n$,
    \[
    I_3\le \frac{\sigma(\extendNearSubSet_{i,J}\cap\partial\Ball(z^*))}{\sigma(\partial\Ball(z^*))}\le c'|J|/m,
    \]
    where $\sigma$ denotes the surface measure. Combining the above estimates, we obtain that 
    \[
    \overline p \le c_2\,c_3\,\Big(\log({m}/{|J|})+1\Big){|J|}/{m} + (1-c_4)\, \overline p + c'|J|/m,
    \]
    which implies \eqref{eq:3.7} immediately.
   
    It remains to deal with the case $z\notin \extendSquare_i$. By the strong Markov property again, we can write $\Pb_{z}(W(\tau_{\TubeComplement_{i', j'}})\in \extendSubSet_{i,J}\mid  W(\tau_{\TubeComplement_{i', j'}})\in \extendSet_i)$ as weighted average of that with starting point on $\partial\extendSquare_i$. Hence,
    $$
    \Pb_{z}\left(W(\tau_{\TubeComplement_{i', j'}})\in \extendSubSet_{i,J}\,\con W(\tau_{\TubeComplement_{i', j'}})\in \extendSet_i\right)
    \leq \sup_{z'\in \partial \extendSquare_i}\Pb_{z'}\left(W(\tau_{\TubeComplement_{i', j'}})\in \extendSubSet_{i,J}\con W(\tau_{\TubeComplement_{i', j'}})\in \extendSet_i\right),
    $$
    concluding the proof.
\end{proof}

The following proposition shows that, with high probability, there is at least one cone that is not visited within the first $[Kn^2]$ transitions. It is a particular type of a more general result provided by Lemma~\ref{Lemma:3.8} in the appendix. We introduce some notation first. Let $\TubeGroup:=\cup_{1\le i\le n}\cup_{1\le j\le m}\{\Tube_{i,j}\}$ be the collection of cones. Recall that $\tau_{\TubeGroup}^{l}$ is the $l$-th hitting time of $\TubeGroup$ for any $l\ge 0$ (see \eqref{eq:oht}), and $\{W_s\}_{1\le s\le k}$ are $k$ Brownian motions under $\Pbx$. Define $X_l^s$ as the cone in $\TubeGroup$ where $W_s$ hits $\TubeGroup$ at the $l$-th hitting time, that is, $W_s(\tau_{\TubeGroup}^{l})\in X_l^s$.

\begin{proposition}\label{lem:955}
For any $K>0$, there exist constants $q=q(k,K,u_1)<1$ and $n_5(k,K,u_1)>0$, such that for all $n>n_5$, we have
\begin{equation}\label{eq:955}
    \Pbx\Big(\TubeGroup\subseteq\{X_l^s:l=0,\ldots,[Kn^2],s=1,\ldots,k\}\Big)<q^{c_3n^2}.
\end{equation}
\end{proposition}
\begin{proof}
    By Lemma~\ref{lem:transition}, we can take $F(x)=c_5x\log(1/x)+c_5x$, which satisfies the condition \eqref{eq:2668}. Hence,
    applying Lemma \ref{Lemma:3.8} with $G=\TubeGroup$, $G^{(i)}=\cup_{1\le j\le m}\{\Tube_{i,j}\}$ ($m=[c_3n]$) and the previous function $F$, we conclude the proof immediately. 
\end{proof}

With Proposition~\ref{lem:955},
to prove Proposition~\ref{prop:cone1}, 
it remains to show that with high probability a Brownian motion will not return to $\B(0,2)$ after $[Kn^2]$ transitions between cones in $\TubeGroup$. By Lemma~\ref{lem:731}, we can reduce it to a corresponding transition estimate between   spheres, as illustrated in the following lemma, which exploits the transience of a three-dimensional Brownian motion. 
Recall from \eqref{eq:hitting} that $L_{\B(0,2)}$ denotes the last-exit time of $\B(0,2)$.

\begin{lemma}\label{lem:883}
     For all $0<u\le 1$, there exists $K_1(u)>0$ such that the following holds. Consider the collection of spheres $\mathscr{S}:=\bigcup_{t\in\mathbb Z}\{\partial\mathcal B_{t/p}\}$ with a real parameter $p\geq1$. Then, for all $n\geq 1$ and $z\in\mathbb R^3$ with $|z|\ge2^{-n}$, we have
    \begin{equation*}
        \Pb_{z}\Big(\tau_{\mathscr{S}}^{[K_1np^2]}<L_{\B(0,2)}\Big)<u^n.
    \end{equation*}
\end{lemma}
\begin{proof}
    Let $Y_i$ be the unique integer that satisfies $W(\tau_{\mathscr{S}}^{i})\in\partial\mathcal B_{Y_i/p}$. Let $X_i:=Y_i-Y_{i-1}$ for all $i\geq 1$.
    By the strong Markov property, $X_i$'s are i.i.d.\ random variables with law 
    $$\Pb_{z}(X_i=-1)={e^{-1/p}}{(1+e^{-1/p})}^{-1}, \quad \Pb_{z}(X_i=1)={(1+e^{-1/p})}^{-1}.$$
    It follows that for any $M>0$,
    \begin{align}
        &\Pb_{z}\Big(X_1+\cdots+X_{[K_1np^2]}<2pM\Big)
        =\Pb_{z}\left(\exp\left(-{{(2p)}^{-1}(X_1+\cdots+X_{[K_1np^2]})}\right)>e^{-M}\right)\notag\\
        \leq\ &e^{M}\left(\mathbb E_z\left[\exp(-{(2p)}^{-1}{X_1})\right]\right)^{[K_1np^2]}
        \leq e^{M-(K_1n-1)c}, \label{eq:896}
    \end{align}
    where we used 
    \[
    \mathbb E_z\left[\exp(-{(2p)}^{-1}{X_1})\right]=\frac{2e^{-{1}/{(2p)}}}{1+e^{-{1}/{p}}},
    \] 
    and there is an absolute constant 
    $c>0$ such that for all $p\ge 1$,
    $$
    \left(\frac{2e^{-{1}/{(2p)}}}{1+e^{-{1}/{p}}}\right)^{p^2}\leq e^{-c}.
    $$
    Also, note that
    \begin{align*}
&\Pb_{z}\left(X_1+\cdots+X_{[K_1np^2]}\geq p(n\log(4/u)+10),\ \tau_{\mathscr{S}}^{[K_1np^2]}<L_{\B(0,2)}\right)\\
        =\ &\Pb_{z}\left(\big|W(\tau_{\mathscr{S}}^{[K_1np^2]})\big|\ge e^{n\log(4/u)+10}\big|W(\tau_{\mathscr{S}}^{0})\big|,\ W[\tau_{\mathscr{S}}^{[K_1np^2]},\infty)\cap\B(0,2)\neq\varnothing\right)\\
        \le\ &\frac{2}{e^{n\log(4/u)+10}|z|e^{-1/p}}\le (u/2)^n.
    \end{align*}
    Hence, combining with \eqref{eq:896}, it yields that
    \begin{align*}
        \Pb_{z}\Big(\tau_{\mathscr{S}}^{[K_1np^2]}<L_{\B(0,2)}\Big)
        \leq\ &\Pb_{z}\left(X_1+\cdots+X_{[K_1np^2]}<p(n\log(4/u)+10)\right)
        \\
        &+\Pb_{z}\left(X_1+\cdots+X_{[K_1np^2]}\geq p(n\log(4/u)+10),\tau_{\mathscr{S}}^{[K_1np^2]}<L_{\B(0,2)}\right)
        \\
        \leq\ & e^{n\log(4/u)/2+5-(K_1n-1)c}+(u/2)^n.
    \end{align*}
    We conclude the result by taking $K_1=((3/2)\log(4/u)+5+c)/c$.
\end{proof}
Next, we derive the transition estimate for cones as promised.
\begin{lemma}\label{Lemma:3.9}
    For all $0<u\le 1$, there exists $K_2(u_1,u)$ such that for all $n>n_2(u_1)$ and $z\in\mathbb R^3$ with $|z|\ge2^{-n}$, we have
    \begin{equation*}
        \Pb_{z}\Big(\tau_{\TubeGroup}^{[K_2n^2]}<L_{\B(0,2)}\Big)<u^n.
    \end{equation*}
\end{lemma}

\begin{proof}
    Assume $z'\in\Tube_{i,j}$ for some $i,j$. By Lemma~\ref{lem:731}, for $n>n_2$,
    $$\Pb_{z'}\Big(\tau_{\TubeComplement_{i, j}}\ge\tau_{\partial \Ball(z')}\Big)\geq1/2.$$
    Moreover, since the radius of $\Ball(z')$ is of order $n^{-1/2}$, we can choose $0<c'(u_1)<1$ small enough such that 
    $$\Pb_{z'}\Big(\tau_{\cup_{t\in I}\partial\mathcal B_{c't/\sqrt n}}<\tau_{\partial \Ball(z')}\Big)\geq3/4,$$ 
    where $I=\mathbb Z\setminus\{t'\}$ if $z\in \partial\mathcal B_{c't'/\sqrt n}$ for some integer $t'$ and $I=\mathbb Z$ otherwise.
    
    Now, consider the union of the spheres $\mathscr{S}=\bigcup_{t\in\mathbb Z}\{\partial\mathcal B_{c't/\sqrt n}\}$, that is, setting $p=\sqrt n/c'$ in Lemma~\ref{lem:883}. Combining the above two estimates, we see that for any $l'\geq 0$,
    $$
    \Pb_{z}\Big(\tau_{\TubeGroup}^{l'}\le\tau_{\mathscr{S}}^l<\tau_{\TubeGroup}^{l'+1} \text{ for some } l \con \mathcal F_{\tau_{\TubeGroup}^{l'}+}\Big)\geq1/2+3/4-1=1/4.
    $$
    As a result, by Hoeffding's inequality, for each $u,K_1>0$, we can pick $K_2(K_1,u,c')$ such that
    \begin{equation}\label{eq:954}
        \Pb_{z}\Big(\tau_{\TubeGroup}^{[K_2n^2]}<\tau_{\mathscr{S}}^{[K_1n^2/c'^2]}\Big)<(u/2)^{n^2}.
    \end{equation}   
    By Lemma \ref{lem:883}, we can pick $K_1=K_1(u/2)$ such that
$$\Pb_{z}\Big(\tau_{\mathscr{S}}^{[K_1n^2/c'^2]}<L_{\B(0,2)}\Big)<(u/2)^n.$$
    The above, combined with \eqref{eq:954}, shows that  
    \begin{align*}
        \Pb_{z}\Big(\tau_{\TubeGroup}^{[K_2n^2]}<L_{\B(0,2)}\Big)
        \leq\ &\Pb_{z}\Big(\tau_{\TubeGroup}^{[K_2n^2]}<\tau_{\mathscr{S}}^{[K_1n^2/c'^2]}\Big)+\Pb_{z}\Big(\tau_{\mathscr{S}}^{[K_1n^2/c'^2]}<L_{\B(0,2)}\Big)\\
        <\ &(u/2)^{n^2}+(u/2)^{n}\leq u^n.
    \end{align*}   
    This finishes the proof.
\end{proof}
\subsubsection{Proof of Proposition~\ref{prop:cone1}}
\begin{proof}[Proof of Proposition~\ref{prop:cone1}]
Combining Proposition~\ref{lem:955} and Lemma~\ref{Lemma:3.9}, we obtain that for any $u_1>1$ and any $0<u\le 1$, there exist $K_2(u_1,u)>0$, $0<q(k,K_2,u_1)<1$ and $n_0(k,K_2,u_1)>0$ such that if $n>n_0$ and $|x_i|\geq 2^{-n}$ for all $1\le i\le k$, then
\begin{align*}
    &\Pbx\left(\Wc\cap\left(\Cone(v,u_1^{-n})\cap \B(0,2)\right)\neq\varnothing \text{ for any unit vector $v$}\right)\\
    \le\ & \sum_{1\le i\le k}\Pb_{x_i}\Big(\tau_{\TubeGroup}^{[K_2n^2]}<L_{\B(0,2)}\Big) + \Pbx\Big(\TubeGroup\subseteq\{X_l^i:l=0,\ldots,[K_2n^2],i=1,\ldots,k\}\Big)\\
    \le\ & k\,u^n+q^{c_3n^2}.
\end{align*}
For any $0<u_2<1$, we can choose small $u$ and large $n_0$ such that $u^n+q^{c_3n^2}\le u_2^n$ for all $n>n_0$. This completes the proof. \end{proof}

\subsection{Lower bound on the conditional non-intersection probability}\label{subsec:lbc}
In this subsection, using Proposition~\ref{prop:cone1}, we show that for most samples of $\Wc$ (i.e., with $\Pbx$-probability at least $1-u_2^{n}$), another independent Brownian motion started from $x$ will not intersect $\Wc$ outside $B(x,2^{-n})$ with probability at least $C_1^{nu_1^{n}}$.

\begin{proposition}\label{prop:lowEstimate}
    For all $u_1>1$ and $0<u_2<1$, there exists $0<C_1(u_1)<1$ such that for all $n>n_0$, $x_i\in\Rb^3$ and $|x|<1$, 
    \begin{equation*}
        \Pbx\bigg(\unionW:\Pbf_{x}\Big(W[0,\tau]\cap (\unionW\setminus \B(x,2^{-n}))=\varnothing\Big)\ge C_1^{nu_1^{n}}\bigg)\ge 1-u_2^{n}.
    \end{equation*}
\end{proposition}
\begin{proof}[Proof of Proposition~\ref{prop:lowEstimate}]
    By running $W_i$'s until their first exit of $B(x,2^{-n})$ and considering the remaining parts, we can assume that $|x_i-x|\ge 2^{-n}$ for all $i$. 
    By Proposition~\ref{prop:cone1} and translation invariance, with probability at least $1-u_2^n$ under $\Pbx$, there is a cone $\Cc$ with radius $u_1^{-n}$ and vertex $x$ such that it does not intersect $\unionW$ in the unit ball. 
    By a standard cone estimate (see e.g.\ Lemma~2.4 of \cite{LV12}), there exists $0<C_1(u_1)<1$ such that
    $$
    \Pbf_x\Big(W[0,\tau] \subseteq \left(\mathcal C\cup\B(x,2^{-n})\right)\Big)\geq C_1^{nu_1^n},
    $$
    which implies
    $$
    \Pbf_x\left(W[0,\tau]\cap \left(\unionW\setminus \B(x,2^{-n})\right)=\varnothing\right)\geq C_1^{nu_1^n},
    $$
    and concludes the proof immediately.
\end{proof}

\begin{remark}
    It is interesting to compare Proposition~\ref{prop:lowEstimate} with an inequality in a different direction: for all $0<u<1$, there exist constants $0<C<1$ and $c>0$ such that if $|x_i|=2^{-n}$ for some $1\le i\le k$, then
    \[
    \Pb_{\vec x}\Big( \Wc: \sup_{|x|=2^{-n}}\Pbf_x(W[0,\tau]\cap \Wc=\varnothing)\le C^n \Big)\ge 1-cu^n.
    \]
    The above inequality was first established by Lawler; see e.g.\ \cite{Law96,RWcuttimes,LP2000}. Although its proof is simpler, which only requires a use of large deviation estimate, it is a powerful tool in the study of non-intersecting Brownian motions (or random walks).  
\end{remark}

\subsection{The sausage estimates}\label{subsec:ubs}
Recall \eqref{eq:sausage} for the definition of the sausage $\sausage{\cdot}{\cdot}$. In this subsection, we will prove two estimates regarding Brownian motion moving among Wiener sausages.

In the lemma below we show that it is unlikely that a Brownian motion will remain in a constant-radius sausage of $\Wc$ for a long time.
\begin{lemma}\label{lem:1}
There exist constants $0<C_2(k)<1$ and $R_0(k)>0$ such that for all $R>R_0$ and $x,x_i\in\Rb^3$,
\begin{equation}\label{eq:353}
    \Pb_{\vec{x}}\times\Pbf_{x}\Big(W[0,\tau_{\partial \B(x,R)}]\subseteq \sausage{\unionW}{2}\Big)\leq C_2^{\sqrt R}.
\end{equation}
\end{lemma}
\begin{proof}
    It suffices to show that there exist $0<C_2(k)<1$ and $R_0(k)>0$ such that for all $R>R_0$ and any continuous curve $\gamma\in \mathcal P$ such that $\gamma(0)=x$ and $\gamma(t_\gamma)\in\partial\B(x,R)$, we have
    \begin{equation*}
        \Pbx\Big(\gamma\subseteq \sausage{\unionW}{2}\Big)\leq C_2^{\sqrt R}.
    \end{equation*}
    Let $r=[{\sqrt R}/4]$. 
    For $j=0,1,\ldots,r$, let $y_j$ be the first hitting of $\partial \B(x,16j^2)$ by $\gamma$. 
    For each $y\in \B(y_j,2)$, let $W'$ be a Brownian motion started from $y$. Then we have
    \begin{align*}
        \Pb_y\Big(W'\cap \B(y_t,2)=\varnothing\text{ for every }t\neq j\Big)
        &\geq 1- \sum_{t\neq j}\Pb_y\Big(W'\cap \B(y_t,2)\neq\varnothing\Big)\\
        &\geq 1- 2\sum_{t\neq j}{(|16j^2-16t^2|-2)}^{-1}
        \geq 1/2.
    \end{align*}
    Let $N_i$ be the number of $j$ such that $W_i\cap \B(y_j,2)\neq\varnothing$. Using the above inequality, for $l\geq 1$, we have $\mathbb P_{x_i}(N_i\geq l)\leq 2^{-(l-1)}$. Hence, by a union bound,
    $$
    \Pbx\Big(\gamma[0,\tau_{\partial \B(x,R)}]\subseteq \sausage{\unionW}{2}\Big)\leq\Pbx\Big(N_1+\cdots+N_k\geq r\Big)\leq k\,2^{-(r/k-1)}\leq k\,2^{-(\sqrt R/(4k)-2)}.
    $$
    Taking $C_2=2^{-1/{(4k+1)}}$ and $R_0$ large, we conclude this lemma.
\end{proof}

The following lemma shows that, with high probability under $\Pbx$, the Brownian motion $W$ started away from $\unionW$ will get separated from $\unionW$, with some positive probability depending on the initial distance.

\begin{lemma}\label{lem:821}
    For any $\varepsilon>0$ and $n\geq1$, there exist constants $\delta_1(\varepsilon,n,\vec x)>0$ and $\delta_2(\varepsilon,n,\vec x)>0$, such that with probability at least $1-\varepsilon$ under $\Pbx$, we have
    $$\inf_{x\in\B(0,1)\setminus \sausage{\unionW}{2^{-n}}}\Pbf_{x}\Big(W[0,\tau]\cap\sausage{\unionW}{\delta_1}=\varnothing\Big)>\delta_2.$$
\end{lemma}
\begin{proof}
    We will freeze $\Wc$ and consider functions depending on $\Wc$ below. First, denote the compact set $\B(0,1)\setminus \sausage{\unionW}{2^{-n}}^\circ$ by $D$. Let $f(x):=\Pbf_{x}(W[0,\tau]\cap \unionW=\varnothing)$ be a continuous function on the compact set $D$. Define $\phi_2(\Wc):=\frac{1}{2}\inf_{x\in D}f(x)$, and note that $\Pbx(\phi_2>0)=1$ since $\B(0,1)\setminus\unionW$ is a.s.\  connected.   
    Consider the monotone sequence of continuous functions $f_m(x):=\Pbf_{x}(W[0,\tau]\cap\sausage{\unionW}{2^{-m}}=\varnothing)$ on $D$ with $m>n$, which converges pointwise to $f$, then the convergence is also uniform by Dini's theorem. Hence,  
    we can pick $\phi_1(\Wc)$ such that $\Pbx(\phi_1>0)=1$ and  
$$\inf_{x\in D} \Pbf_{x}\Big(W[0,\tau]\cap\sausage{\unionW}{\phi_1}=\varnothing\Big)>\phi_2.$$
For any $0<\eps<1$, there exist positive constants $\delta_1$ and $\delta_2$ such that $\Pbx(\phi_1>\delta_1,\phi_2>\delta_2)>1-\varepsilon$. This concludes the result.
\end{proof}

\section{Conditional separation lemma and BHP}\label{sec:sep_bhp}
In Section~\ref{subsec:csl1}, we complete the proof of Theorem \ref{thm:sep} (CSL). In Section~\ref{subsec:equi}, we obtain Theorem~\ref{theorem:bhp} (BHP), by showing that Theorems~\ref{thm:sep} and~\ref{theorem:bhp} are equivalent. 

\subsection{Conditional separation lemma}\label{subsec:csl1}

\begin{proof}[Proof of Theorem \ref{thm:sep}]We divide the proof into three steps.

\smallskip
\noindent \textbf{Step 1: Definition of the sequence $g_m$.}
Let $m_0$ be some large constant to be chosen later to satisfy \eqref{eq:m0choice0}, \eqref{eq:m0choice1} and \eqref{eq:m0choice2}. By Lemma~\ref{lem:821}, for $m_0$ and any $\eps>0$, there exist constants $\delta_1(\varepsilon,m_0,\vec x)>0$ and $\delta_2(\varepsilon,m_0,\vec x)>0$, such that with probability at least $1-\varepsilon$ under $\Pbx$, we have
\begin{equation}\label{eq:gm0}
    \inf_{x\in\B(0,1)\setminus \sausage{\unionW}{2^{-m_0}}}\Pbf_{x}\Big(W[0,\tau]\cap\sausage{\unionW}{\delta_1}=\varnothing\Big)>\delta_2.
\end{equation}

For any $n>m_0$, 
consider the decreasing sequence $r_n:=1/2$ and $r_{m}:=r_{m+1}+(3/4)^{m}+2^{-m-2}$ for all $m_0\le m\le n-1$. Choose $m_0$ large enough such that 
\begin{equation}\label{eq:m0choice0}
\sum_{m=m_0}^\infty ((3/4)^{m}+2^{-m-2})<1/4.
\end{equation}
Let $\Ac$ be any closed subset of $\Wc$. Consider the following random sequence with $m_0\le m\le n$, 
\[
g_m=g_m(\Acc,\delta_1):=\mathop{\inf}\limits_{x\in \B(0,r_m)\setminus \sausage{\unionW}{2^{-m}}}\Pbf_{x}\Big(W(\tau)\notin \sausage{\unionW}{\delta_1}\con W[0,\tau]\cap \Acc=\varnothing\Big).
\]
Note that $r_{m_0}\le 3/4$, and therefore, by \eqref{eq:gm0}, we have $g_{m_0}>\delta_2$.
In the following, we aim to show that $g_n\ge g_{m_0}/2$ for sufficiently large $m_0$.

\smallskip
\noindent \textbf{Step 2:  Bounding the error from discretization.}
As mentioned in Section~\ref{subsec:strategy}, we need to discretize the unit ball. More precisely, we approximate any point $x\in B(0,3/4)$ by $\overline x_m:=[2^{m+2}x+(1,1,1)/2]/2^{m+2}\in \mathbb Z^3/2^{m+2}$ such that $\dist(x,\overline x_m)\leq\sqrt{3}/2^{m+3}<2^{-m-2}$.
For brevity, we write $\tau_{x,m}$ for $\min\{\tau_{\partial\sausage{\unionW}{2^{-m+1}}} ,\tau_{\partial\B(x,(3/4)^m)}\}$. See Figure~\ref{fig:proof of prop 1.3} for an illustration.

    \begin{figure}[t]
        \centering
        \includegraphics[width=0.5\linewidth]{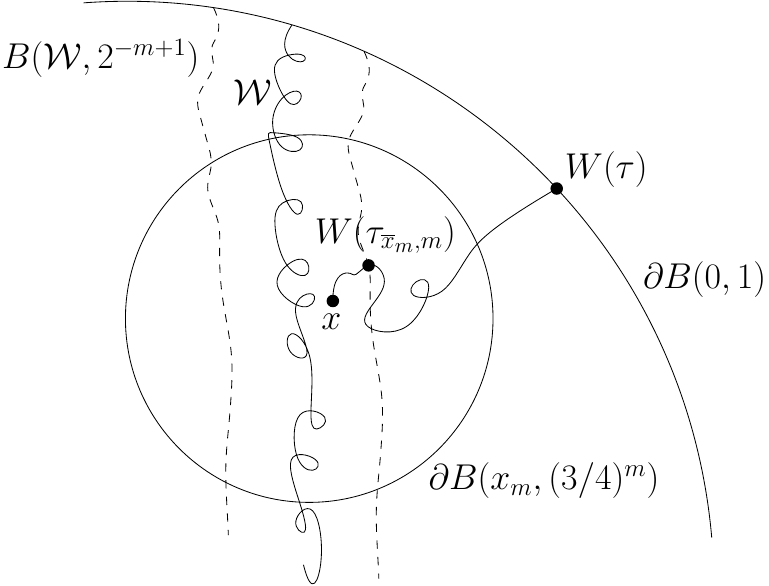}
        \caption{Illustration for the proof of Theorem \ref{thm:sep}.}
        \label{fig:proof of prop 1.3}
    \end{figure}

Now, applying Proposition~\ref{prop:lowEstimate} with $u_1=1.1$ and $u_2=1/9$, for any $m>n_0$, 
\begin{equation}\label{inequality:3}
    \Pbx\left(\Pbf_{\ol x_m}\left(W[0,\tau]\cap (\unionW\setminus \B(x,2^{-m-2}))=\varnothing\right)\geq C_1^{m1.1^m}\right)\geq 1-9^{-m}.
\end{equation}
Moreover, applying Lemma \ref{lem:1} with $R=1.5^m$, combined with Markov's inequality and the scaling invariance, we have for $m>\log_{1.5}R_0$,
\begin{equation}\label{inequality:2}
    \Pbx\left(\Pbf_{\ol x_m}\left(W(\tau_{\overline x_m,m})\notin \partial \sausage{\unionW}{2^{-m+1}}\right)\leq9^mC_2^{1.5^{m/2}}\right)\geq 1-9^{-m}.
\end{equation}
Note that there are at most $8^{m+3}$ points $\ol x_m$. Choose $m_0$ large enough such that 
\begin{equation}\label{eq:m0choice1}
m_0>\max\{n_0,\log_{1.5}R_0\}\quad \text{and} \quad
\sum_{m=m_0}^{\infty} {8^{m+3}}/{9^m}<\varepsilon/4.
\end{equation}
By a union bound, with probability at least $(1-\eps)$ under $\Pbx$, the inequalities regarding $\Pbf_{\ol x_m}$ in \eqref{inequality:3} and \eqref{inequality:2} hold simultaneously for all $\ol x_m$, assumed below. Hence, 
\begin{equation}\label{eq:discr}
    \frac{\Pbf_{\ol x_m}(W(\tau_{\ol x_m,n})\notin \partial \sausage{\unionW}{2^{-m+1}})}{\Pbf_{\ol x_m}(W[0,\tau]\cap (\unionW\setminus \B(\ol x_m,2^{-m-2}))=\varnothing)}\le \dfrac{9^mC_2^{1.5^{m/2}}}{C_1^{m1.1^m}}.
\end{equation}

Next, we will assume that $x\in B(0,3/4)$ and $2^{-m-1}\leq\dist(x,\unionW)\leq2^{-m}$, and derive an upper bound under $\Pbf_x$, similar to \eqref{eq:discr}, by Harnack's inequality. 
First, note that $\dist(\overline x_m,\unionW)> 2^{-m-2}$ since $\dist(x,\unionW)\geq 2^{-m-1}$ and $|x-\overline x_m|\leq 2^{-m-2}$. Hence, the denominator in the LHS of \eqref{eq:discr} is equal to $\Pbf_{\ol x_m}(W[0,\tau]\cap \unionW)=\varnothing)$. Moreover, since
$\dist(x,\partial\sausage{\unionW}{2^{-m+1}})\geq 2^{-m}$, by Harnack's inequality, there exists an absolute constant $C$ such that
\begin{equation}\label{eq:pf_csl_4}
\Pbf_{x}\Big(W(\tau_{\overline x_m,m})\notin \partial \sausage{\unionW}{2^{-m+1}}\Big)\leq C\,\Pbf_{\overline x_m}\Big(W(\tau_{\overline x_m,m})\notin \partial \sausage{\unionW}{2^{-m+1}}\Big),
\end{equation}
and
\begin{equation}\label{eq:pf_csl_5}
\Pbf_{x}\big(W[0,\tau]\cap\unionW=\varnothing\big)\geq C^{-1}\,\Pbf_{\overline x_m}\big(W[0,\tau]\cap\unionW=\varnothing\big).
\end{equation}
Plugging \eqref{eq:pf_csl_4} and \eqref{eq:pf_csl_5} into \eqref{eq:discr},
we conclude 
\begin{equation}\label{eq:discr1}
    \frac{\Pbf_{x}(W(\tau_{\ol x_m,m})\notin \partial \sausage{\unionW}{2^{-m+1}})}{\Pbf_{x}(W[0,\tau]\cap \unionW)=\varnothing)}\le \dfrac{C^2\,9^m\,C_2^{1.5^{m/2}}}{C_1^{m1.1^m}}.
\end{equation}

\smallskip
\noindent \textbf{Step 3: A recursive formula for $g_m$.} Now, we show a recursive formula for $g_m$ with error term bounded by  \eqref{eq:discr1}. Now, we assume $|x|\le r_{m+1}$ and $2^{-m-1}\leq\dist(x,\unionW)\leq2^{-m}$. Then, $|W(\tau_{\overline x_m,m})|\leq |x|+2^{-m-2}+(3/4)^m\le r_m$. Decomposing $W$ at the time $\tau_{\overline x_m,m}$ (see Figure~\ref{fig:proof of prop 1.3}), we obtain that
\begin{equation*}
    \begin{split}
        &\ \Pbf_{x}\Big(W(\tau)\notin \sausage{\unionW}{\delta_1}\con  W[0,\tau]\cap \Acc=\varnothing\Big)\\
        \geq\,&\ \Pbf_{x}\Big(W(\tau)\notin \sausage{\unionW}{\delta_1},W(\tau_{\overline x_m,m})\in \partial \sausage{\unionW}{2^{-m+1}}\con  W[0,\tau]\cap \Acc=\varnothing\Big)\\
        \geq\,&\ g_m\,\Pbf_{x}\Big(W(\tau_{\overline x_m,m})\in \partial \sausage{\unionW}{2^{-m+1}}\con  W[0,\tau]\cap \Acc=\varnothing\Big)\\
        \geq\,&\ g_m\,\Bigg(1-\frac{\Pbf_{x}\Big(W(\tau_{\overline x_m,m})\notin \partial \sausage{\unionW}{2^{-m+1}}\Big)}{\Pbf_{x}(W[0,\tau]\cap\unionW=\varnothing)}\Bigg).
    \end{split}
    \end{equation*}
    Using \eqref{eq:discr1} and taking the infimum over $x$ in the above inequality, we have
    \begin{equation}\label{eq:gm}
        g_{m+1}\ge \Bigg(1- \dfrac{C^2\,9^m\,C_2^{1.5^{m/2}}}{C_1^{m1.1^m}} \Bigg)\, g_m.
    \end{equation}  
We further choose large $m_0$ such that 
\begin{equation}\label{eq:m0choice2}
\sum_{m=m_0}^{\infty}C^2\,9^m\,C_2^{1.5^{m/2}}/C_1^{m1.1^m}<1/2.
\end{equation}
Using \eqref{eq:gm} repeatedly, we obtain the following inequality as promised:
$$
g_n\geq g_{m_0}\left(1-\sum_{m=m_0}^{n-1}\dfrac{C^2\,9^mC_2^{1.5^{m/2}}}{C_1^{m1.1^m}}\right)\geq g_{m_0}/2> \delta_2/2.
$$
Since $n$ can be taken arbitrarily large, we conclude that with probability at least $1-\varepsilon$ under $\Pbx$, for any $x\in \B(0,1/2)\setminus\unionW$ and any closed subset $\Acc\subset\unionW$, we have
\begin{equation}\label{eq:pf_csl_9}
\Pbf_{x}\Big(W(\tau)\notin\sausage{\unionW}{\delta_1} \con W[0,\tau]\cap \Acc=\varnothing\Big)>\delta_2/2.
\end{equation}
Indeed, \eqref{eq:pf_csl_9} also holds for $x\in(\unionW\setminus \Acc)\cap \B(0,1/2)$ since $\unionW$ is almost surely a nowhere-dense closed set. Since $\eps$ can be taken arbitrarily small, we finish the proof. 
\end{proof}

\subsection{Equivalence between CSL and BHP}\label{subsec:equi}

In this section, we show that Theorem \ref{thm:sep} and Theorem \ref{theorem:bhp} are equivalent, based on a general result  provided by the following Proposition \ref{pro:equivalence}. Since we have already proved Theorem \ref{thm:sep}, we obtain Theorem \ref{theorem:bhp} immediately.

\begin{proposition}\label{pro:equivalence}
    Let $\hA $ be a closed set in $\mathbb R^3$ such that $\hA ^\circ=\varnothing$ and $U\setminus\hA $ is connected for any domain $U\subset\mathbb{R}^{3}$. Then, the following two statements are equivalent:

    \textnormal{(i)} For any rational point $x\in\mathbb R^3$ and any rational $r>0$, there exist positive constants $\delta_1$ and $\delta_2$ that only depend on $\hA , x$ and $r$ such that for all $A\subseteq\hA $ and $y\in\B(x,r/2)\setminus A$,
    \begin{equation}\label{eq:2.7i}
        \Pbf_y\Big(W(\tau_{\partial\B(x,r)})\notin\sausage{\hA}{\delta_1} \con  W[0,\tau_{\partial\B(x,r)}]\cap A=\varnothing\Big)>\delta_2.
    \end{equation}

    \textnormal{(ii)} For each domain $U$ and compact set $K\subseteq U$, there exists a constant $C(U,K,\hA )$, such that for any closed subset $A$ of $\hA $, and any two bounded positive harmonic functions $u,v$ on $U\setminus A$ that vanish continuously on regular points of $A\cap U$, we have
    \[
    u(x_1)\,v(x_2)<C\,u(x_2)\,v(x_1),\quad \text{ for all } x_1,x_2\in K\setminus A.\]
\end{proposition}

\begin{remark}
    In fact, the above result can be generalized to more general sets $\hA$. We will not bother with this generalization to avoid unnecessary complexity.
\end{remark}

\begin{proof}[Proof of Proposition \ref{pro:equivalence}]
We start by showing 
\textnormal{(i)$\Rightarrow$(ii).}
    By restriction and approximation, it is sufficient to work under the assumption that $U$ is a bounded domain with smooth boundary, $u,v$ can be extended continuously to $\partial U\setminus A$ and they vanish continuously on all regular points of $A\cap\overline U$. 
    
    Write $D:=U\setminus A$. Recall that $H_D(x,y)$ is the Poisson kernel in $D$ (see the paragraph just above \eqref{eq:pb}). Using Lemma \ref{lemma:harnack}, we have
    \begin{equation*}
        u(y)=\int_{\partial U\setminus A}H_D(y,z)\,u(z)\,\sigma_U(dz) \quad \text{   and   } \quad
        v(y)=\int_{\partial U\setminus A}H_D(y,z)\,v(z)\,\sigma_U(dz).
    \end{equation*}
    Therefore, it suffices to show that there exists a constant $C(U,K,\hA )$ such that for all $A\subseteq\hA $, $y,y'\in K$ and $z,z'\in \partial U\setminus A$,
    \begin{equation}\label{eq:4.4}
        H_D(y,z)\,H_D(y',z')\le C\,H_D(y,z')\,H_D(y',z).
    \end{equation}
    
    Let $V_1,V_2,V_3,V_4$ be open sets such that \[K\subseteq V_1\subseteq\overline V_1\subseteq V_2\subseteq\overline V_2\subseteq V_3\subseteq\overline V_3\subseteq V_4\subseteq\overline V_4\subseteq U.\]
    Since $K$ and $\partial V_3$ are compact sets, we can pick finite balls with rational centers and rational radii $\B(y_i,r_i)\subseteq V_1$ $(i=1,\ldots,m)$ and $\B(z_j,s_j)\subseteq V_4\setminus\overline V_2$ $(j=1,\ldots,m)$ such that $\cup_i\B(y_i,r_i/2)\supseteq K$ and $\cup_j\B(z_j,s_j/2)\supseteq \partial V_3$. By (i), there exist positive constants $\delta_1,\delta_2$ such that \eqref{eq:2.7i} holds for all $(x,r)=(y_i,r_i)$ or $(z_j,s_j)$. 
    Define two constants 
    \[
    \delta_3:=\inf_{y\in\overline V_1\setminus\sausage{\hA}{\delta_1},z\in\overline V_4\setminus (V_2\cup\sausage{\hA}{\delta_1})} G_{U\setminus\hA }(y,z) \quad \text{and} \quad c:=\sup_{y\in\overline V_1,z\notin V_2} G_{U}(y,z),
    \]
    where $G$ denotes the Green's function.
    Recall the harmonic measure $h$ defined in \eqref{eq:h}.
    For $y\in \B(y_i,r_i/2)\setminus A$ and $z\in \B(z_j,s_j/2)\setminus A$, on the one hand,
        \begin{align*}
        G_D(y,z)
        =\ &\int_{\partial \B(y_i,r_i)\setminus A}\int_{\partial \B(z_j,s_j)\setminus A} G_D(y',z')\,h_{\partial \B(y_i,r_i)\cup A}(y,dy')\,h_{\partial \B(z_j,s_j)\cup A}(z,dz')\\
        \leq\ &\dfrac{c}{\delta_2^2}\int_{\partial \B(y_i,r_i)\setminus\sausage{\hA}{\delta_1}}\int_{\partial \B(z_j,s_j)\setminus\sausage{\hA}{\delta_1}}h_{\partial \B(y_i,r_i)\cup A}(y,dy')\,h_{\partial \B(z_j,s_j)\cup A}(z,dz'),
    \end{align*}
    where we used \eqref{eq:2.7i}, and on the other hand,
    \begin{align*}
        G_D(y,z)
        \geq \delta_3\, \int_{\partial \B(y_i,r_i)\setminus\sausage{\hA}{\delta_1}}\int_{\partial \B(z_j,s_j)\setminus\sausage{\hA}{\delta_1}} h_{\partial \B(y_i,r_i)\cup A}(y,dy')\,h_{\partial \B(z_j,s_j)\cup A}(z,dz').
    \end{align*}
    It follows that for all $y,y'\in K\setminus A$ and $z,z'\in\partial V_3\setminus A$,
    \begin{equation}\label{eq:4.3}
        G_{D}(y,z)\,G_D(y',z')
        \leq \frac{c^2}{\delta_3^2\delta_2^4}G_D(y,z')\,G_D(y',z).
    \end{equation}
We can also show that \eqref{eq:4.3} holds for $z,z'\in U\setminus (V_3\cup A)$ by using the following formula 
    \begin{equation*}
        G_D(y,z)
        =\int_{\partial V_3\setminus A}
        G_D(y,x)\,
        h_{\partial U\cup A\cup\partial V_3}(z,dx).
    \end{equation*}
    We conclude \eqref{eq:4.4} by taking appropriate limits in \eqref{eq:4.3} with $z,z'\to\partial U$. This finishes the proof of (i)$\Rightarrow$(ii).
    
    Next, we deal with the other direction (ii)$\Rightarrow$(i), which is easier. To use (ii), we set $U=\B(x,r)^\circ$, $K=\B(x,r/2)$ and $v(y)=\Pbf_{y}(W[0,\tau]\cap A=\varnothing)$. Fix some point $x_0\in K\setminus\hA $. Since $\hA ^c$ is connected, we have $v(x_0)>0$. 
    Thus, there exists $\delta_1>0$ such that 
    $$
    \Pbf_{x_0}\Big(W(\tau)\notin\sausage{\hA}{\delta_1},W[0,\tau]\cap\hA =\varnothing\Big)>0.
    $$
    Let $u(y)=\Pbf_{y}\Big(W(\tau)\notin\sausage{\hA}{\delta_1},W[0,\tau]\cap A=\varnothing\Big)$.
    It follows from (ii) that for all $y\in K\setminus A$,
    $$
    \Pbf_{y}\Big(W(\tau)\notin\sausage{\hA}{\delta_1} \con  W[0,\tau]\cap A=\varnothing\Big)=\frac{u(y)}{v(y)}> C^{-1}\,\frac{u(x_0)}{v(x_0)}.
    $$
    This concludes the result by letting $\delta_2=u(x_0)/(Cv(x_0))$. 
\end{proof}

Proposition \ref{pro:equivalence} now immediately leads to the following.

\begin{proof}[Proof of equivalence between Theorems \ref{thm:sep} and \ref{theorem:bhp}]
    Recall the notation for Theorem~\ref{thm:sep}. We apply Proposition \ref{pro:equivalence} with $\hA $ and $A$ therein as $\Wc$ and $\Ac$, respectively. Note that 
    almost surely under $\Pbx$, we have $\unionW^\circ=\varnothing$ and $\unionW$ is a closed set such that $U\setminus\unionW$ is connected for every domain $U$. Then, the equivalence follows immediately from Proposition \ref{pro:equivalence}. 
\end{proof}

\section{Analyticity}\label{sec:analyticity}
In this section, we will prove Theorem \ref{thm:analyticity}. Our proof will mainly follow the strategy in \cite{LSW02a}, with most of the ingredients replaced by their 3D versions. In Section~\ref{subsection:RK}, we introduce two notions that are useful in our couplings. In Section~\ref{subsec:prep}, we introduce various objects regarding paths which will be used frequently. 
In Section~\ref{subsec:pfm}, we complete the proof of Theorem \ref{thm:analyticity} based on some ingredients. Then, these ingredients are proved in Sections~\ref{subsec:pro:1}--\ref{subsec:pro:4}, separately.

\subsection{Extremal total variation distance and switching constant}\label{subsection:RK}
In this subsection, we introduce two special functionals, $R(p)$ the ``extremal total variance distance'' and $K(p)$ the ``switching constant'', associated with a function $p$ defined on a product space. For applications in this work, we take $p$ as a Poisson kernel. In this case, $K(p)$ is equal to the optimal comparison constant of BHP (see Lemma~\ref{lem:5.17}), and is useful for us to couple $h$-processes (see Proposition \ref{pro:1}). 
But for independent interest, we will state the results in a more general framework, and postpone most proofs to Appendix~\ref{appen:A}.

In the following, let $p$ be a non-negative Borel measurable function $p$ defined on some product space $E\times F$. For any finite measure $\mu$ on $F$, we reweight $\mu$ by $p(x_1,\cdot)$ and $p(x_2,\cdot)$ respectively, and define the maximal total variation distance between them among all possible $x_1,x_2\in E$ by
$$
\T{p}{\mu}:=\dfrac{1}{2}\sup_{x_1,x_2}\int \left|\dfrac{p(x_1,y)}{\int p(x_1,z)\mu(dz)}-\dfrac{p(x_2,y)}{\int p(x_2,z)\mu(dz)}\right|\mu(dy),
$$
where $x_1,x_2$ range over $x\in E$ such that $0<\int p(x,y)\mu(dy)<\infty$ (let $\T{p}{\mu}:=0$ if such $x$ does not exist). Then, the \emph{extremal total variance distance} of $p$ is defined by the supremum of $\T{p}{\mu}$ over $\mu$, i.e.:
\begin{equation}\label{eq:R}
    R(p):=\sup\{\T{p}{\mu}:\mu \mathrm{\ is\ a\ positive\ finite\ measure\ on\ } F\}.
\end{equation}
Moreover, we define the \emph{switching constant} of $p$ by 
\begin{equation}\label{eq:K}
    K(p):=\inf\{t\geq1:p(x_1,y_1)\,p(x_2,y_2)\leq t\, p(x_1,y_2)\,p(x_2,y_1),\forall x_1,x_2\in E,\forall y_1,y_2\in F\},
\end{equation}
which describes the ``lowest cost'' to switch the endpoints of $p$.

The following lemma reveals a simple identity between $K(p)$ and $R(p)$.
\begin{lemma}
\label{lemma:RK}
    For all Borel measurable functions $p:E\times F\to [0,\infty)$, we have
    \begin{equation*}
        R(p)=1-\dfrac{2}{1+\sqrt{K(p)}}.
    \end{equation*}
\end{lemma}
Suppose $p:E\times F\to [0,\infty)$ and $q:F\times G\to [0,\infty)$ are two functions, and $\mu$ is a measure on $F$, then we define the convolution of $p$ and $q$ (with respect to $\mu$) as 
$$
p*q(\mu)(x,z):=\int _F p(x,y)\,q(y,z)\,\mu(dy), \quad \text{ for all $x\in E$ and $z\in G$.}
$$
The following two lemmas show that convolution decreases the weighted total variance distance. 

\begin{lemma}\label{lemma:1090}
    If $\int p(x,y)\mu(dy)=1$ for all $x\in E$, and $\int q(y,z)\nu(dz)=1$ for all $y\in F$, then,
    $$
    \T{p*q(\mu)}{\nu}\leq \T{p}{\mu}\,\T{q}{\nu}.
    $$
\end{lemma}

\begin{lemma}\label{lemma:basic}
    Suppose $p:E\times F\to [0,\infty)$ and $q:F\times G\to [0,\infty)$ are two bounded functions, and $\mu$ is a measure on $F$. Then, we have $$R(p*q(\mu))\leq R(p)\,R(q).$$
\end{lemma}

The following lemma shows that averaging decreases the extremal total variation distance, and equivalently, the switching constant. 

\begin{lemma}\label{lemma:mean}
    Suppose that for each $x\in E$, $\mu(x,\cdot)$ is a measure on $F$. Let $q:F\times G\to\mathbb [0,\infty)$ be a measurable function. If 
    $$
    r(x,z)=\int_F q(y,z)\,\mu(x,dy),
    $$
    then we have $R(r)\leq R(q)$ and $K(r)\leq K(q)$.
\end{lemma}
The following lemma will be useful in the proof of Proposition \ref{pro:1}.

\begin{lemma}\label{lemma:tricky}
    Let $E,F\subseteq \mathbb R^3$ be non-empty Borel sets. Let $\mu$ be a finite measure on $F$, and $\nu(y,\cdot)$ be a finite measure on $E$ for each $y\in F$. 
    Suppose $p(x,y)$, $s(y)$ and $r(z)$ are non-negative bounded Borel functions on $E\times F,\,F,\,E$, respectively, such that the following equality holds for all $x\in E$,
    \begin{equation}\label{eq:482}
        r(x)=\int _F p(x,y)\left (\int _E r(z)\nu(y,dz)\right )\mu(dy)+\int _F p(x,y)s(y)\mu(dy).
    \end{equation}
    Moreover, suppose there exists a constant $0\le a <1$ such that for all $x\in E$,
    \begin{equation}\label{eq:486}
        \int_F p(x,y)\nu(y,E)\mu(dy)\leq a,
    \end{equation}
    then for all $x\in E$,
    \begin{equation*}
        \int p(x,y)s(y)\mu(dy)\geq \dfrac{1-a}{1+(K(p)-1)a}r(x).
    \end{equation*}
\end{lemma}

Finally, we show that the switching constant of the Poisson kernel is equal to the optimal comparison constant of the BHP. Recall the Poisson kernel $p_{a,b}^A$ defined in \eqref{eq:pb}.
\begin{lemma}\label{lem:5.17}
    Let $a<b$. Let $A$ be a closed set in $\mathbb R^3$ such that $\partial\mathcal B_a\not\subseteq A$ and $\partial\mathcal B_b\not\subseteq A $. Let $\mathcal U$ be the class of all bounded positive harmonic functions $u$ on $\mathcal B_b^\circ\setminus A$ such that  $u$ vanish continuously on the regular points of $ A\cap\mathcal B_b^\circ$ and can be extended continuously to $\partial\mathcal B_b\setminus A$. We have
    \begin{equation}\label{eq:bestconstant}
    \sup_{u,v\in\mathcal U,x_1,x_2\in\partial\mathcal B_a\setminus A}\dfrac{u(x_1)v(x_2)}{u(x_2)v(x_1)}
    =K(p_{a,b}^ A).    
    \end{equation}
\end{lemma}
\begin{proof}
    Both the LHS and RHS of \eqref{eq:bestconstant} are equal to 
    $$\sup_{u,v\in\mathcal U,x_1,x_2\in\partial\mathcal B_a\setminus A}\dfrac{\int u(y_1)p_{a,b}^ A(x_1,y_1)\sigma_b(y_1)\int v(y_2)p_{a,b}^ A(x_2,y_2)\sigma_b(y_2)}{\int u(y_1)p_{a,b}^ A(x_2,y_1)\sigma_b(y_1)\int v(y_2)p_{a,b}^ A(x_1,y_2)\sigma_b(y_2)},$$
thanks to the continuity of $p_{a,b}^A$.
\end{proof}

\subsection{Preparations}\label{subsec:prep}

For brevity, we will only deal with the case $\xi:=\xi_3(1,\lambda)$ throughout Section~\ref{sec:analyticity} as other cases follow from a similar argument. We introduce some notation first (somewhat different from those in Section~\ref{subsec:analyticity}). 
Let $\Gamma$ be the set of continuous path $\gamma:[0,t_\gamma]\to \mathcal B_0$ such that $0<|\gamma(t)|<1$ for $0\le t<t_\gamma$, $|\gamma(t_\gamma)|=1$, and $\gamma$ does not intersect the line segment $[0,(1,0,0)]$. Endow $\Gamma$ with the metric
\begin{equation*}
    d(\alpha, \gamma):=\inf_{\phi} \sup_{s\in[0,t_\alpha]} |\alpha(s)-\gamma(\phi(s))|,
\end{equation*}
where $\phi$ ranges over all increasing homeomorphisms $\phi:[0,t_\alpha]\to[0,t_\gamma]$, such that $\Gamma$ is a measurable space. For $\gamma\in\Gamma$ with $\gamma(0)\in \Bc_{-m}$, let $\gamma_m$ be the part of $\gamma$ after its first hitting of $\partial\Bc_{-m}$, that is, 
\begin{equation}\label{eq:shift}
\gamma_m(t):=\gamma(t+\tau_{\partial\mathcal{B}_{-m}})\quad \text{for all } 0\le t\le t_{\gamma}-\tau_{\partial\mathcal{B}_{-m}}.
\end{equation}
Let $\hg$ be a Brownian motion started from the endpoint of $\gamma$, and we write $\hg^n:=\hg[0,\tau_{\partial\mathcal B_n}(\hg)]$. Let $\tg^n$ be the concatenation of $\gamma$ and $\hg^n$, and write $\olg^n:=e^{-n}\tg^n$. Note that $\olg^n\in\Gamma$ almost surely.

For every $\gamma\in\Gamma$, consider the $h$-process $\hpro$ given by the Brownian motion started from $0$ and conditioned to exit $\mathcal B_0\setminus\gamma$ from $\partial\mathcal B_0$ (which occurs with positive probability since $\gamma\cap [0,(1,0,0)]=\varnothing$). 
Attach to the endpoint of $\hpro$ (on $\mathcal B_0$) an independent Brownian motion $\widehat \hpro$, and define the paths $\widehat \hpro^n$ and $\widetilde \hpro^n$ as above. Let 
\begin{equation*}
    Z_n=Z_n(\tg^n):=\mathbb P\big(\widetilde \hpro^n\cap\tg^n=\varnothing \mid \tg^n\big),
\end{equation*}
and $Z=Z_1$. For $n,\lambda>0$, define the linear operator $T_\lambda^n:\Jcc\to\Jcc$ by
\begin{equation}\label{eq:Tl}
    T_\lambda^nf(\gamma):=\mathbb E\big[f(\olg^n)Z_n(\tg^n)^\lambda\big],
\end{equation}
and let $T_\lambda=T_\lambda^1$. The expectation is over the randomness in $\tg^n$. Define 
\begin{equation}\label{eq:rq}
    \rQ_n(\gamma):=e^{n\xi}(T_\lambda^n1)(\gamma).
\end{equation}
For $\tg^n$-measurable events $\mathcal E$, define the weighted probability measure
\begin{equation}\label{eq:reweight}
    \widetilde{\mathbb P}^n_{\gamma}(\mathcal E):={\mathbb E[Z_n^\lambda\,\ind{\mathcal E}]}/{\mathbb E[Z_n^\lambda]}.
\end{equation}
For $0<k\leq n$, let $\delta(n,k,\gamma)$ denote a random variable with the same law that $\olg^k$ has under $\widetilde{\mathbb P}_\gamma^n$.

Recall the Poisson kernel $p_{a,b}^A$ defined in \eqref{eq:pb}, and the switching constant $K(p)$ defined in \eqref{eq:K}. 
Let $M>0$ be a constant to be chosen later in \eqref{eq:Mdef}.
We say that the $i$-th layer $\Bc_{i+1}^\circ\setminus (\Bc_i\cup\gamma)$ is \emph{good} if 
$K(p_{i+0.5, i+1}^{\mathcal B_{i}\cup \gamma})$, the switching constant of the Poisson kernel in the $i$-th layer, is smaller than $M$.
Let $m'\ge m''+1$ and $m'-m''\ge m\ge 1$. 
Define the set of \emph{nice} paths that have at least $m$ good layers by
\begin{equation}\label{eq:Yc}
    \mathcal Y_{m',m'',m}(M) := \Big\{\gamma\in \Gamma:\sum_{i=-m'}^{-m''-1}\mathbf{1}{\{K(p_{i+0.5, i+1}^{\mathcal B_{i}\cup \gamma})<M\}}\ge m\Big\}.
\end{equation}
We also define the set of pairs of nice paths that have at least $m$ good layers and coincide with each other after their first visits of $\partial\Bc_{-m'}$ for some $m'\ge m$ as follows
\begin{equation}\label{eq:Xm}
    \mathcal X_m(M) := \left\{(\gamma,\gamma')\in \Gamma^2:\text{ there exists } m'\ge m\ \ \mathrm{such\ that}\ \ \gamma_{m'}=\gamma_{m'}'\in \mathcal Y_{m',0,m}(M) \right\}.
\end{equation}
For $M,u>0$, define $$\|f\|_{u,M}:=\max\Big\{\|f\|,\sup\Big\{e^{mu}|f(\gamma)-f(\gamma')|: m\ge 1, (\gamma,\gamma')\in\mathcal X_m(M)\Big\}\Big\}.$$ For a linear operator $T:\Jcc\to\Jcc$, define
$$N_{u,M}(T):=\sup_{\|f\|_{u,M}=1}\{\|Tf\|_{u,M}\}.$$
Let $\Jcc_{u,M}$ be the space of all $f\in\Jcc$ such that $\|f\|_{u,M}<\infty$. Let $\mathcal L_{u,M}$ be the Banach space of linear operators from $\Jcc_{u,M}$ to $\Jcc_{u,M}$ with the operator norm $N_{u,M}$.

\subsection{Proof of Theorem~\ref{thm:analyticity}}\label{subsec:pfm}
In this subsection, we prove Theorem~\ref{thm:analyticity} by adapting the framework treating the $2$D case provided in \cite{LSW02a}. To this end, we need to derive a couple of corresponding results in three dimensions.
We would like to mention that some results are much more involved, and require different techniques, including the BHP (Theorem \ref{theorem:bhp}) and various properties of $K(p)$ and $R(p)$ discussed in Section~\ref{subsection:RK}. 
To illustrate the main differences, we will only give a detailed proof for
Propositions \ref{pro:1}, \ref{pro:2} and \ref{pro:4} respectively in the next three subsections, and omit the proof of some intermediate results to avoid repetition.   

Note that the definition \eqref{eq:Tl} can be extended to $T_z$ for complex $z$ with $\mathrm{Re}(z)>0$. As already observed in \cite[Proposition 2.1]{LSW02a}, the analyticity of exponents reduces to the following conclusion about the operator $T_z$. 

\begin{proposition}\label{pro:6}
    \textnormal{(i)} For all $\lambda>0$ and $M<\infty$, there exists $v=v(\lambda,M)>0$, such that the map $z\to T_{\lambda+z}$ is analytic from the disk $\{z:|z|<\lambda/2\}$ to $\mathcal L_{u,M}$ for all $0<u<v$.

    \textnormal{(ii)} For all $\lambda>0$, there exist an $M=M(\lambda)$ and a $v'=v'(\lambda)$, such that for all $0<u<v'$, $e^{-\xi_3(1,\lambda)}$ is an isolated simple eigenvalue of $T_{\lambda}$ in $\Jcc_{u,M}$.
\end{proposition}
Theorem \ref{thm:analyticity} is now a direct consequence of Proposition \ref{pro:6}.

\begin{proof}[Proof of Theorem \ref{thm:analyticity} assuming Proposition \ref{pro:6}]
    For each $\lambda$, we take $M=M(\lambda)$ and $u$ such that $0<u<\min\{v,v'\}$. So $z\to T_{\lambda+z}$ is analytic from $\{z:|z|<\lambda/2\}$ to $\mathcal L_{u,M}$, and $e^{-\xi_3(1,\lambda)}$ is an isolated simple eigenvalue of $T_\lambda$. Then, it follows that $\lambda\to\xi_3(1,\lambda)$ is analytic, in the same fashion as in \cite{LSW02a}. The proof for $\xi_3(k,\lambda)$ with $k\ge 2$ is similar.
\end{proof}

In the following, we first collect the results we need to show Proposition \ref{pro:6}, and then finish its proof at the end of this section.

We begin with an important coupling result, which can be regarded as a $3$D counterpart of \cite[Proposition 4.1]{LSW02a}.
It shows that for any $(\gamma,\gamma')\in\mathcal X_m(M)$, i.e., a pair of nice paths that coincide for at least $m$ good layers, we can couple the $h$-processes $\hpro, \hpro'$, associated with $\gamma,\gamma'$ respectively, such that they agree with each other outside $\Bc_{-m/2}$ with probability at least $1-O(e^{-cm})$.

\begin{proposition}\label{pro:1}
    For all $0<M<\infty$, there exist constants $v_1=v_1(M)>0$ and $a_1=a_1(M)$ such that for each $m\ge1$ and $(\gamma,\gamma')\in\mathcal X_m(M)$, we can couple $\hpro$ and $\hpro'$ in such a way that
    \begin{equation}\label{eq:1574}
        \mathbb P\big(\hpro\setminus\mathcal B_{-m/2}\neq \hpro'\setminus\mathcal B_{-m/2}\big)<a_1\,e^{-v_1m}.
    \end{equation}
\end{proposition}

We will prove Proposition \ref{pro:1} in Section~\ref{subsec:pro:1}. Based on this result, we obtain the following exponential decay. We omit the proof and refer to \cite[Proposition 4.3]{LSW02a} for details. 
\begin{proposition}\label{pro:3}
    For every $M<\infty$ and $\lambda>0$, there exist $v_2=v_2(M,\lambda)>0$ and $a_2=a_2(M,\lambda)>0$ such that for each $(\gamma,\gamma')\in\mathcal X_m(M)$, 
    \[
        \mathbb E\Big|Z_n(\tg^n)^\lambda
        -Z_n(\tgp^n)
        ^\lambda\Big|\leq a_2\,e^{-\xi n-v_2m}.
    \]
\end{proposition}

Using the BHP (Theorem \ref{theorem:bhp}), we can show that with high probability, the trace of a Brownian motion has a positive fraction of good layers, as illustrated below.

\begin{proposition}\label{pro:2}
    For each $w>0$, there exist $M_0=M_0(w)<\infty$ and $a_3=a_3(w)>0$ such that for $x\in\partial \mathcal B_0$ and $n'\geq n\geq 1$,
    \begin{equation*}
        \Pb_{x}\big(e^{-n'}W[0,\tau_{\partial\mathcal B_{n'}}]\notin\mathcal Y_{n',n'-n,n/2}(M_0)\big)\leq a_3\,e^{-wn}.
    \end{equation*}
\end{proposition}

We will prove Proposition \ref{pro:2} in Section~\ref{subsec:pro:2}, and use it as an input to show the following proposition in Section~\ref{subsec:pro:4}, which gives a coupling for the reweighted paths $\delta(n,k,\gamma)$ defined just below \eqref{eq:reweight}.

\begin{proposition}\label{pro:4}
    There exist positive constants $v_3(\lambda)$ and $a_4(\lambda)$ such that the following holds.
    For all $n\ge 1$ and $(\gamma,\gamma')\in\Gamma^2$,
    let $\delta:=\delta(n,n,\gamma)$ and 
    $\delta':=\delta(n,n,\gamma')$. 
    We can define $\delta$ and $\delta'$ in the same 
    probability space $(\Omega,\mathcal F,\mu)$  such that 
    $$\mu((\delta,\delta')\in\mathcal X_{n/16}(M_0))
    \geq 1-a_4\,e^{-v_3n},$$
    where $M_0=M_0(\max\{4,2\xi+1\})$ comes from Proposition \ref{pro:2}.
\end{proposition}

Combining all the above ingredients, we are able to conclude Proposition \ref{pro:6} now. Since the proof is quite similar to that of Proposition 2.1 in \cite{LSW02a}, we only give the key ideas below.

\begin{proof}[Sketch of proof of Proposition \ref{pro:6}]
    Since Proposition \ref{pro:6} (i) can be proved by following the same argument as in the proof of Proposition 2.1 (i) in \cite{LSW02a}, using Proposition \ref{pro:3} as an input, we do not repeat the proof here.

In the following, we focus on the proof of (ii), which consists of three steps, as in the proof of Proposition 2.1 (ii) in \cite{LSW02a}. 
We will state without proof the corresponding results in the first and third steps (as they are exactly the same), and explain the second step in detail,
where some caution needs to be exercised in dealing with good layers. 

Below, we take $M_0=M_0(\max\{4,2\xi+1\})$ as in Proposition \ref{pro:4}.
   Using Proposition \ref{pro:4}, one first shows that there exists a bounded linear functional $h:\Jcc_{u,M_0}\mapsto \Cb$ such that $|h(f)|\le \|f\|$, and for all $f\in\Jcc_{u,M_0}$, $\gamma\in\Gamma$ and $u\le 16v_3$,
    \begin{equation}\label{eq:fu}
    \big|T^nf(\gamma)-h(f)T^n1(\gamma)\big|\le c\,e^{-n(\xi+u/16)}\|f\|_{u,M_0}.
    \end{equation}

    The second step is to give an upper bound for the operator norm $N_{u,M_0}$ of the operator $f\mapsto T^n(f)-h(f)T^n1$. It boils down to show that for all sufficiently small $u$, there exists $c'=c'(u,\lambda)$, such that for $n\ge1$, $1\le m\le n$ and $(\gamma,\gamma')\in\mathcal X_m(M_0)$,
    \begin{equation}\label{eq:Num}
    \bigg|\Big(T^nf(\gamma)-h(f)T^n1(\gamma)\Big)-\Big(T^nf(\gamma')-h(f)T^n1(\gamma')\Big)\bigg|\le c'\,e^{-n\xi-nu/32-mu}\|f\|_{u,M_0}.
    \end{equation}
    Note that \eqref{eq:Num} obviously holds when $m\le n/32$ by \eqref{eq:fu}. 
    
    Hence, we only need to prove \eqref{eq:Num} for $m>n/32$.
    Since $\gamma$ and $\gamma'$ have the same endpoints, we can choose
    $\hg=\hgp$. By Proposition \ref{pro:3}, we have  
    \[
    \Big|T^nf(\gamma)-T^nf(\gamma')\Big|+h(f)\Big|T^n1(\gamma)-T^n1(\gamma')\Big|
    \le\mathbb E\Big[\Big|f(\olg^n)-f(\olgp^n)\Big|Z_n(\tg^n)^\lambda\Big]+2\|f\|a_2\,e^{-\xi n-v_2m}.
    \] 
    It remains to bound the first term on the right-hand side above.
    Our strategy, which takes into account the good layers, is a bit different from \cite{LSW02a}.  
    For $k\geq0$ and $0\le l\le n$, define the event that $\hg^n$ has deepest down-crossing to $e^{-k}$ and prior highest up-crossing to $e^l$ as below,
    \begin{equation*}
    \mathcal U_{k,l}:=\left\{
\begin{matrix}
\hg^n\cap \mathcal B_{-k}\neq\varnothing,\hg^n\cap \mathcal B_{-k-1}=\varnothing, \\
\hg^n[0,\tau_{\mathcal B_{-k}}]\cap\partial \mathcal B_l\neq\varnothing,\hg^n[0,\tau_{\mathcal B_{-k}}]\cap\partial \mathcal B_{l+1}=\varnothing
\end{matrix}\right\}.
\end{equation*}
    On the event $\mathcal U_{k,l}$, we define
    $$\tau_k:=\inf\{t:t\geq \tau_{\mathcal B_{-k}},\hg^n(t)\in \partial \mathcal B_0\}\mathrm{\ \ and\ \ }\eta_{n,k}:=e^{-n}\hg^n[\tau_k,\tau_{\partial \mathcal B_n}].$$
    Since there are at most $(k+l+2)$ good layers spoiled by the part of $\olg^n$ before $\eta_{n,k}$ on the event $\mathcal U_{k,l}$, we have $(\olg^n,\olgp^n)\in\mathcal X_{m+n/2-k-l-2}(M_0)$ if $\eta_{n,k}\in\mathcal Y_{n,0,n/2}(M_0)$ and $k+l+2<m+n/2$. 
    By the strong Markov property and \eqref{eq:c1}, we have 
    \begin{align*}
        &\mathbb E\Big[\Big|f(\olg^n)-f(\olgp^n)\Big|\,Z_n(\tg^n)^\lambda\,\ind{\mathcal U_{k,l}}\,\ind{\eta_{n,k}\in\mathcal Y_{n,0,n/2}(M_0)}\Big]\\
        \le\ &2e^{-(m+n/2-k-l-2)u}\,\|f\|_{u,M_0}\,\mathbb E[Z_n(\eta_{n,k})^\lambda\,\ind{\mathcal U_{k,l}}]
        \leq 2b_1\,e^{-\xi n}e^{-(k+l)}e^{-(m+n/2-k-l-2)u}\|f\|_{u,M_0}.
    \end{align*}
    Moreover, applying Proposition \ref{pro:2} with $w=2\xi+1$, we have
    \begin{align*}
        &\mathbb E\Big[\Big|f(\olg^n)-f(\olgp^n)\Big|\,\ind{\mathcal U_{k,l}}\,\ind{\eta_{n,k}\notin\mathcal Y_{n,0,n/2}(M_0)}\Big]\leq 2a_3\,e^{-(2\xi+1)n}e^{-(k+l)}e^{-(m-k-1)u}\|f\|_{u,M_0}.
    \end{align*}
    Summing the above two estimates over $k\ge0$ and $0\le l< n$ or $k=l=0$, we obtain that for all $m>n/32$, with $u<1/2$,
    \[
    \mathbb E\Big[\Big|f(\olg^n)-f(\olg^n)\Big|Z_n(\tg^n)^\lambda\Big]\le 
    c''e^{-n\xi-nu/2-mu}\|f\|_{u,M_0}.
    \] 
    Combining these estimates, we obtain \eqref{eq:Num} and conclude that 
    \begin{equation}\label{eq:on}
        N_{u,M_0}(T^n(\cdot)-h(\cdot)T^n1)\le c'e^{-n\xi-nu/32}.
    \end{equation}

    Using \eqref{eq:on}, one can finish the last step as in the proof of Proposition 2.1 (ii) of \cite{LSW02a}. We omit the details and conclude the proof of Proposition \ref{pro:6}.
\end{proof}

\subsection{Proof of Proposition \ref{pro:1}}
\label{subsec:pro:1}
Our main goal of this subsection is to prove Proposition \ref{pro:1}. Before diving into the proof, we give a simple inequality for the switching constant $K(p)$ when $p$ is a Poisson kernel, which is a direct consequence of Lemma~\ref{lemma:mean}, combined with a last-exit decomposition. 

\begin{lemma}    \label{lemma:compare}
    Let $a<b<c$. For any closed set $A$, we have 
    \begin{equation}\label{eq:rk}
        R(p_{c}^A|_{\mathcal B_b\setminus A\times \partial \mathcal B_c\setminus A})\leq R(p_{b,c}^{A\cup\mathcal B_a})\quad\text{ and }\quad K(p_{c}^A|_{\mathcal B_b\setminus A\times \partial \mathcal B_c\setminus A})\leq K(p_{b,c}^{A\cup\mathcal B_a}).
    \end{equation}
    In particular, for any $d\leq b$, 
    \begin{equation}
        R(p_{d,c}^A)\leq R(p_{b,c}^{A\cup\mathcal B_a})\quad\text{ and }\quad K(p_{d,c}^A)\leq K(p_{b,c}^{A\cup\mathcal B_a}).
    \end{equation}

\end{lemma}
\begin{proof}
Recall the definition of Poisson kernel $H$ defined just above \eqref{eq:pb}.
   For any $x\in \Bc_b\setminus A$, define the positive finite measure $\mu_x(dz)$ on $\partial\Bc_b^\circ\setminus A$ by
\begin{align*}
\mu_x(dz)&:=\int_{z_1\in\partial\Bc_a\setminus A} G_{\Bc_c^\circ\setminus A}(x,z_1)\, H_{\Bc_b^\circ\setminus(\Bc_a\cup A)}(z_1,z)\, \sigma_a(dz_1)\,\sigma_b(dz) \\[1mm]
&\qquad+ \mathbf{1}{\{x\in\Bc_b^\circ\setminus\Bc_a\}}\, H_{\Bc_b^\circ\setminus(\Bc_a\cup A)}(x,z)\,\sigma_b(dz)+ \mathbf{1}{\{x\in\partial\Bc_b\}}\,\delta_x(dz),
\end{align*}
where $H_{\Bc_b^\circ\setminus(\Bc_a\cup A)}(z_1,z)$ with $z_1\in \partial\Bc_a\setminus A$ denotes the ``boundary Poisson kernel'' (see e.g.\ Section 5.2 in \cite{MR2129588}), and $\delta_x(dz)$ denotes the Dirac measure. Then, for all $x\in \Bc_b\setminus A$ and $y\in\partial\mathcal B_c\setminus A$, we can represent $p_c^A(x,y)$ as 
\begin{equation}\label{eq:represent}
        p_c^A(x,y)=\int p_{b,c}^{A\cup\partial\mathcal B_a}(z,y)\,\mu_x(dz).
\end{equation}
Applying Lemma \ref{lemma:mean}, we obtain \eqref{eq:rk} immediately.
\end{proof}

We will use the following simple inequality.
\begin{lemma}\label{lem:si}
    Let $f$ and $g$ be two positive measurable functions with respect to some measure $\nu$. Suppose $f'\ge f$ and $g'\ge g$ such that $\nu(f')=\nu(g')=1$. Then, 
    \[
    1-\nu(\min\{f,g\})\le (1-\nu(f))+(1-\nu(g))+\frac12\,\nu(|f'-g'|).
    \]
\end{lemma}
\begin{proof}
    It follows from the following two facts
    \[
    \nu(\min\{f,g\})+\nu(\max\{f,g\})=\nu(f)+\nu(g), \mbox{ and}
    \] 
    \[
    \nu(\max\{f,g\})\le \nu(\max\{f',g'\})=1+\frac12\,\nu(|f'-g'|),
    \]
    where in the second inequality we used $\max\{f',g'\}=(f'+g'+|f'-g'|)/2$ and $\nu(f')=\nu(g')=1$.
\end{proof}

We are now ready to prove Proposition \ref{pro:1}. 

\begin{proof}[Proof of Proposition \ref{pro:1}]We divide the proof into four steps.

\smallskip
\noindent \textbf{Step 1: Couple $\hpro$ and $\hpro'$.} 
    By the definition of $\mathcal X_m(M)$ (see \eqref{eq:Xm}), we can pick $m'\geq m$ such that $\gamma_{m'}=\gamma_{m'}'\in \mathcal Y_{m',0,m}(M)$.
    
    For $|y|<1$ and a closed set $A$, define $$u_A(y):=\mathbb P_y\big(W(0,\tau]\cap A=\varnothing\big).$$
    Define the finite measure
    $$
    \mu:=\int_{\partial \mathcal B_{-m/2}}\mathbb P_y^{\mathcal B_0}\, \sigma_{-m/2}(dy),
    $$
    where $\Pb_y^{\mathcal B_0}$ denotes the probability measure of a Brownian motion started from $y$ and stopped when it hits $\partial\Bc_0$. Recall from \eqref{eq:shift} that $X_{m/2}$ is the part of $X$ after its first visit of $\partial\Bc_{-m/2}$.
    Then, the laws of $X_{m/2}$ and $X'_{m/2}$
    have the following Radon-Nikodym derivatives with respect to $\mu$, respectively:
    $$
    q(\alpha) := \frac{p^\gamma_{-m/2}(0,\alpha(0))}{u_\gamma(0)}\ind{\alpha\cap\gamma=\varnothing}\quad \mbox{ 
    and } \quad
    q'(\alpha) := \frac{p^{\gamma'}_{-m/2}(0,\alpha(0))}{u_{\gamma'}(0)}\ind{\alpha\cap\gamma'=\varnothing}.
    $$
    
    Note that $\gamma\cup\mathcal B_{-m'}=\gamma'\cup\mathcal B_{-m'}$.
    By the maximal coupling (see e.g.\ (4.3) in \cite{LSW02a}), we can couple $\hpro$ and $\hpro'$ such that (note that one can normalize $\mu$ in the beginning, but the normalizing constant will cancel with the Radon-Nikodym derivatives)
    \begin{align*}
        &\mathbb P\big(\hpro\setminus\mathcal B_{-m/2}=\hpro'\setminus\mathcal B_{-m/2}\big)\ge \int\min\{q,q'\}\,d\mu\\
        \ge&\int\min\left\{\frac{p^\gamma_{-m/2}(0,\alpha(0))}{u_\gamma(0)},\frac{p^{\gamma'}_{-m/2}(0,\alpha(0))}{u_\gamma'(0)}\right\}\ind{\alpha\cap(\gamma\cup \mathcal B_{-m'})=\varnothing}\,\mu(d\alpha)\\
        =& \int_{\partial \mathcal B_{-m/2}}\int \min\left\{\frac{p^\gamma_{-m/2}(0,\alpha(0))}{u_\gamma(0)},\frac{p^{\gamma'}_{-m/2}(0,\alpha(0))}{u_\gamma'(0)}\right\} \ind{\alpha\cap(\gamma\cup \mathcal B_{-m'})=\varnothing}\, \mathbb P_y^{\mathcal B_0}(d\alpha) \,\sigma_{-m/2}(dy)\\
        =& \int_{\partial \mathcal B_{-m/2}} \min\left\{\frac{p_{-m/2}^{\gamma}(0, y)}{u_\gamma(0)},\dfrac{p_{-m/2}^{\gamma'}(0, y)}{u_{\gamma'}(0)}\right\}\,u_{\gamma\cup\mathcal B_{-m'}}(y)\, \sigma_{-m/2}(dy).
    \end{align*}  
    We write $\nu(dy):=u_{\gamma\cup\mathcal B_{-m'}}(y)\,\sigma_{-m/2}(dy)$, and note that $u_{\gamma}(0)\ge\int p_
    {-m/2}^\gamma(0,y)\,\nu(dy)$ and a similar result holds for $\gamma'$. Applying Lemma~\ref{lem:si} with 
    \[
    f=\frac{p_{-m/2}^{\gamma}(0, y)}{u_\gamma(0)}\le f'=\dfrac{p_{-m/2}^{\gamma}(0, y)}{\int p_{-m/2}^{\gamma}(0, z)\nu(dz)}\quad \text{and} \quad g=\frac{p_{-m/2}^{\gamma'}(0, y)}{u_\gamma'(0)}\le
    g'=\dfrac{p_{-m/2}^{\gamma'}(0, y)}{\int p_{-m/2}^{\gamma'}(0, z)\nu(dz)},
    \]
    we obtain that \begin{equation}\label{eq:T123}
    \mathbb P\big(\hpro\setminus\mathcal B_{-m/2}\neq \hpro'\setminus\mathcal B_{-m/2}\big)
        \leq 1-\nu(\min\{f,g\})\le (1-\nu(f))+(1-\nu(g))+\frac12\,\nu(|f'-g'|).
    \end{equation}
    
\smallskip
\noindent \textbf{Step 2: Upper bound $\frac12\,\nu(|f'-g'|)$.} 
    Denote $p:=p_{-m'+1/2,-m/2}^{\gamma\cup\mathcal B_{-m'}}$ for brevity.  
    According to \eqref{eq:represent},
    there are finite measures $\mu_1$ and $\mu_2$ on $\partial \mathcal{B}_{-m'+1/2}$ such that 
    \[
    p_{-m/2}^{\gamma}(0, y)=\int p(x,y)\mu_1(dx) \quad \text{and} \quad p_{-m/2}^{\gamma'}(0,y)=\int p(x,y)\mu_2(dx).
    \]
    Moreover, we define the function $q$ on $\{1,2\}\times\partial\mathcal B_{-m/2}\setminus\gamma$ such that $q(1,\cdot):=p_{-m/2}^\gamma(0,\cdot)$ and $q(2,\cdot):=p_{-m/2}^{\gamma'}(0,\cdot)$. By the definition of extremal total variation distance $R(\cdot)$ (see \eqref{eq:R}) and Lemma \ref{lemma:mean}, we have
    \begin{equation}\label{eq:Rp}
    \frac12\,\nu(|f'-g'|)\leq R(q)\leq R(p).
    \end{equation}
    Moreover, we have
    \begin{equation}\label{eq:Rp1}
        R(p)
        \leq \prod_{i=-m'}^{[-m/2]-1}R(p_{i+1/2,i+1}^{\gamma\cup \mathcal B_{-m'}})
        \leq\prod_{i=-m'}^{[-m/2]-1}R(p_{i+1/2,i+1}^{\gamma\cup \mathcal B_{i}})
        \leq (1-\frac{2}{\sqrt M +1})^{[m/2]},
    \end{equation}
    where we used \eqref{eq:329} and Lemma \ref{lemma:basic} in the first inequality, used Lemma \ref{lemma:compare} in the second one, and used Lemma \ref{lemma:RK} and the fact that there are at least $[m/2]$ good layers from $-m'$ to $[-m/2]-1$ in the last one.

\smallskip
\noindent \textbf{Step 3: Upper Bound $(1-\nu(f))+(1-\nu(g))$.} 
    We aim to prove
    \begin{equation}\label{eq:T1}
        1-\nu(f)\leq \dfrac{K(p)e^{-m/2}}{1+(K(p)-1)e^{-m/2}}.
    \end{equation}
    A similar bound also holds for $(1-\nu(g))$.
    Recall the harmonic measure $h$ defined in \eqref{eq:h}.
    For $x\in\partial\mathcal B_{-m'}\setminus\gamma$, we have
    \begin{align*}
        u_\gamma(x)=\ &\int p_{-m',-m/2}^{\gamma}(x,y)\left(\int_{\partial\mathcal B_{-m'}\setminus\gamma} h_{\gamma\cup\partial\mathcal B_0\cup\partial\mathcal B_{-m'}}(y,dz)u_\gamma(z)\right)\sigma_{-m/2}(dy)\\
        \ &+\int p_{-m',-m/2}^{\gamma}(x,y)\,u_{\gamma\cup\mathcal B_{-m'}}(y)\,\sigma_{-m/2}(dy).
    \end{align*}
    By Lemma \ref{lemma:tricky} and the fact that $$\int p_{-m',-m/2}^{\gamma}(x,y)\, h_{\gamma\cup\partial\mathcal B_0\cup\partial\mathcal B_{-m'}}(y,\partial\mathcal B_{-m'}\setminus\gamma)\,\sigma_{-m/2}(dy)\leq e^{-m/2},$$ 
    we obtain that for all $x\in\partial\mathcal B_{-m'}\setminus \gamma$,
    \begin{equation*}
        \int p_{-m',-m/2}^{\gamma}(x,y)\,u_{\gamma\cup\mathcal B_{-m'}}(y)\,\sigma_{-m/2}(dy)\geq \frac{1-e^{-m/2}}{1+(K(p_{-m',-m/2}^{\gamma})-1)\,e^{-m/2}}\,u_\gamma(x).
    \end{equation*}
    Thus, 
    \begin{align*}
        1-\nu(f)=\ & 1-\frac{\int p_{-m/2}^{\gamma}(0, y)\,u_{\gamma\cup\mathcal B_{-m'}}(y)\,\sigma_{-m/2}(dy)}{u_\gamma(0)}\\
        =\ & 1-\frac{\iint p_{-m'}^\gamma(0,x)\,p_{-m',-m/2}^{\gamma}(x,y)\,u_{\gamma\cup\mathcal B_{-m'}}(y)\,\sigma_{-m'}(dx)\,\sigma_{-m/2}(dy)}{\int p_{-m'}^\gamma(0,x)\,u_\gamma(x)\,\sigma_{-m'}(dx)}\\
        \leq\ & \frac{K(p_{-m',-m/2}^{\gamma})\,e^{-m/2}}{1+(K(p_{-m',-m/2}^{\gamma})-1)\,e^{-m/2}}
        \leq\,\frac{K(p)\,e^{-m/2}}{1+(K(p)-1)\,e^{-m/2}},
    \end{align*}
    where the last inequality is obtained by Lemma \ref{lemma:compare}.

\smallskip
\noindent    \textbf{Step 4: Conclusion.} Combining \eqref{eq:T123}, \eqref{eq:Rp} and \eqref{eq:T1}, we have
    \begin{equation*}
        \mathbb P\big(\hpro\setminus\mathcal B_{-m/2}\neq \hpro'\setminus\mathcal B_{-m/2}\big)
        \leq R(p)+\frac{2K(p)e^{-m/2}}{1+(K(p)-1)e^{-m/2}},
    \end{equation*}
    which combined with Lemma \ref{lemma:RK} and \eqref{eq:Rp1} completes the proof.
\end{proof}

\subsection{Proof of Proposition \ref{pro:2}}
\label{subsec:pro:2}

The main purpose of this subsection is to prove Proposition \ref{pro:2}.

We first give a version of Theorem \ref{theorem:bhp}, which is tailored for our use.
Let $D$ be a bounded domain with smooth boundary. Suppose $x\in D$ and $y\in\partial D$. We denote by $\Pb^D_{x\to y}$ the law of $\eta[0,\tau_{\partial D}]$, where $\eta$ is a Brownian motion started from $x$ conditioned on $\eta(\tau_{\partial D})=y$. 

\begin{lemma}\label{lem:1530}
    Consider the triples $(x_i,y_i,D_i)$ with $x_i\in D_i$ and $y_i\in \partial D_i$ that belong to one of the following three cases
    \begin{enumerate}[label=(\arabic*)]
        \item $x_i\in\partial\Bc_{-1}$ and $D_i=\mathcal B_l^\circ\setminus \mathcal B_{-4}$ with $l=2,3,4$.
        \item $x_i\in\partial\Bc_{-1}$ and $D_i=\mathcal B_2^\circ$.
        \item $x_i\in\partial\Bc_{2}$ and $D_i=\mathcal B_l^\circ\setminus \mathcal B_{-1}$ with $l=5,6,7$.
    \end{enumerate}
   Let $\eta_i$ be sampled according to $\Pb_{x_i\to y_i}^{D_i}$ and $\ol\eta:=\cup_{i=1}^k \eta_i$. Then, for each $\varepsilon>0$, there exists $0<C_1=C_1(\varepsilon,k)<\infty$ such that
    \begin{equation}
        \prod_{i=1}^k\mathbb P_{x_i\to y_i}^{D_i}\Big(K(p_{1/2,1}^{\ol\eta\cup\mathcal B_0})<C_1\Big)>1-\varepsilon.
    \end{equation}
\end{lemma}
\begin{proof}
    Let $V=\overline{\mathcal B_1\setminus\mathcal B_0}$. 
    Let $x_i^0$ be $(e^{-1},0,0)$ if $x_i\in\partial\Bc_{-1}$ and $(e^2,0,0)$ otherwise.
    Then, the law of $\eta_i[\tau_V,L_V]$ is absolutely continuous with respect to that of $W_i[\tau_V,L_V]$ (with uniformly bounded density), where $W_i$ is a Brownian motion started from $x_i^0$.
    We conclude the result immediately by using Theorem \ref{theorem:bhp} and Lemma \ref{lem:5.17}.
\end{proof}
 
Next, we prove Proposition \ref{pro:2} based on Lemma \ref{lem:1530}.

\begin{proof}[Proof of Proposition \ref{pro:2}\\]
    Recall that $W$ is the Brownian motion started from some point $x\in\partial\Bc_0$.
    Note that we only need to prove the result for large $n$ and $n'$.
    Suppose $n'\geq 3$. 
    Set $\Cc_{-1}=\varnothing$ and $\Cc_{[n'/3]}=\mathcal B_{n'}$. For $j=0,\ldots,[n'/3]-1$, let
    \[
    \Cc_j:=\mathcal B_{3j} \quad \text{and} \quad
    D_j:=\Cc_{j+1}^\circ\setminus \Cc_{j-1}.
    \]  
    Let $\Xi:=\{\partial \Cc_j: 0\le j\le [n'/3]\}$, and $(\tau^{i}_{\Xi})_{0\le i\le m}$ be the sequence of hitting times associated with $\Xi$ and $W$ (see \eqref{eq:oht}), such that $\tau^{m}_{\Xi}$ is the first time $W$ hits $\partial\Bc_{n'}$.
    Note that $m$ is a random integer. For each $0\le i\le m$, let $j_i$ be the random integer between $0$ and $n'$ such that $W(\tau^{i}_{\Xi})\in\partial\mathcal B_{j_i}$. Let 
    \[
    \Nc:=|\{ 0\le i< m: j_i\leq n \}|.
    \]
    Let $\mathscr S=\cup_{i\in\mathbb Z}\{\partial\mathcal B_{i}\}$. 
    By Lemma~\ref{lem:883} and scaling, for $w>0$, we can choose a positive integer $C$ such that for each $n\ge1$,
    \begin{equation}\label{align:1583}
        \Pb(\Nc>Cn) \le \Pb_{x}\big(\tau_{\mathscr S}^{Cn}\le L_{\Bc_{n}}\big)\le e^{-wn}.
    \end{equation}

    Below, we condition on $(j_i)_{0\le i\le m}$, such that $\{\Nc\le Cn\}$ occurs, and the points $x_i:=W(\tau^{i}_{\Xi})\in\partial\mathcal B_{j_i}$ for all $i$ (note $x_0=x$), then the conditional law of $W[0,\tau_{\partial\mathcal B_{n'}}]$, denoted by $\Qb$, can be written as (by the strong Markov property) 
    \[
    \mathbb Q=\oplus_{i=0}^{m-1}\,\Pb^{D_{j_i/3}}_{x_{i}\to x_{i+1}},
    \]
    where $\oplus$ denotes the image measure obtained by concatenation of paths.
    For all $i=1,\ldots,n$, let
    $$
    a_i=a_i(M):=\mathbf{1}\Big\{K\Big(p^{W[0,\tau_{\partial\mathcal B_{n'}}]\cup\mathcal B_{i-1}}_{i-1/2,i}\Big)<M\Big\}.
    $$
    For $h=1,\ldots,[(n+3)/6]$, let 
    $$f(h):=|\{0\le i\le m-1:j_i=6h-6\text{ or }j_i=6h-3\}|,$$ and define
    $$H:=\{h\ge 1:f(h)\le 24C,6h-4\le n-1\}.$$
    By Lemma \ref{lem:1530}, for any $\varepsilon>0$, there exists a constant 
    \begin{equation}\label{eq:Mdef}
    M=\max_{1\le k\leq24C}\{C_1(\varepsilon,k)\}\in (0,\infty)    
    \end{equation}
     such that for all $h\in H$,
    \begin{equation}\label{eq:Qeps}
        \mathbb Q(a_{6h-4}(M)=0)\le\varepsilon.
    \end{equation}
    Moreover, note that $a_{6h-4}$ for different $h$ are independent under $\Qb$.
    Since $\Nc\le Cn$, we have 
    \[
    |H|\geq [(n+3)/6]-\Nc/(24C)>n/8-1/2.
    \]
    Choosing $\varepsilon=e^{-144w}-e^{-153w}$ in \eqref{eq:Qeps}, by a standard large deviation estimate, we obtain that
    \begin{equation}\label{align:1600}
    \mathbb Q\left(\sum_{h=1}^{[(n+4)/6]}a_{6h-4}<\frac{n}{9}\right)
        \leq\ \mathbb Q\Big(\sum_{h\in H} a_{6h-4}<\frac{n}{9}\Big)\le e^{72w-nw}.
    \end{equation}
    Combining \eqref{align:1583} and \eqref{align:1600}, 
    we get
    $$
    \Pb_{x}\left(\sum_{h=1}^{[(n+4)/6]}a_{6h-4}<\frac{n}{9}\right)\le(1+e^{72w})e^{-wn}.
    $$
    For each $i=0,\ldots,5$, applying the above estimate to $W$ after its first visit of $\partial\mathcal B_i$, we get 
    $$\Pb_{x}\left(\sum_{h=1}^{[(n+4-i)/6]}a_{6h-4+i}<\frac{n-i}{9}\right)\le(1+e^{72w})e^{-w(n-i)}.$$ 
    It follows that 
    \[
    \Pb_{x}\Big(\sum_{i=1}^n a_i<\frac{2n}{3}-3\Big)<6(1+e^{72w})e^{-w(n-5)}.
    \] 
    By definition of nice paths (see \eqref{eq:Yc}) and rescaling, we conclude the proof.
\end{proof}

\subsection{Proof of Proposition \ref{pro:4}}
\label{subsec:pro:4}
This subsection is devoted to the proof of Proposition~\ref{pro:4}, the 3D counterpart of Proposition~5.1 in \cite{LSW02a}, whose proof we can heuristically explain as a three-step strategy: 1) make the two paths ``good'' (in other words, easily avoided by an $h$-process); 2) make them match up and walk together for a little while; 3) show that they are unlikely to decouple as long as they have already been walking together for quite a while. In our case, we will basically follow the same strategy, but there are two essential differences, as we now explain.

The first difference is the measurement of the quality of paths. To measure a ``good'' path in $2$D, one can simply require the path to stay in some cone; see the set $\Gamma^+$ introduced in Section 3.2 of \cite{LSW02a}. However, as such a condition is not sufficient in $3$D, we will use the switching constant (or BHP) introduced before instead; see the set $\Gamma_1$ defined in \eqref{eq:Ga1} below. Another difference is that we need to ensure the paths have enough  good layers when we couple.

We first review some results from \cite{G98} (see also \cite[Section 3]{LSW02a}), which will be used later. First,
by (10) in \cite{G98} and (3.6) in \cite{LSW02a}, we have the following result.
\begin{lemma}
    There exist constants $0<b_1=b_1(\lambda), b_2=b_2(\lambda)<\infty$ such that 
    \begin{equation}\label{eq:c1}
    \sup_{n\ge 1}\,\sup_{\gamma\in\Gamma}\,\rQ_n(\gamma)\le b_1,
\end{equation}
and for all $n,n'\geq1$ and $\gamma\in\Gamma$,
\begin{equation}\label{eq:b2}
        \rQ_n(\gamma)\leq b_2\, \rQ_{n'}(\gamma).
    \end{equation}
\end{lemma}

Consider two opposite cones $\Cone^-:=\Cone((-1,0,0),1/10)$ and $\Cone^+:=\Cone((1,0,0),1/10)$ (see \eqref{eq:cone} for the definition). 
For $\gamma\in \Gamma$, recalling that $X$ is the $h$-process from $0$ to $\partial\Bc_0$ avoiding $\gamma$, we define the following function in $\gamma$,
\begin{equation}\label{eq:varphi}
\varphi(\gamma):=\ind{\gamma_{1/2}\subseteq \Cone^-}\,\mathbb P(\hpro_{1/2}\subseteq\Cone^+ ),
\end{equation}
where $\gamma_{1/2}$ and $\hpro_{1/2}$ are the truncated versions of $\gamma$ and $\hpro$, respectively; see \eqref{eq:shift}.
We will use the following version of separation lemma, which is a consequence of \cite[Lemma 4.2]{G98}.
\begin{lemma}[Separation lemma]\label{lem:sep}
    There exists a constant $c_1=c_1(\lambda)$, such that for all $\gamma\in\Gamma$ and $n\ge 1$,
    $$\mathbb E\Big[Z_n(\widetilde \gamma^n)^\lambda \,\varphi(\olg^1)^\lambda\Big]\geq c_1\,\mathbb E\Big[Z_n(\widetilde \gamma^n)^\lambda\Big].$$
\end{lemma}

It follows from Lemma~\ref{lem:sep} that there exist positive constants $c_2(\lambda)$ and $c_3(\lambda)$ such that if we define the set of ``good'' paths by
    \begin{equation}\label{eq:Ga+}
\Gamma^+=\Gamma^+(\lambda):=\{\gamma\in\Gamma:\varphi(\gamma)\geq c_2\},
\end{equation}
then,
    \begin{equation}\label{eq:c23}
    \inf_{n\ge 1}\,\inf_{\gamma\in\Gamma}\,\widetilde{\mathbb P}^n_{\gamma}( \olg^1\in\Gamma^+)\ge c_3,
    \end{equation}
where $\widetilde{\mathbb P}^n_{\gamma}$ is the weighted measure defined in \eqref{eq:reweight}. Note that there is a $2$D version of $\Gamma^+$ introduced and used in \cite{LSW02a}. Our introduction of $\Gamma^+$ in $3$D is a little bit different, but it will play the same role as that in \cite{LSW02a}.
However, $\varphi(\gamma)$ involves all the information of $\gamma$ (when we define the associated $h$-process $X$), which is not good enough for later use (in particular, to derive \eqref{eq:U'} below if one replaces $\Gamma_1$ by $\Gamma^+$ in the second condition of $\Uc'$ there). Hence, we introduce the auxiliary set $\Gamma_1$ as follows.

\begin{equation}\label{eq:Ga1}
\Gamma_1(M):=\Big\{\gamma\in\Gamma:\gamma(0)\in\mathcal B_{-2},
K(p^{\gamma\cup\mathcal B_{-3/2}}_{-1,-1/2})<M,\gamma_2\subseteq\{{x^1}<0\}\setminus\mathcal B_{-3}\Big\},
\end{equation}
where we use $\{{x^1}<0\}$ to denote the set $\{(x^1,x^2,x^3)\in\Rb^3: x^1<0\}$.
Note that $\Gamma_1(M)$ only involves the information of $\gamma$ after time $\tau_{\partial\mathcal B_{-2}}$, compared to $\Gamma^+$.
\begin{lemma}\label{lem:ga1}
    There are constants $0<M_1<\infty$ and $c_4>0$ such that 
    \[\mathbb P_x(W[0,\tau]\in\Gamma_1(M_1))>c_4 \quad \text{ for all }\  x\in\partial\mathcal B_{-2}\cap\{x^1<-e^{-2}/2\}.
\]
\end{lemma}
\begin{proof}
    Note that $W[\tau_{\partial\mathcal B_{-2}},\tau]\subseteq\{{x^1}<0\}\setminus\mathcal B_{-3}$ occurs with uniformly positive probability. Moreover, combining Theorem \ref{theorem:bhp} and Lemma~\ref{lem:5.17}, we see that almost surely,
    \[
    K\Big(p^{W[0,\tau]\cup\mathcal B_{-3/2}}_{-1,-1/2}\Big)<\infty.
    \]
    Combining these two facts, we conclude the proof.
\end{proof}

In the following, we will fix $\Gamma_1:=\Gamma_1(M_1)$ with $M_1$ given by Lemma~\ref{lem:ga1}. 
Next, we show that $\rQ_1(\gamma)$ is uniformly bounded from below for all $\gamma\in\Gamma_1$.
\begin{lemma}\label{lem:142} 
    There exists a constant $C_1=C_1(\lambda)>0$ such that 
    \[       \inf_{\gamma\in\Gamma_1}\,\rQ_1(\gamma)\ge C_1. \]    
\end{lemma}
\begin{proof}
    Once the beginning paths are well-separated, we can attach non-intersecting Brownian motions from their endpoints to the next scale with positive probability. Therefore,
    it is sufficient to prove that for some absolute constant $C>0$,
    \begin{equation}\label{eq:Cse}
        \mathbb P_0\left(W[\tau_{\partial\mathcal B_{-1/2}},\tau]\subseteq \Cone((1,0,0),1/5) \con  W[0,\tau]\cap\gamma=\varnothing\right)\ge C.
    \end{equation}
    Note that there is an absolute constant $c$ such that for all $x\in \Cone^+=\Cone((1,0,0),1/10)$,
    \[
    \mathbb P_x\Big(W[\tau_{\partial\mathcal B_{-1/2}},\tau]\subseteq \Cone((1,0,0),1/5)\Big)>c.
    \]
    Letting $z_0=(e^{-1},0,0)$,
    we have
    \begin{align*}
        &\ \mathbb P_0\left(W[\tau_{\partial\mathcal B_{-1/2}},\tau]\subseteq \Cone((1,0,0),1/5) \con  W[0,\tau]\cap\gamma=\varnothing\right)\ge\ \frac{c\int_{\Cone^+} p_{-1/2}^\gamma(0,y)\sigma_{-1/2}(dy)}{\int p_{-1/2}^\gamma(0,y)\sigma_{-1/2}(dy)}\\
        =&\ \frac{c\int_{\Cone^+}\int p_{-1}^\gamma(0,z)p_{-1,-1/2}^\gamma(z,y)\sigma_{-1}(dz)\sigma_{-1/2}(dy)}{\iint p_{-1}^\gamma(0,z)p_{-1,-1/2}^\gamma(z,y)\sigma_{-1}(dz)\sigma_{-1/2}(dy)}
        \ge\ c\inf_{z\in\partial\mathcal B_{-1}\setminus\gamma}\frac{\int_{\Cone^+} p_{-1,-1/2}^\gamma(z,y)\sigma_{-1/2}(dy)}{\int p_{-1,-1/2}^\gamma(z,y)\sigma_{-1/2}(dy)}\\
        \ge&\ \frac{c}{K(p_{-1,-1/2}^\gamma)}\frac{\int_{\Cone^+} p_{-1,-1/2}^\gamma(z_0,y)\sigma_{-1/2}(dy)}{\int p_{-1,-1/2}^\gamma(z_0,y)\sigma_{-1/2}(dy)}
        \ge\ \dfrac{c\int_{\Cone^+}p_{-1,-1/2}^{\{x^1\le0\}\cup\mathcal B_{-2}}(z_0,y)\sigma_{-1/2}(dy)}{M_1},
    \end{align*}
    where Lemma \ref{lemma:compare} is used in the last inequality.
    This implies \eqref{eq:Cse} and 
    completes the proof.
\end{proof}

Recall that $\gamma_m$ is the truncated version of $\gamma$ (see \eqref{eq:shift}), and $\Yc_{m'',m',m}$ is the set of nice paths defined in \eqref{eq:Yc}.
For all $l,l'\geq1$, define
\[
    \mathcal K_{l',l}:=\Big\{(\gamma,\gamma')\in\Gamma^2\con \text{there exists } l''\geq l \text{ such that } \gamma_{l''+l'}=\gamma'_{l''+l'} \text{ and } \gamma\in\mathcal Y_{l''+l',l',l}\Big\}.
\]
Suppose that $u>0$ is some small constant to be fixed. We further define
\[
    \ol{\mathcal K}_{l',l}=\ol{\mathcal K}_{l',l}(u):=\Big\{(\gamma,\gamma')\in\mathcal K_{l',l}\con \rQ_1(\gamma)\geq C_1\,e^{-lu},\rQ_1(\gamma')\geq C_1\,e^{-lu}\Big\},
\]
where $C_1$ is the constant in Lemma \ref{lem:142}.
Recall the random variable $\delta(n,k,\gamma)$ defined below \eqref{eq:reweight}. We have the following coupling for $\delta(n,k,\gamma)$ and $\delta(n,k,\gamma')$ when $\gamma,\gamma'\in\Gamma^+$, which is the $3$D counterpart of Lemma 5.3 in \cite{LSW02a}. 
Note that the proof is different since the mirror coupling used therein does not work here, and moreover, one needs to work with $\Gamma_1$ rather than $\Gamma^+$ at some intermediate steps.

\begin{lemma}\label{lem:167}
    There exists some large constant $a_0>0$ such that the following holds for all $a\ge a_0$.
    There exists $\overline c_1=\overline c_1(\lambda,a)>0$ such that for all $\gamma,\gamma'\in\Gamma^+$,
    $n\geq 3a+1$ and $u>0$, 
    one can define $\delta=\delta(n,3a+1,\gamma)$ 
    and $\delta'=\delta(n,3a+1,\gamma')$ 
    on the same probability space $(\Omega,\mathcal F,\mu)$ 
    such that
    \begin{equation*}
        \mu\big((\delta,\delta')\in \ol{\mathcal K}_{a,a}\big)\geq\overline c_1.
    \end{equation*}
\end{lemma}
\begin{proof}
    The Radon-Nikodym derivatives of the laws of $\delta$ and $\delta'$ with respect to those of $\hg^{3a+1}$ and $\hgp^{3a+1}$ are given by
    \[q=q(\hg^{3a+1}):=
    \dfrac{e^{(3a+1)\xi} Z_{3a+1}(\tg^{3a+1})^\lambda \rQ_{n-3a-1}(\olg^{3a+1})}{\rQ_n(\gamma)},\quad \text{and}
    \]
    \[q'=q'(\hgp^{3a+1}):=
    \dfrac{e^{(3a+1)\xi} Z_{3a+1}(\tgp^{3a+1})^\lambda \rQ_{n-3a-1}(\olgp^{3a+1})}{\rQ_n(\gamma')}\quad\text{resp.}
    \]

    Let $a\ge 3$. Let $\mathcal U$ be the set of paths $\eta$ such that $\eta(0)\in\partial\mathcal B_1$, $\eta(t_{\eta})\in\partial\mathcal B_{3a+1}$, $\eta(0,t_\eta)\in\mathcal B_{3a+1}^\circ$. Let $\Uc'$ be the paths in $\Uc$ such that the following three conditions hold.
    \begin{enumerate}[label=(\arabic*)]
        \item $\eta\subseteq \{x^1<0\}\setminus\mathcal B_{-1/2}$ and $|\eta(0)/|\eta(0)|-(-1,0,0)|<1/5$.
        \item $e^{-3a-1}\eta\in\Gamma_1$.
        \item $e^{-3a-1}\eta\in\mathcal Y_{3a,a,a}$.
    \end{enumerate} 
    Now, consider the following probability measure $\nu$ on $\Uc$,
    \[
        \nu:=\sigma_1(\partial\Bc_1)^{-1}\,\int_{\partial\Bc_1}\mathbb P_{x}^{\mathcal B_{3a+1}}\,\sigma_1(dx),
    \]
    where $\Pb_{x}^{\mathcal B_{3a+1}}$ denotes the measure of a Brownian motion started from $x$ and stopped when it first hits $\partial\Bc_{3a+1}$.
    On one hand, by Lemma \ref{lem:ga1}, the first two conditions hold with uniformly positive probability under $\nu$. On the other hand,
    by Proposition \ref{pro:2}, the third condition holds with probability close to $1$ when $a$ large. 
    Combined, we conclude that there are absolute constants $\rho$ and $a_0$ such that for all $a\ge a_0$, 
    \begin{equation}\label{eq:U'}
        \nu(\mathcal U')\ge \rho.
    \end{equation}
    
    We define the measures $\nu_1$ and $\nu_2$ as follows. For any measurable set $\Lambda\subseteq\Uc$,
     \[
    \nu_1(\Lambda):=\mathbb E\Big[q(\hg^{3a+1});\,\hg^{3a+1}_1\in\Lambda,\hg[0,\tau_{\partial\mathcal B_1}]\subseteq \{x^1<0\}\setminus\mathcal B_{-1/2}\Big],
    \]
    and define $\nu_2$ with respect to $\hgp$ in the same way.
    
    We now claim that there exists a constant $c=c(\lambda,a)>0$ such that $\nu_i(\Lambda)\ge c\,\nu(\Lambda)$ 
    for $i=1,2$ and all $\Lambda\subseteq\mathcal U'$. For brevity we only discuss the case of $\nu_1$. To this end, it is sufficient to show that $q(\hg^{3a+1})$ is bounded away from $0$, if $\hg[0,\tau_{\partial\mathcal B_1}]\subseteq \{x^1<0\}\setminus\mathcal B_{-1/2}$ (which occurs with positive probability), and $\hg^{3a+1}_1\in\Lambda$.
    Note that $\rQ_n(\gamma)\le b_1$ by \eqref{eq:c1}, and $\rQ_{n-3a-1}(\olg^{3a+1})\ge C_2/b_2$ by Lemma \ref{lem:142} and \eqref{eq:b2}. Moreover, $Z_{3a+1}(\tg^{3a+1})\ge c'(\lambda,a)$, because the $h$-process $X$ has probability at least $c_2$ to stay in the cone $\Cone^+$, which is well-separated from $\gamma\in \Gamma^+$ by \eqref{eq:Ga+}, and then we can force the future part $\widehat X^{3a+1}$ to stay in a larger cone $\Cone((1,0,0),1/5)$ to avoid the whole path $\tg^{3a+1}$ with positive probability. Hence, we conclude that $q(\hg^{3a+1})$ is bounded away from $0$ and $\nu_i(\Lambda)\ge c\,\nu(\Lambda)$, restricted to the mentioned events. 
    
    From this fact, combined with \eqref{eq:U'}, we can construct a coupling $\mu$ such that 
    \[\mu\big(\{\delta_{3a}=\delta_{3a}'\}\cap\{\delta\in\mathcal Y_{3a,a,a}\cap\Gamma_1\}\big)\ge c\,\nu(\mathcal U ')\ge c\rho.\]
    Moreover, from Lemma \ref{lem:142}, we have $\rQ_1(\delta),\rQ_1(\delta')\ge C_1$. This concludes the proof.
\end{proof}

The next lemma shows that the paths are not likely to decouple as long as they are matched up. For technical reason, we distinguish two cases, by
setting $\mathcal K_{l',l}^0:=\mathcal K_{l',l}$, and $\mathcal K_{l',l}^n:=\overline{\mathcal K}_{l',l}$ for $n\ge1$.
\begin{lemma}\label{lem:226}
    There exists $u_1=u_1(\lambda)>0$, 
    such that
    for all $0<u<u_1(\lambda)$, 
    there exist $w_1=w_1(\lambda,u)>0$ and $\overline c_2=\overline c_2(\lambda,u)>0$ such that 
    for all $n\geq 1$, $l'\geq l\geq 1$, and
    $(\gamma,\gamma')\in\ol{\mathcal K}_{l,l '}$, 
    one can define  $\delta=\delta(n,1,\gamma)$ 
    and $\delta'=\delta(n,1,\gamma')$ on the same 
    probability space $(\Omega,\mathcal F,\mu)$ such that
    \begin{equation*}
        \mu\big((\delta,\delta')\in\mathcal K_{l'+1,l}^{n-1}\big)\geq 1-\overline c_2\,e^{-w_1l}.
    \end{equation*}
\end{lemma}
    Basically, one can follow the proof of Lemma 5.2 in \cite{LSW02a} line by line to obtain Lemma \ref{lem:226}. Here we only mention two differences in the proof detail.
    \begin{itemize}
        \item When $n=1$, there is no ``next scale'' for us to use to deduce the condition $\rQ_1(\gamma)\geq C_1\,e^{-lu}$ in $\overline{\mathcal K}_{l',l}$. In this case, we use $\mathcal K_{l',l}$ alternatively.
        \item As an intermediate step, one needs to deal with the situation that 
        $(\gamma,\gamma')\in\mathcal K_{l',l}$ and $(\olg^1,\olgp^1)\in\Gamma^2\setminus \mathcal K_{l'+1,l}$.
        In this case, $\hg^1$ has a down-crossing from $1$ to $e^{-l'}$ to destroy some good layers inside $\Bc_{-l'}$, which has probability $O(e^{-c\,l})$ (note that $l'\ge l$) from the transience of $3$D Brownian motion. 
    \end{itemize}

Finally, we follow a strategy similar to that used in the proof Proposition 5.1 of \cite{LSW02a} to obtain Proposition \ref{pro:4}, which requires iterated use of Lemmas~\ref{lem:167} and~\ref{lem:226}. 
We remind readers that there are some extra complexities here since we also need to increase the number of good layers along the way, which is ensured by Proposition \ref{pro:2}.
\begin{proof}[Proof of Proposition \ref{pro:4}]
    Below, we fix $u=\min\{1/2,u_1(\lambda)/2\}$.
    By Lemma \ref{lem:226},
    there exist $w_1=w_1(\lambda,u)>0$ and $\overline c_2=\overline c_2(\lambda,u)>0$ such that 
    for all $n\geq 1$, $l'\geq l\geq 1$, and 
    if $(\gamma,\gamma')\in\ol{\mathcal K}_{l',l }$, 
    then one can define  $\delta=\delta(n,1,\gamma)$ 
    and $\delta'=\delta(n,1,\gamma')$ on the same 
    probability space such that
    \begin{equation}\label{eq:129}
        \mu\big((\delta,\delta')\in\mathcal K_{l '+1,l}^{n-1}\big)\geq 1-\overline c_2\,e^{-w_1l}.
    \end{equation}
    
    Using \eqref{eq:129} iteratively for $4l$ times, for all $n\geq 4l$, $l\geq 1$, and $(\gamma,\gamma')\in\ol{\mathcal K}_{l,l}$, one can construct  $\delta=\delta(n,4l,\gamma)$ and $\delta'=\delta(n,4l,\gamma')$ on the same probability space such that
    \begin{equation*}
        \mu\big((\delta,\delta')\in\mathcal K_{5l,l}^{n-4l}\big)\geq 1-4l\,\overline c_2\,e^{-w_1l}.
    \end{equation*}
    Now, we take into account the number of good layers.
    We observe that the law of $\delta$ has the Radon-Nikodym derivative $q$ with respect to the law of $\hg^{4l}$ such that
    \begin{equation*}
    q(\hg^{4l})=\dfrac{e^{4l\xi}Z_{4l}(\tg^{4l})^\lambda \rQ_{n-4l}(\olg^{4l})}{\rQ_n(\gamma)}
        \le e^{4l\xi}\,b_1\,b_2\,C_1^{-1}\,e^{lu},
    \end{equation*}
    where we used the fact $Z_{4l}(\tg^{4l})\le 1$, $\rQ_{n-4l}(\olg^{4l})\le b_1$ and $\rQ_n(\gamma)\ge b_2^{-1}\,\rQ_1(\gamma)\ge b_2^{-1}\,C_1\,e^{-lu}$ by \eqref{eq:b2}. 
    Applying Proposition \ref{pro:2} with $w=2\xi+1$ and $M_0=M_0(\max\{4,2\xi+1\})$, we get
    \begin{equation*}
        \mathbb E\Big(q(\hg^{4l})\,\ind{\olg^{4l}\notin\mathcal Y_{4l,2l,l}(M_0)}\Big)\leq (e^{4l\xi}\,b_1\,b_2\,C_1^{-1}\,e^{lu})\,a_3\,e^{(-4\xi-2)l}.
    \end{equation*}
    It follows that for all $n\geq 4l$, $l\geq 1$, and $(\gamma,\gamma')\in\ol{\mathcal K}_{l,l}$, one can construct  $\delta=\delta(n,4l,\gamma)$ and $\delta'=\delta(n,4l,\gamma')$ such that for $w_2=\min\{w_1/2,1\}$ and some constant $\overline c_3=\overline c_3(\lambda)$,
    \begin{equation}\label{eq:152}
        \mu\big((\delta,\delta')\in\mathcal K_{2l,2l}^{n-4l}\big)
        \geq \mu\big((\delta,\delta')\in\mathcal K_{5l,l}^{n-4l},
        \delta\in\mathcal Y_{4l,2l,l}(M_0)\big)
        \geq 1-\overline c_3\,e^{-w_2l}.
    \end{equation}

    Next, we fix a constant $a$ such that $\overline c_2\,e^{-w_1 a}<1/2$, $\overline c_3\,e^{-w_2 a}<1/2$, and $a>a_0$, where $a_0$ is as in Lemma \ref{lem:167}. Then, 
    by Lemma \ref{lem:167}, 
    there exists 
    a constant $\overline c_1=\overline c_1(\lambda,a)>0$, 
    such that for all $\gamma,\gamma'\in\Gamma^+$, 
    and $n\geq 3a+1$, 
    one can construct $\delta=\delta(n,3a+1,\gamma)$ 
    and $\delta'=\delta(n,3a+1,\gamma')$ 
    in such a way that
    \begin{equation}\label{eq:168}
        \mu\big((\delta,\delta')\in \ol{\mathcal K}_{a,a}\big)\geq\overline c_1.
    \end{equation}

    Finally, we use the above two types of couplings, \eqref{eq:152} and \eqref{eq:168}, alternately to construct the desired coupling.
    Similar to (5.2) in \cite{LSW02a}, we have for all $j+k\le n$,
    \begin{equation}\label{eq:induc}
    \delta(n-j,k,\delta(n,j,\gamma)) \text{ has the same law as } \delta(n,k+j,\gamma).
    \end{equation}
    Suppose $n\ge 6a+4$. 
    We construct a sequence $(\delta(j),\delta'(j),N_j,K_j)$ 
    ($j=0,\ldots,\iota$) with
    $N_0=0,K_0=0$ and $\delta(0)=\gamma,\delta'(0)=\gamma'$ as follows, where $\iota$ is some stopping time to be defined.
    For $j\ge 1$, we use an inductive construction and distinguish two cases below:
    \begin{itemize}
        \item If $K_j\ge 1$, $n\ge N_j+4K_j$ and $(\delta(j),\delta'(j))\in\ol{\mathcal K}_{K_j,K_j}$, we let $\delta(j+1)$ 
        have the law of $\delta(n-N_j,4K_j,\delta(j))$ and $\delta'(j+1)$ 
        have the law of $\delta(n-N_j,4K_j,\delta'(j))$, and couple $\delta(j+1)$ and $\delta'(j+1)$ as in \eqref{eq:152}. Set $N_{j+1}=N_j+4K_j$.
        Set $K_{j+1}=2K_j$ if the coupling succeeds, and $K_{j+1}=0$ otherwise.
        \item If $K_j=0$ and $n\ge N_j+3a+2$, we let $\delta(j+1)$ have the law of $\delta(n-N_j,3a+2, \delta(j))$ and $\delta'(j+1)$ have the law of $\delta(n-N_j,3a+2, \delta'(j))$. In this case, we first use \eqref{eq:c23} to concatenate $\delta(j)$ (resp.\ $\delta'(j)$) with the first layer of $\delta(j+1)$ (resp.\ $\delta'(j+1)$) in $\Gamma^+$ with positive probability, and then couple the next $j$ layers of $\delta(j+1)$ and $\delta'(j+1)$ using \eqref{eq:168}. Set $N_{j+1}=N_j+3a+2$. 
        Set $K_{j+1}=a$ if the coupling succeeds, and $K_{j+1}=0$ otherwise.
    \end{itemize}
    Let $\iota$ be the first $j\ge 0$ such that $N_{j+1}>n$.
    By \eqref{eq:induc}, $\delta(j)$ has the same law as $\delta(n,N_j,\gamma)$ (which is also true for $\delta'(j)$ and $\gamma'$).
    Introduce an auxiliary sequence $X_0,\ldots,X_n$ such that $X_{N_{j}}=0$ if $K_j=0$ for $j=0,\ldots,\iota$, and $X_{l}=X_{l-1}+1$ otherwise.
    By Theorem 3.1 of \cite{vermesi2008intersection} (one can easily verify that the sequence $(X_i)_{0\le i\le n}$ satisfies the condition therein), there exist positive constants $\ol c_4(\lambda)$ and $w_3(\lambda)$ such that
    $$
    \mathbb P(X_n\ge n/2)\ge 1-\ol c_4\,e^{-w_3\,n}.
    $$
    Moreover, it is easy to check that $X_n\le 8K_\iota$ if $X_n\ge 3a+2$. Hence, we also obtain that $\mathbb P(K_\iota\ge n/16)\ge 1-\ol c_4\,e^{-w_3\,n}$. It means that one can couple $\delta(n,N_{\iota},\gamma)$ and $\delta(n,N_{\iota},\gamma')$ such that they are in $\ol{\mathcal K}_{n/16,n/16}$ with high probability. Note that there are at most $(n-N_\iota)\le n/2$ layers remaining to be coupled, for which we can use the coupling \eqref{eq:129} repeatedly $O(n)$ times with a cost $O(e^{-w_1\, n})$ each time. We conclude the proof by choosing the constants $a_4$ and $v_3$ in the statement of Proposition~\ref{pro:4} sufficiently small. 
\end{proof}

\section{Twisted H\"older domain}\label{sec:twisted}
This section is dedicated to the proof of Theorem~\ref{thm:twisted}. We begin with the definition of twisted H\"older domain, which can be found in Definition~4.1 of \cite{MR1127476}. 
\begin{definition}\label{def:thd}
A bounded domain $D$ is called a twisted H\"older domain of order $\alpha$, $\alpha\in (0,1]$, if there exist positive constants $C_1,\ldots,C_5$, a point $x_0\in D$ and a bounded continuous function $\delta:D\to(0,\infty)$ such that the following properties (i), (ii) and (iii) hold.

(i) $\delta(x)\le C_1\,\dist(x,\partial D)^\alpha$ for all $x\in D$;

(ii) For all $x\in D$, there exists a rectifiable curve $\gamma_x\subseteq D$ connecting $x$ and $x_0$ such that $$\delta(x')\ge C_2\big(\delta(x)+l(\gamma_x(x,x'))\big)$$ for all $x'$ lying on $\gamma_x$, where $l(\gamma_x(x,x'))$ denotes the length of the piece of curve connecting $x$ and $x'$;

(iii) Let $F_a:=\{y\in D:\delta(y)\le a\}$. For any $a\le C_5$ and $x\in F_a$,
\begin{equation}\label{eq:cap}
\CAP(\B(x,C_3a)\cap F_a^c)\geq C_4\CAP(\B(x,C_3a)),
\end{equation}
where $\CAP(\cdot)$ is the Newtonian capacity in $\Rb^3$ (see Definition 4.31 in \cite{BM1}).
\end{definition}

We will use the following lemma to verify conditions of twisted H\"older domains.
\begin{lemma}[{\cite[Proposition 3.1]{bass1992lifetimes}}]\label{lem:vthd}
    If for any $x\in D$, there exists $b>0$ such that
\begin{equation}\label{eq:2193}
\dist(x',\partial D)\ge C_6(b+l(\gamma_x(x,x')))^{1/\alpha},\ \text{ for all } x'\in\gamma_x,
\end{equation}
letting $\delta(x)$ be the supremum of $b$'s that satisfy \eqref{eq:2193}, then this ${\delta}$ is bounded continuous, and satisfies conditions (i) and (ii) in Definition~\ref{def:thd}.
\end{lemma}

For simplicity, we assume $U=\B(0,1)^\circ$ below. The general case can be tackled in a similar way.
We recall $\Wc$ and $\Pbx$ introduced just above Theorem~\ref{thm:sep}.
We write $\Cone_x(v,r,l):=\Cone_x(v,r)\cap\B(x,l)$ and $\Dc:=U\setminus\unionW$.

In the following lemma, we show that for any point $x\in U$ with $\dist(x,\Wc)\ge 2^{-n}$, there is an uncovered cone with vertex $x$ such that it has a macroscopic section (at least unit length) inside $\Dc$. This is carried out by placing the cones within some appropriate direction.
\begin{lemma}\label{lem:cone_in}
For $\Pbx$-almost every $\Dc$ and any $u_1>1$, there exists a integer $n_0=n_0(u_1,\Dc)>0$, such that for any $n>n_0$ and any $x\in(2^{-n}\mathbb{Z}^3)\cap U$ with $\dist(x,\Wc)\ge 2^{-n}$, there exists a unit vector $v_x$, such that
\[\Cone_x(v_x,u_1^{-n},1)\subseteq \Dc.\]
\end{lemma}
\begin{proof}
    First, by adapting the proof of Proposition~\ref{prop:cone1}, one can obtain a stronger result as follows. 
    Suppose $v_0$ is a unit vector in $\mathbb R^3$ and $t<1$.
    Then, under the setting of Proposition~\ref{prop:cone1}, there exists 
    $n_1(u_1,u_2,k,t)>0$, such that for all $n>n_1$, we have
    \begin{equation}\label{eq:strong}
        \Pbx\big(\mathrm{there\ exists\ a\ unit\ vector\ }v,~ v\cdot v_0>t\text{ and } \unionW\cap\Cone(v,u_1^{-n},2)=\varnothing\big)>1-u_2^n.
    \end{equation}
    Roughly speaking, one can further construct an uncovered cone which follows closely to some prescribed direction.
    Using this result, for any $x$, with probability at least $1-u_2^n$, one can choose an uncovered cone with direction $v_x$ such that $\Cone_x(v_x,u_1^{-n},1)\subseteq\Dc$.
    Choosing $u_2$ sufficiently small and applying a union bound, we conclude the proof immediately.
\end{proof}

Next, we use Lemma~\ref{lem:cone_in} to construct an arc $\gamma_x$ in $\Dc$ satisfying some bound that can be applied to show \eqref{eq:2193}. The strategy is to let $\gamma_x$ be a poly-line following directions of successive uncovered cones.
\begin{lemma}\label{lem:continuous curve}
    Let $x_0\in U\setminus \{x_1,\ldots,x_k\}$. For $\Pbx$-almost every $\Dc$ and any $0<\beta<1$, there exist a constant $c_5=c_5(\beta,\Dc,x_0)>0$ such that for all $x\in \Dc$, we can construct a rectifiable curve $\gamma_x \subseteq \Dc$ connecting $x$ and $x_0$ satisfying
    \begin{equation}\label{eq:lem:6.1}
        l(\gamma_x(x,x'))\le c_5\,\dist(x',\partial \Dc)^\beta,\ \text{ for all } x' \in \gamma_x.
    \end{equation}
\end{lemma}
\begin{proof}
    Let $1<u_1<1.5$ be a constant such that $(u_1/2)< (2u_1)^{-\beta}$. Let $n_0=n_0(u_1,\Dc)$ be the corresponding constant in Lemma~\ref{lem:cone_in}.
    For $x\in \Dc\setminus(2^{-n_0}\mathbb Z^3)$, we are going to use a well-chosen poly-line to construct $\gamma_x$ that satisfies \eqref{eq:lem:6.1}. 
    First, let 
    \[n=n(x):=1+\max\{n_0,[-\log_{2}(\dist(x,\Dc^c))]\},\]
    so that $\dist(x,\Dc^c)\ge 2^{-n+1}$. Then $\gamma_x(x,x_0)$ is constructed by concatenating the poly-line segment $\overline{z_ny_nz_{n-1}\cdots y_{n_0+1}z_{n_0}y_{n_0}}$ and $\gamma_{y_{n_0}}$, where $\gamma_{y_{n_0}}=\gamma_{y_{n_0}}(x_0)$ is a finite-length rectifiable curve connected $y_{n_0}$ and $x_0$, and $z_i,y_i$ are chosen inductively as follows.
\begin{enumerate}[label=(\arabic*)]
\item Let $z_n=x$ and pick $y_n\in(2^{-n}\mathbb Z^3)$ such that $|y_n-x|<2^{-n}$.
\item Suppose $y_m\in(2^{-m}\mathbb Z^3)\cap U$ has already been constructed $(m>n_0)$. Let \[z_{m-1}=y_m+10(u_1/2)^mv_{y_m},\]
    where $v_{y_m}$ is the unit vector that satisfies Lemma \ref{lem:cone_in} with respect to $y_m$, and pick $y_{m-1}\in(2^{-m+1}\mathbb Z^3)\cap U$ such that $|y_{m-1}-z_{m-1}|<2^{-m+1}$.
\end{enumerate}    
    By Lemma \ref{lem:cone_in}, for any $n_0<m\le n$, we have 
    \[\B(z_{m-1},2^{-m+2})\subseteq\Cone_{y_m}(v_{y_m},u_1^{-m},1)\subseteq \Dc.\]
    Therefore, $y_{m-1}\in \Dc$ and $\dist(y_{m-1},\Dc^c)\ge 2^{-m+1}$. 
    From the above construction, it is not hard to verify \eqref{eq:lem:6.1} by brute-force computation. We omit the details and conclude the proof.
\end{proof}

Now, we are going to prove Theorem~\ref{thm:twisted} in three steps. 
In the first two steps, we show that $\Dc=U\setminus\unionW$ is a twisted H\"older domain of any order $\alpha<1$ by verifying \eqref{eq:2193} and \eqref{eq:cap}, making use of Lemmas~\ref{lem:continuous curve} and \ref{lem:1}, respectively. In the last step, we deal with the general case when $\Wc$ is replaced by any of its closed subset.
\begin{proof}[Proof of Theorem~\ref{thm:twisted}]
Let $x_0\in U\setminus \{x_1,\ldots,x_k\}$. Let $\alpha\in (0,1)$. Set $\beta:=\sqrt{\alpha}$. For $x\in \Dc$, let $d(x):=\dist(x,\Dc^c)$. 

\smallskip
\noindent \textbf{Step 1: Verify \eqref{eq:2193}.} We first show that the constant $b$ in \eqref{eq:2193} exists and $\delta(x)$ defined in Lemma \ref{lem:vthd} satisfies $\delta(x)\ge c_6d(x)^\beta$ for some constant $c_6$.

By Lemma~\ref{lem:continuous curve}, there exists a constant $c_5$ and a rectifiable curve $\gamma_x \subseteq \Dc$ connecting $x$ and $x_0$ such that 
    \[l(\gamma_x(x,x'))\leq c_5\,d(x')^\beta\le c_5\,d(x')^\alpha,\ \text{ for all } x'\in\gamma_x.\]
    For any $x'\in\gamma_x(x,x_0)$, we have either $d(x')\ge d(x)/2$ or $d(x)/2\le l(\gamma_x(x,x'))\le c_5 d(x')^\beta$ since $d(x)\le d(x')+l(\gamma_x(x,x'))$.
    Hence, there exists a constant $c_6$ depending on $c_5$, such that for any $x\in \Dc$ and $x'\in\gamma_x(x,x_0)$,
    \[2c_5\,d(x')^\alpha-l(\gamma_x(x,x'))\ge c_5\,d(x')^\alpha\ge c_5\,\min\{d(x)/2,(d(x)/2c_5)^{1/\beta}\}^\alpha\ge c_6d(x)^\beta.\]  
    Taking $C_6=(2c_5)^{-1/\alpha}$, we have $\delta(x)\ge c_6d(x)^\beta$.

\smallskip
\noindent \textbf{Step 2: Verify \eqref{eq:cap}.} 
From the first step,
we have for any $a>0$, 
\begin{equation}\label{eq:pf_thd_1}F_{c_6a}^{c}=\{y\in \Dc:\delta(y)>c_6a\}\supset\{y\in \Dc:d(y)>a^{1/\beta}\}.\end{equation}
    Taking $C_3=2/c_6$, it suffices to show that there exists an absolute constant $c>0$ and a constant $0<C_5(\Dc)<1$, such that for any $x\in\B(0,1)$ and any $0<a\le C_5$, we have 
    \[\CAP(\B(x,2a)\cap F_{c_6a}^c)\ge c\,\CAP(\B(x,2a)).\]
    It can be deduced from Theorem 8.24 in \cite{BM1} that, for any subset $A\subseteq B(x,2a)$, there are universal constants $\rho_1,\rho_2>0$ such that 
    \[
    \rho_1\, a^{-1}\, \Pbf_z(W\cap A\neq\varnothing)\le \CAP(A)\le \rho_2\, a^{-1}\, \Pbf_z(W\cap A\neq\varnothing),
    \]
    where $z:=x+(3a,0,0)$.
    Therefore, it suffices to prove that, using \eqref{eq:pf_thd_1}, for all $x\in\B(0,1)$ and any $0<a\le C_5$, we have 
    \[
    \Pbf_z\Big(W\cap (\B(x,2a)\cap\{x:d(x)>a^{1/\beta}\})\neq\varnothing\Big)>c'.
    \]
    Set $\overline x:=x(1-a)$. By  strong Markov property, we have
    \begin{align*}
        &\Pbf_z\Big(W\cap (\B(x,2a)\cap\{x:d(x)>a^{1/\beta}\})\neq\varnothing\Big)\\
        \ge\ &\Pbf_z\Big(W\cap \B(\overline x,a/16)\neq\varnothing\Big)\inf_{y\in\partial\B(\overline x,a/16)}\Pbf_y\Big(W\cap (\B(\overline x,a)\cap\{x:d(x)>a^{1/\beta}\})\neq\varnothing\Big).
    \end{align*}
    Note that the first probability on the RHS is bounded away from $0$. Hence,
    it suffices to show
    that the following holds almost surely under $\Pbx$. For any $0<a\le C_5(\Dc)$, any $x\in B(0,1-a)$ and  $y\in\partial\B(x,a/16)$, we have
\begin{equation}\label{eq:pro:6.2}
        \Pf_y\big(W\cap\B(x,a)\subseteq \sausage{\unionW}{a^{1/\beta}}\big)<c''.
    \end{equation}
    The above can be proved by using a similar argument as  the proof of Theorem~\ref{thm:sep} (in particular Step 2), using a combination of Lemma~\ref{lem:1}, discrete approximation and Harnack's inequality. We do not repeat the argument here. This yields \eqref{eq:cap}, which combined with the first step and Lemma \ref{lem:vthd}, implies that $U\setminus\Wc$ is a twisted H\"older domain of any order $\alpha<1$.
    
    \smallskip
   \noindent \textbf{Step 3: General $\Ac$.} Suppose $\Ac$ is any closed subset of $\unionW$. Let $\Dc_1=U\setminus \Ac$. Let $\delta_{\Dc}$ and $\delta_{\Dc_1}$ be $\delta$ associated with $\Dc$ and $\Dc_1$, respectively, as in Lemma \ref{lem:vthd}, such that $C_6(\Dc_1)=C_6(\Dc)/2$. Since $\Dc\subseteq \Dc_1$, we have $\delta_{\Dc_1}\ge \delta_{\Dc}$ on $\Dc$. For $x\in \Dc_1\setminus \Dc$ and $\varepsilon>0$, pick $y\in \Dc$ such that $|y-x|\le\min\{\dist(x,\partial \Dc_1)/2),\varepsilon\}$, construct $\gamma_x$ by joining the line segment $\ol{xy}$ and $\gamma_y$. We can see that $\delta_{\Dc_1}(x)>0$ if we pick a sufficiently small $\varepsilon$. Moreover, it is obvious that Definition~\ref{def:thd} (iii) holds true for $\Dc_1$ from the fact $\delta_{\Dc_1}\ge\delta_{\Dc}$. We thus finish the proof.
\end{proof}

\appendix
\section{Properties of $K(p)$ and $R(p)$}\label{appen:A}
In this section, we prove several properties of the extremal total variance distance and the switching constant introduced in Section \ref{subsection:RK}, namely Lemmas \ref{lemma:RK}--\ref{lemma:tricky}. 
\begin{proof}[Proof of Lemma \ref{lemma:RK}]
    We first show $R(p)\leq 1-\dfrac{2}{1+\sqrt{K(p)}}$. Since $R(p)\le 1$, we can assume $K(p)<\infty$. 
    It suffices to show that for any $x_1,x_2\in E$, $x_1\neq x_2$ and $\mu$ such that $0<\int p(x_i,z)\mu(dz)<\infty\ (i=1,2)$,

    \begin{equation}\label{eq:Kone}
    \frac{1}{2}\int \left|q(x_1,y)-q(x_2,y)\right|\mu(dy)\le 1-\frac{2}{1+\sqrt{K(p)}},
    \end{equation}
    where  
    $$q(x_i,y):=\frac{p(x_i,y)}{\int p(x_i,z)\mu(dz)},\ \ i=1,2.$$
    Let $$S:=\{y\in F:q(x_1,y)>q(x_2,y)\}\mathrm{\ \ and\ \ }T:=\{y\in F:q(x_1,y)\le q(x_2,y)\}.$$
    For $i=1,2$, let $u_i=\int_S q(x_i,y)\mu(dy)$ and $v_i=\int_T q(x_i,y)\mu(dy)$. Then we immediately have $u_1+v_1=1$, $u_2+v_2=1$, $0\le u_2\le u_1\le 1$ and
    \begin{equation}\label{eq:Ktwo}
    \frac12\int\left| q(x_1,y)-q(x_2,y)\right|\mu(dy)=u_1-u_2.
    \end{equation}
    Note that either $u_1=1$ or $u_2=0$ will imply $u_1-u_2=0\leq 1-\frac{2}{1+\sqrt{K(p)}}$ since $K(p)\geq1$. Hence, we assume that $0<u_2\leq u_1<1$.
    By the definition of $K(p)$, we have
    \begin{equation}\label{eq:A1}
    \begin{split}
        u_1(1-u_2)&=\,
        \int_S\int_Tq(x_1,y_1)q(x_2,y_2)\mu(dy_2)\mu(dy_1)\\
        &\leq \,K(p)\int_S\int_Tq(x_1,y_2)q(x_2,y_1)\mu(dy_2)\mu(dy_1)= 
        K(p)u_2(1-u_1).
    \end{split}
    \end{equation}
    By \eqref{eq:A1} and Cauchy-Schwarz inequality, we have
    $$
    1-\frac{2}{1+\sqrt{K(p)}}\ge 1-\frac{2}{1+\sqrt{\frac{u_1(1-u_2)}{u_2(1-u_1)}}}=\frac{u_1-u_2}{(\sqrt{u_1(1-u_2)}+\sqrt{u_2(1-u_1)})^2}\geq u_1-u_2,
    $$
    which, combined with \eqref{eq:Ktwo}, implies \eqref{eq:Kone}. 

\smallskip
    It remains to show the other direction $R(p)\geq 1-\frac{2}{1+\sqrt{K(p)}}$. For this, it is sufficient to show the following holds.
    For any $x_1,x_2\in E$ and $y_1,y_2\in F$ such that $K:=\dfrac{p(x_1,y_1)p(x_2,y_2)}{p(x_2,y_1)p(x_1,y_2)}>1$, there is a collection of  measures $\mu_\varepsilon$ on $F$, depending on $x_1,x_2,y_1,y_2$, such that 
    \begin{equation}\label{eq:Kthr}
    \frac{1}{2}\int \left|\frac{p(x_1,y)}{\int p(x_1,z)\mu_{\varepsilon}(dz)}-\frac{p(x_2,y)}{\int p(x_2,z)\mu_{\varepsilon}(dz)}\right|\mu_{\varepsilon}(dy)
    \to1-\dfrac{2}{1+\sqrt{K}},\quad\varepsilon\to0.
    \end{equation}
    Write $p_{ij}:=p(x_i,y_j)$. Let $\mu_{\varepsilon}$ be supported on $\{y_1,y_2\}$ with $$\mu_{\varepsilon}(\{y_1\})=\varepsilon+\sqrt{p_{12}p_{22}}\ \ \mathrm{and}\ \ \mu_{\varepsilon}(\{y_2\})=\varepsilon+\sqrt{p_{11}p_{21}}.$$
    One can directly verify $\mu_\eps$ constructed as above satisfies \eqref{eq:Kthr}. We omit the details and finish the proof.
\end{proof}
\begin{proof}[Proof of Lemma \ref{lemma:1090}]
    It suffices to show that for every $x_1,x_2\in E$,
    \begin{equation}\label{eq:169}
    \frac{1}{2}\int \left |\int\left [p(x_1,y)-p(x_2,y)\right ]q(y,z)\mu(dy)\right |\nu(dz)\leq \T{p}{\mu}\T{q}{\nu}.
    \end{equation}
    Let $p_1(y):=p(x_1,y)-p(x_2,y)$ and $p_2(y):=p(x_2,y)-p(x_1,y)$. Let
    $$
    S_1 = \{y\in F:p(x_1,y)-p(x_2,y)\ge 0\}\ \ \mathrm{and}\ \ S_2 = \{y\in F:p(x_1,y)-p(x_2,y)<0\}.
    $$
    We can assume without loss of generality that $\int_{S_1} p_1(y)\mu(dy)>0$ (otherwise, \eqref{eq:169} trivially holds). Since $\int p(x,y)\mu(dy)=1$ for all $x\in E$ and $\int q(y,z)\nu(dz)=1$ for all $y\in F$, we have
    \begin{align*}
    &\dfrac{1}{2}\int\left |\int(p(x_1,y)-p(x_2,y))q(y,z)\mu(dy)\right |\nu(dz)\\
    =\ &\dfrac{1}{2}\dfrac{\int \left| \iint_{S_1\times S_2} p_1(y_1)p_2(y_2)(q(y_1,z)-q(y_2,z))\mu(dy_1)\mu(dy_2)\right |\nu(dz)}{\int_{S_1}p_1(y)\mu(dy)}\\
    \leq\ & \dfrac{1}{2}\dfrac{\iint_{S_1\times S_2}\int p_1(y_1)p_2(y_2)\left |q(y_1,z)-q(y_2,z)\right |\nu(dz)\mu(dy_1)\mu(dy_2)}{\int_{S_1}p_1(y)\mu(dy)}\\
    \le\ & \T{q}{\nu}\int_{S_1}p_1(y)\mu(dy)
    \leq \T{q}{\nu}\,\T{p}{\mu},
    \end{align*}
    which concludes \eqref{eq:169} and hence the lemma.
\end{proof}
\begin{proof}[Proof of Lemma \ref{lemma:basic}]
    It suffices to prove that for any positive finite measure $\nu$ on $G$, 
    \begin{equation}
        \T{p*q(\mu)}{\nu}\leq R(p)R(q).
    \end{equation}
    Let
    \[g(y):=\int_G q(y,z)\nu(dz),\]
    and
    \[h(x):=\int_{F} p(x,y)g(y)\mu(dy)=\iint _{F\times G}p(x,y)q(y,z)\mu(dy)\nu(dz).\]
    Let
    $$
    F_1=\{y\in F:g(y)>0\}\ \ \mathrm{and}\ \ E_1=\{x\in E:h(x)>0\}.
    $$
    Define $q_1:F_1\times G\to [0,\infty)$ and $p_1:E_1\times F_1\to[0,\infty)$ by
    \begin{equation*}
        q_1(y,z):={q(y,z)}{g(y)}^{-1}\ \ \mathrm{and}\ \ p_1(x,y):={p(x,y)g(y)}{h(x)}^{-1}.
    \end{equation*}
    Then, for any $x\in E_1$, we have $p_1*q_1(\mu)(x,\cdot)= p*q(\mu)(x,\cdot)h(x)^{-1}$ a.s.\ under $\nu$, which implies
    \[
    \T{p*q(\mu)}{\nu}=\T{p_1*q_1(\mu)}{\nu}.
    \]
    By the definition of $K$, it is easy to see that $K(p_1)\leq K(p), K(q_1)\leq K(q)$, which implies $R(p_1)\leq R(p),R(q_1)\leq R(q)$ by Lemma \ref{lemma:RK}. Since $\int_{G} q_1(y,z)\nu(dz)=1$ and $\int_{F_1}p_1(x,y)\mu(dy)=1$, by Lemma~\ref{lemma:1090}, we have
\begin{equation}\label{eq:5.3ineq}
        \T{p_1*q_1(\mu)}{\nu}
        \leq \T{p_1}{\mu}\T{q_1}{\nu}\leq R(p_1)R(q_1)\leq R(p)R(q),
    \end{equation}
    which concludes the proof. 
\end{proof}
\begin{proof}[Proof of Lemma \ref{lemma:mean}]
    For any $x_1,x_2\in E$, $z_1,z_2\in G$, we have
    \begin{align*}
        &r(x_1,z_1)r(x_2,z_2)
        = \int_F\int _Fq(y_1,z_1)q(y_2,z_2)\mu(x_1,dy_1)\mu(x_2,dy_2)\\
        \leq\ &K(q)\int_F\int _Fq(y_1,z_2)q(y_2,z_1)\mu(x_1,dy_1)\mu(x_2,dy_2)
        =K(q)r(x_2,z_1)r(x_1,z_2).
    \end{align*}
    This implies $K(r)\le K(q)$. By Lemma~\ref{lemma:RK}, we also have $R(r)\le R(q)$.
\end{proof}
\begin{proof}[Proof of Lemma \ref{lemma:tricky}]
    Without loss of generality, we assume $\sup_{x\in E} r(x)=1$ and $K(p)<\infty$.
    Let $$g(y):=\int _E r(z)\nu(y,dz).$$
    By \eqref{eq:486}, we have for all $x\in E$,  
    $$
    \int_F p(x,y)g(y)\mu(dy)\leq \int_F p(x,y)\nu(y,E)\mu(dy)\leq a.
    $$
    For $\varepsilon >0$, pick $x_0$ such that $r(x_0)>1-\varepsilon$. By  \eqref{eq:482}, we have
    \begin{equation*}
        \int _F p(x_0,y)s(y)\mu(dy)=r(x_0)-\int _F p(x_0,y)g(y)\mu(dy)> (1-\varepsilon)-a.
    \end{equation*}  
    It follows that for any $x\in E$, we have
    \begin{align*}
        &(1-a-\varepsilon)\int _F p(x,y)g(y)\mu(dy)\le\ \int _F p(x,y)g(y)\mu(dy)\int _F p(x_0,y)s(y)\mu(dy)\\
        \leq\ & K(p)\int _F p(x_0,y)g(y)\mu(dy)\int _F p(x,y)s(y)\mu(dy)
        \ \leq\  aK(p) \int _F p(x,y)s(y)\mu(dy).
    \end{align*}
    Thus, for all $x\in E$,
    \begin{equation*}
        (1-a-\varepsilon)r(x)\leq (aK(p)+1-a-\varepsilon)\int p(x,y)s(y)\mu(dy).
    \end{equation*}
    We conclude the proof by letting $\varepsilon\to 0$.
\end{proof}

\section{Hitting probability estimates}\label{sec:hpe}
In this part, we first derive some hitting estimates for cylinders, spheres and cones, and then use them to prove Lemma~\ref{Lemma:3.6}. 

\begin{lemma}\label{lem:503}
    \textnormal{(i) (Cylinder vs Cylinder)}
    Let $d<1/2$. Consider two concentric cylinders
    $$
    D=\{(z_1,z_2,z_3):0\leq z_3\leq \pi, z_1^2+z_2^2\leq 1\},
    $$
    and
    $$
    D_1=\{(z_1,z_2,z_3):0\leq z_3\leq \pi, z_1^2+z_2^2\leq d^2\}.
    $$
    There exists an absolute constant $c>0$, such that for any $z\in B((0,0,\pi/2),1/2)$, we have
    \begin{equation}\label{eq:745}
        \Pb_z\Big(\tau_{D_1}\le \tau_{\partial D}\Big)
        \geq \frac{c}{\log(1/d)}.
    \end{equation}    
    \textnormal{(ii) (Sphere vs Cylinder)}
    Let $0\leq a\leq 1/2$, $d<1/16$ and $ D_2=\{(z_1,z_2,z_3):z_1^2+(z_2-a)^2<d^2\}$. Recall $\tau=\tau_{\partial B(0,1)}$. There exists an absolute constant $c'>0$ such that
    \begin{equation}\label{eq:753}
        \Pb_0\Big(\tau_{D_2}\le \tau\Big)
        \geq \frac{c'}{\log(1/d)}.
    \end{equation}
\end{lemma}

\begin{proof}
    (i)
    If $z\in D_1$, we have \eqref{eq:745} directly by choosing $c<\log 2$. 
    Now, assume $z\notin D_1$ and write $g(z):=\Pb_z(\tau_{D_1}\le \tau_{\partial D})$. 
    Since $g(z)$ is harmonic in $ D\setminus D_1$, we can bound $g$ by two other harmonic functions by the maximum principle. Define $$g_1(z):=\log(1/\sqrt{z_1^2+z_2^2})/\log(1/d)\ \ \mathrm{and}\ \ g_2(z):=f(\sqrt{z_1^2+z_2^2})\sin(z_3)
    ,$$
    where $f: [d,1] \mapsto [0,\infty)$ is a solution to the ordinary differential equation
    \begin{equation}\label{eq:578}
        (rf'(r))'=rf(r),
    \end{equation}
    with boundary conditions $f(d)=1$ and $f(1)=0$.
    
    By the existence and uniqueness theorem in the ODE theory 
    (see e.g.\ \cite[p.107]{MR2242407}), 
    there exists a solution $f_0$ to \eqref{eq:578} with (different) initial values $f_0(1)=0$ and $f_0'(1)=-1$. We now claim that $f_0'<0$ on $[d,1]$. If it is not true, then we can find $r_0:=\sup\{r\in(d,1):f_0'(r)=0\}$. From $f'(r)<0$ on $(r_0,1]$ and $f(1)=0$, we deduce that $f_0(r_0)>0$ and $f_0''(r_0)\le 0$, which contradicts \eqref{eq:578} at $r_0$. Thus, $f_0'<0$ on $[d,1]$ and $f_0>0$ on $[d,1)$. We then define $f(r):=f_0(r)/f_0(d)$, which is a solution to \eqref{eq:578} as required.
    
    One can directly verify that both $g_1$ and $g_2$ are harmonic inside $D\setminus D_1$ and $0\leq g_2\leq g\leq g_1$ on $\partial(D\setminus D_1)$. By the maximum principle, we get $0\leq g_2(z)\leq g(z)\leq g_1(z)$ for all $z\in D\setminus D_1$.
    From $g_2(z)\leq g_1(z)$ and letting $z_3=\pi/2$, we have 
    \begin{equation}\label{eq:594}
        f(r)\leq \frac{\log(1/r)}{\log(1/d)},\ r\in [d,1].
    \end{equation}
    By \eqref{eq:578}, we have $(rf'(r))'=rf(r)\geq 0$, which implies
    \begin{equation}\label{eq:532}
    rf'(r)\geq d\,f'(d),\ r\in [d,1].
    \end{equation}
    Combining \eqref{eq:594} and \eqref{eq:532}, we get 
    \begin{equation}\label{eq:B1}\frac{\log(1/r)}{\log(1/d)}\geq f(r)=f(d)+\int_d^r f'(s)ds\geq 1+\int_d^r \frac{df'(d)}{s}ds.\end{equation}
    Letting $r=1$ in \eqref{eq:B1}, we have $df'(d)\leq -1/\log(1/d)$. Combined with (\ref{eq:578}) and (\ref{eq:594}), this implies
    \begin{align*}rf'(r)=df'(d)+\int_d^r sf(s)ds&\leq df'(d)+\int_d^r \frac{s\log(1/s)}{\log(1/d)}ds\\
    &\leq df'(d)+\int_0^1 \frac{s\log(1/s)}{\log(1/d)}ds\leq-\frac{3}{4\log(1/d)}.\end{align*}
    Hence, for all $r\in[d,1]$,
    $$f(r)=-\int_r^1 f'(s)ds\geq \frac{3}{4\log(1/d)}\int_r^1\frac{1}{s} ds= \frac{3\log(1/r)}{4\log(1/d)}.$$
    Finally, we get the lower bound of $g(z)$ when $z\in \B((0,0,\pi/2),1/2)\setminus  D_1$:
    $$
    g(z)\geq g_2(z)=f(\sqrt{z_1^2+z_2^2})\sin(z_3)\geq\frac{3\log 2}{4\log(1/d)}\sin(\frac{\pi}{2}-\frac{1}{2}),
    $$ 
    which completes the proof by letting $c=3(\log 2)\cos(1/2)/4<\log2$. 

    \smallskip
    (ii)
    Let $D_3:=\{(z_1,z_2,z_3):-\pi/16\leq z_3\leq\pi/16, z_1^2+(z_2-a)^2\leq 1/64\}$ be a cylinder.
    It is easy to show that there exists a constant $C>0$ such that for all $0\leq a\leq 1/2$,
    \begin{equation}\label{eq:826}
    \Pb_0\Big(  \tau_{\B((0,a,0),1/16)}<\tau\Big)>C.
    \end{equation}
    By the strong Markov property, \eqref{eq:745} and \eqref{eq:826}, we have 
    \begin{align*}
        \Pb_0\Big(\tau_{D_2}\le \tau\Big)
        \geq\  \Pb_0\Big(  \tau_{\B((0,a,0),1/16)}\le \tau\Big)
        \inf_{z\in  \B((0,a,0),1/16)}\Pb_z\Big(\tau_{D_2}\le \tau_{\partial  D_3}\Big)
        \geq\  \frac{c\,C}{\log(1/(8d))},
    \end{align*}
    which concludes the proof.
\end{proof}

Next, we estimate the probability that a Brownian motion visits a specific cone before it escapes from a local ball.
Recall the cones $S_i$ and $T_{i,j}$ in Definition~\ref{def:cone}, and the local ball $\Ball(z)$ in \eqref{eq:rz}. 
\begin{lemma}[Sphere vs Cone]\label{lem:592}
    \textnormal{(i) (Lower bound)} There exist $c=c(u_1)>0$ and $n_1(u_1)>0$, such that for all $z\in \extendSquare_i$ and $n>n_1(u_1)$, we have
    \begin{equation}\label{eq:862}
        \Pb_z\Big(\tau_{\Tube_{i,j}}<\tau_{\partial \Ball(z)}\Big)>\frac{c}{n}.
    \end{equation}
    \textnormal{(ii) (Upper bound)} There exist $c'=c'(u_1)>0$ and $n_2(u_1)>0$, such that for $z\in\extendSquare_i$, $z'\in\Ball(z)$ and $n>n_2(u_1)$, we have
    \begin{equation}\label{eq:866}
        \Pb_{z'}\Big(\tau_{\Tube_{i,j}}<\tau_{\partial \Ball(z)}\Big)<\frac{c'(\log(r(z')/|z'-|z'|v_{i,j}|)+1)}{n}.
    \end{equation}
\end{lemma}
\begin{proof}
    (i) By \eqref{eq:rz} and \eqref{eq:440}, we have
    \begin{equation*}
        \Cylinder(v_{i,j},u_1^{-n}r(z))\cap \Ball(z)\subseteq \Tube_{i,j}.
    \end{equation*}
    By \eqref{6-1}, the distance between $z$ and the line connecting $0$ and $v_{i,j}$ is less than $r(z)/2$. Picking $n_1$ large such that $u_1^{-n_1}<1/16$, the result follows immediately from \eqref{eq:753}.
    
    (ii) Denote $\mathscr D:=\Cylinder(v_{i,j},2u_1^{-n}|z|)$ and $d:=\inf_{x\in\Rb}|z'-x\cdot v_{i,j}|$.
    By (\ref{eq:440}), we have $\Tube_{i,j}\cap \Ball(z)\subseteq\mathscr D$, and thus
    \begin{equation*}
        \Pb_{z'}\Big(\tau_{\Tube_{i,j}}<\tau_{\partial \Ball(z)}\Big)
        \leq\Pb_{z'}\Big(\tau_{\mathscr D}<\tau_{\partial \Ball(z)}\Big)
        \leq\Pb_{z'}\Big(\tau_{\mathscr D}<\tau_{\partial \Cylinder(v_{i,j},2r(z))}\Big)
        \leq\frac{\log(2r(z)/d)}{\log(2r(z)/(2u_1^{-n}|z|))},
    \end{equation*}
    where we used a standard hitting probability estimate for $2D$ Brownian motion in the last inequality. Noting that $d\geq|z'-|z'|v_{i,j}|/\sqrt 2$ by $z'\cdot v_{i,j}>0$, and $r(z')\ge r(z)/2$ by \eqref{eq:rz}, we conclude the result.
\end{proof}

We finally complete the proof of  Lemma~\ref{Lemma:3.6} by using Lemma~\ref{lem:592}.

\begin{proof}[Proof of Lemma~\ref{Lemma:3.6}]
    (i) Assume $z\in \partial V_{i,j}$ for some $j\in J$.
    By the union bound, we have
    \begin{equation*}
        \Pb_{z}\Big(\tau_{\TubeComplement_{i', j'}}=\tau_{\extendSubSet_{i,J}} ,\tau_{\TubeComplement_{i', j'}}<\tau_{\partial \Ball(z)}\Big)
        \leq \Pb_{z}\Big(\tau_{\extendSubSet_{i,J}}<\tau_{\partial\Ball(z)}\Big)
        \leq \sum_{j_l\in J}\Pb_{z}\Big(\tau_{\extendSubSet_{i,j_l}}<\tau_{\partial\Ball(z)}\Big).
    \end{equation*}
    By Definition \ref{def:cone}, for $j_l\in J$, we have
    \begin{equation}\label{eq:zz'}
        \big|z-v_{i,j_l}|z|\big|/r(z)>c_1\,\max\{|j-j_l|, 1\}/(mr_0).
    \end{equation} 
    Then, using \eqref{eq:866}, for all $n>n_2$ and $l=1,2,\ldots,|J|$,
    $$\Pb_{z}\Big(\tau_{\Tube_{i,j_l}}<     \tau_{\partial \Ball(z)}\Big)<\frac{c'\log(mr_0/\max\{c_1|j-j_l|, c_1\})+c'}{n}.$$
Rearranging the sum and using the above estimate, we obtain that
    \begin{align*}
        &\  \sum_{j_l\in J}\Pb_{z}\Big(\tau_{\extendSubSet_{i,j_l}}<\tau_{\partial\Ball(z)}\Big)
= \sum_{l=1}^{|J|}\Pb_{z}\Big(\tau_{\extendSubSet_{i,j_l}}<\tau_{\partial\Ball(z)}\Big)
        \le\sum_{l=1}^{|J|} \frac{c'\log(r_0m/\max\{c_1|j-j_l|, c_1\})+c'}{n}\\
        \le&\ \sum_{l=1}^{|J|} \frac{c'\log(3r_0m/(c_1l))+c'}{n}
       \le\dfrac{c'|J|\log(m/{|J|})+c'|J|(2+\log 3+\log(r_0/c_1))}{n}.
    \end{align*}
    Taking $c_2=\max\{c',c'(2+\log 3+\log(r_0/c_1))\}$, we conclude \eqref{eq:924}, and \eqref{eq:930} is obtained similarly.
    
    (ii) Let $n_3(u_1)>\max\{n_1(u_1),n_2(u_1)\}$.
    Pick $j$ and $j_l$ such that $j\neq j'$ and $j_l\neq j$. We first bound from above the probability that a Brownian motion started from $z\in\partial\extendSquare_i$ visits $\Tube_{i,j_l}$ and $\Tube_{i,j}$ successively before exiting $\Ball(z)$.
    By Definition \ref{def:cone}, $\dist(v_{i,j_l},\partial S_i)\ge c_1\, n^{-1/2}$, therefore, 
    $$\big|z-v_{i,j_l}|z|\big|/r(z)\geq c_1/r_0,$$
    and by \eqref{eq:866}, we have
    $$
    \Pb_{z}\Big(\tau_{\Tube_{i,j_l}}<\tau_{\partial \Ball(z)}\Big)\leq \frac{c'(\log(r_0/c_1)+1)}{n}.
    $$    
    Moreover, by \eqref{eq:zz'},
    for $j_l\neq j$ and $z'\in\Tube_{i,j_l}$, we have
    $$\big|z'-v_{i,j}|z'|\big|/r(z')>c_1|j-j_l|/(r_0m).$$
    Using \eqref{eq:866} again, we obtain 
    $$ \sup_{z'\in\Tube_{i,j_l}\cap\Ball(z)}\Pb_{z'}\Big(\tau_{\Tube_{i,j}}<\tau_{\partial \Ball(z)}\Big)\leq \frac{c'(\log(r_0m/c_1|j-j_l|)+1)}{n}.
    $$
    By the strong Markov property,
    \begin{align*}
    \Pb_{z}\Big(\tau_{\Tube_{i,j_l}}<\tau_{\Tube_{i,j}}<\tau_{\partial\Ball(z)}\Big)
    \leq\ \Pb_{z}\Big(\tau_{\Tube_{i,j_l}}<\tau_{\partial\Ball(z)}\Big) \sup_{z'\in\Tube_{i,j_l}\cap\Ball(z)}\Pb_{z'}\Big(\tau_{\Tube_{i,j}}<\tau_{\partial\Ball(z)}\Big).
    \end{align*}     
    By \eqref{6-1}, $\Ball(z)\cap \Tube_{k}=\varnothing$ for all $k\neq i$. Thus, by the union bound, for some constant $\overline c=\overline c(u_1)$,
    \begin{align*}
    \Pb_{z}\Big(\tau_{\TubeComplement_{i',j'}}<\tau_{\Tube_{i,j}}<\tau_{\partial\Ball(z)}\Big)\le\ &\sum_{l=1}^{m-1}\Pb_{z}\Big(\tau_{\Tube_{i,j_l}}<\tau_{\Tube_{i,j}}<\tau_{\partial\Ball(z)}\Big)\notag\\
    \leq\ & \sum_{l=1}^{m-1} \frac{c'(\log(r_0/c_1)+1)}{n}\frac{c'(\log(r_0m/c_1|j-j_l|)+1)}{n}\label{eq:1040}\\
    \leq\ & \sum_{l=1}^{m-1} \frac{c'(\log(r_0/c_1)+1)}{n}\frac{c'(\log(r_0m/c_1l/3)+1)}{n}\le \dfrac{\overline c\,m}{n^2}.\notag
    \end{align*} 
    The above combined with \eqref{eq:862} gives that
    \begin{align*}
        \Pb_{z}\Big(\tau_{\TubeComplement_{i',j'}}=\tau_{\Tube_{i,j}},\tau_{\Tube_{i,j}}<\tau_{\partial\Ball(z)}\Big)
        =\ \Pb_{z}\Big(\tau_{\Tube_{i,j}}<\tau_{\partial\Ball(z)}\Big)-\Pb_{z}\Big(\tau_{\TubeComplement_{i',j'}}<\tau_{\Tube_{i,j}}<\tau_{\partial\Ball(z)}\Big)
        \geq\frac{c}{n}-\dfrac{\overline c\,m}{n^2}.
    \end{align*}
    It follows that
    \begin{equation*}
        \Pb_{z}\Big(\tau_{\TubeComplement_{i', j'}}=\tau_{\extendSet_i} ,\tau_{\TubeComplement_{i', j'}}<\tau_{\partial \Ball(z)}\Big) = \sum_{j=1}^{m}\Pb_{z}\Big(\tau_{\TubeComplement_{i', j'}}=\tau_{\Tube_{i,j}} ,\tau_{\TubeComplement_{i', j'}}<\tau_{\partial \Ball(z)}\Big)\geq (m-1)\left(\frac{c}{n}-\frac{\overline c\,m}{n^2}\right).
    \end{equation*}
    This completes the proof of this lemma by letting $$c_3=\min\left(1,\frac{c}{2\overline c},\frac{1}{2c_2}\right)\ \ \mathrm{and}\ \ c_4=\frac{c_3c}{4},$$
    and choosing $n_3(u_1)$ large enough.
\end{proof}
\section{Downward deviation of cover times in a general setting}
In this part, we state and prove a downward deviation of  the cover time for a general, possibly non-Markovian, stochastic process.
\begin{lemma}\label{Lemma:3.8}
    Let $m,n,k$ be positive integers. 
    For each $s=1,\ldots,k$, consider a probability space $(\Omega_s,\mathcal F^s,\mathbb P_s)$ and denote their product probability space by $(\Omega, \mathcal F,\mathbb P)$. Let $G=\bigcup_{i=1}^n G^{(i)}$ where $G^{(i)}$'s are disjoint sets with $|G^{(i)}|=m$. For each $s$, let $X_l^s:\Omega_s\to G$ with $l=0,1,\ldots$ be random variables adapted to $\{\mathcal F_l^s\}_{l\geq 0}$, where $\{\mathcal F_l^s\}_{l\geq 0}$ is a sequence of increasing $\sigma$-fields.
    Suppose that for any $H^{(i)}\subseteq G^{(i)}$, any integer $1\leq s\leq k$ and any integer $l\geq 1$, we have
    \begin{equation}
        \mathbb P_s\left(X_l^s\in H^{(i)} \con  \mathcal F_{l-1}^s\right)<F(|H^{(i)}|/{m})\,\mathbb P_s\left(X_l^s\in G^{(i)}\con  \mathcal F_{l-1}^s\right),
    \end{equation}
    where $F:(0,1]\to\mathbb R^+$ is a continuous and increasing function such that 
    \begin{equation}\label{eq:2668}
        \int_0^1 \frac{1}{F(x)}dx=\infty.
    \end{equation}
    Then, for any $K>0$, there exist constants $q=q(k,K,F)<1$ and $m_0=m_0(k,K,F)>0$, such that for any $m>m_0$ and any $n\geq1$, we have
    \begin{equation}
    \mathbb P\Big(G\subseteq\{X_l^s:l=0,\ldots,[Kmn],s=1,\ldots,k\}\Big)<q^{mn}.
    \end{equation}
\end{lemma}

\begin{proof}
    $Step$ 1. For $l\geq 0$ and $1\leq s\leq k$, let $Y_{lk+s}:=X_l^s$. Then, for any $t\geq k$ and any $1\le i\le n$, we have
    \begin{equation}\label{eq:835}
        \mathbb P\left(Y_{t+1}\in H^{(i)}\con  Y_1,\ldots,Y_{t}\right)\le F(|H^{(i)}|/m)\,\mathbb P\left(Y_{t+1}\in G^{(i)}\con  Y_1,\ldots,Y_{t}\right).
    \end{equation}
    Let $Z_{i,j}\sim G(p_j)$ $(i=1,\ldots,n;j=1,\ldots,m)$ be independent geometric random variables with parameter $p_j$, where $p_j:= \min\{F(j/m),1\}$. Set $p_0:=0$.
    For integers $k_i\le m$ and $r\ge 0$, let 
    $$L_r(k_1,\ldots,k_n):=\mathbb P\Big(\sum_{i=1}^n\sum_{j=1}^{k_i}Z_{i,j}\leq r\Big).$$
    Let $\mathcal{Y}_{t}:=\{Y_1,\ldots,Y_t\}$ and $U_t^{(i)}:= G^{(i)}\setminus\mathcal{Y}_t$. For simplicity, we write 
    \[\mathcal{U}_t=(|U_t^{(1)}|,\ldots,|U_t^{(n)}|)\] 
    and 
    \[\mathcal{U}_t^{i}=(|U_t^{(1)}|,\ldots,|U_t^{(i)}|-1,\ldots,|U_t^{(n)}|).\]
    Fix $r\ge k$. We now claim that for $k\leq t\leq r$,
    \begin{equation}\label{eq:798}
        \mathbb P\left(G\subseteq\mathcal{Y}_r\con  Y_1,\ldots,Y_t\right)\leq L_{r-t}( \mathcal{U}_t).
    \end{equation}
    We will prove \eqref{eq:798} by induction. Note that \eqref{eq:798} automatically holds when $t=r$. Suppose \eqref{eq:798} holds for $t+1$ ($k\leq t\leq r-1$), we will show it also holds for $t$. By \eqref{eq:835}, inductive hypothesis and the monotonicity of $L_{r}(\cdot)$, we have
    \begin{align*}
        &\,\mathbb P\Big(G\subseteq\mathcal{Y}_r\con  Y_1,\ldots,Y_t\Big)\\
        =&\,\sum_{i=1}^n \mathbb P\Big(G\subseteq\mathcal{Y}_r,Y_{t+1}\in G^{(i)}\setminus U_t^{(i)}\con  Y_1,\ldots,Y_t\Big)
        +\sum_{i=1}^n\mathbb P\Big(G\subseteq\mathcal{Y}_r,Y_{t+1}\in U_t^{(i)}\con  Y_1,\ldots,Y_t\Big)\\
        \leq&\,\sum_{i=1}^n L_{r-t-1}(\mathcal{U}_t)\mathbb P\Big(Y_{t+1}\in G^{(i)}\setminus U_t^{(i)}\con  Y_1,\ldots,Y_t\Big)
        +\sum_{i=1}^n L_{r-t-1}(\mathcal{U}_t^{i})\mathbb P\Big(Y_{t+1}\in  U_t^{(i)}\con  Y_1,\ldots,Y_t\Big)\\
        \leq&\, \sum_{i=1}^n L_{r-t-1}(\mathcal{U}_t)\mathbb P\Big(Y_{t+1}\in G^{(i)}\con  Y_1,\ldots,Y_t\Big)\mathbb P\Big(Y_{t+1}\in G^{(i)}\setminus U_t^{(i)}\con  Y_{t+1}\in G^{(i)},Y_1,\ldots,Y_t\Big)    \\
        &+\, \sum_{i=1}^n L_{r-t-1}(\mathcal{U}_t^{i})\mathbb P\Big(Y_{t+1}\in G^{(i)}\con  Y_1,\ldots,Y_t\Big)\mathbb P\Big(Y_{t+1}\in U_t^{(i)}\con  Y_{t+1}\in G^{(i)},Y_1,\ldots,Y_t\Big)\\
        \leq &\,\sum_{i=1}^n\mathbb P\Big(Y_{t+1}\in G^{(i)}\con  Y_1,\ldots,Y_t\Big)\cdot
        \Big[ L_{r-t-1}(\mathcal{U}_t)(1-p_{|U_t^{(i)}|})
        + L_{r-t-1}(\mathcal{U}_t^{i})p_{|U_t^{(i)}|}\Big]\\
        \leq &\,\max_{1\leq i\leq n}\left\{L_{r-t-1}(\mathcal{U}_t)(1-p_{|U_t^{(i)}|})+L_{r-t-1}(\mathcal{U}_t^{i})p_{|U_t^{(i)}|})\right\}.
    \end{align*}
    Therefore, it suffices to show that for any $r\ge 0$, $1\le i\le n$ and $0\le k_i\le m$, we have
    \begin{equation}\label{eq:C1}
    L_{r}(k_1,\ldots,k_n)(1-p_{k_i})+L_{r}(k_1,\ldots,k_i-1,\ldots,k_n)p_{k_i}
    \leq L_{r+1}(k_1,\ldots,k_n).
    \end{equation}
    If $k_i=0$, the above inequality is trivial from the definition of $L$. If $k_i\ge1$, let $Z:=(\sum_{i'=1}^n\sum_{j=1}^{k_{i'}}Z_{i',j})-Z_{i,k_i}$. Since $Z_{i,k_i}$ is independent of $Z$, we have
    \begin{equation}\label{eq:actualequality}
        \mathbb P(Z+Z_{i,k_i}\leq r)(1-p_{k_i})+\mathbb P(Z\leq r)p_{k_i}\leq\mathbb P(Z+Z_{i,k_i}\leq r+1),
    \end{equation}
    which implies \eqref{eq:C1}, and thus completes the proof of the claim.
    
    $Step$ 2. We will prove that there exist constants $q=q(k,K,F)<1$ and $m_0(k,K,F)$, such that for any $m>m_0$, we have
    \begin{equation}\label{eq:3.8}
        \mathbb P\Big(\sum_{i=1}^n\sum_{j=1}^{m-k} Z_{i,j}\le Kmn\Big)<q^{mn}.
    \end{equation}
    
    For any $u>0$, we have 
    $$\mathbb E[e^{-uZ_{i,j}}]=\dfrac{p_j}{e^u-1+p_j}\le\dfrac{F(j/m)}{e^u-1+F(j/m)}.$$
    By Chebyshev inequality, for $m>2k$ and any $v>0$,
    \begin{align*}
    \mathbb P\Big(\sum_{i=1}^n\sum_{j=1}^{m-k} Z_{i,j}\le Kmn\Big)
    \leq\ e^{vKmn}\prod_{i=1}^n\prod_{j=1}^{m-k} \mathbb E[e^{-vZ_{i,j}}]
    \leq\  e^{vKmn}\left(\prod_{j=1}^{m-k}\dfrac{F(j/m)}{e^v-1+F(j/m)}\right)^n.
    \end{align*}

    By \eqref{eq:2668}, we can choose $v=v(F,K)>0$ such that $\int_0^1 (e^v-1+F(x))^{-1}dx>3K$. 
    Then, we can choose $m_1(k,K,F)$, such that for any $m>m_1(k,K,F)$, we have 
    \begin{equation*}
    \frac 1m\sum_{j=1}^{m-k} \dfrac{1}{e^v-1+F(j/m)}>2K,
    \end{equation*}
    which implies
    \begin{align*}
        \prod_{j=1}^{m-k}\dfrac{F(j/m)}{e^v-1+F(j/m)} 
        =\ &\exp\left( -\sum_{j=1}^{m-k}\int_0^v \dfrac{e^u }{e^u-1+F(j/m)}du \right)\\
        \leq\ & \exp\left( -\int_0^v \sum_{j=1}^{m-k} \frac{1}{e^v-1+F(j/m)}du\right)
        \leq e^{-2vKm}.
    \end{align*}
    Thus, we have
    \begin{equation}\label{eq:1301}
        \mathbb P\Big(\sum_{i=1}^n\sum_{j=1}^{m-k} Z_{i,j}\le Kmn\Big) \leq e^{-vKmn}.
    \end{equation}

    Step 3.
    Combining \eqref{eq:798} with \eqref{eq:3.8}, there are $m_0=m_1(k,kK,F)$ and $q=e^{-vkK}$ such that for all $m> m_0$,
    \begin{align*}
        \mathbb P\Big(G\subseteq \{Y_1,\ldots,Y_{k[Kmn]}\}\con  Y_1,\ldots,Y_{k}\Big)
        &\le L_{k[Kmn]-k}\big(|U_{k}^{(1)}|,\ldots,|U_{k}^{(n)}|\big)\\
        &=\mathbb P\Big(\sum_{i=1}^n\sum_{j=1}^{|U_{k}^{(i)}|} Z_{i,j}\le k[Kmn]-k\Big)\\
        &\le\mathbb P\Big(\sum_{i=1}^n\sum_{j=1}^{m-k} Z_{i,j}\le kKmn\Big)
        <q^{mn}.
    \end{align*}
    This concludes the proof.
\end{proof}

\bibliographystyle{abbrv}
\bibliography{analyticity}

\begin{thebibliography}{10}

\bibitem{MR1800526}
H.~Aikawa.
\newblock Boundary {H}arnack principle and {M}artin boundary for a uniform
  domain.
\newblock {\em J. Math. Soc. Japan}, 53(1):119--145, 2001.

\bibitem{MR2464701}
H.~Aikawa.
\newblock Equivalence between the boundary {H}arnack principle and the
  {C}arleson estimate.
\newblock {\em Math. Scand.}, 103(1):61--76, 2008.

\bibitem{MR2204573}
H.~Aikawa, K.~Hirata, and T.~Lundh.
\newblock Martin boundary points of a {J}ohn domain and unions of convex sets.
\newblock {\em J. Math. Soc. Japan}, 58(1):247--274, 2006.

\bibitem{An78}
A.~Ancona.
\newblock Principe de {H}arnack \`a{} la fronti\`ere et th\'eor\`eme de {F}atou
  pour un op\'erateur elliptique dans un domaine lipschitzien.
\newblock {\em Ann. Inst. Fourier (Grenoble)}, 28(4):169--213, x, 1978.

\bibitem{MR2242407}
V.~I. Arnold.
\newblock {\em Ordinary differential equations}.
\newblock Universitext. Springer-Verlag, Berlin, 2006.
\newblock Translated from the Russian by Roger Cooke, Second printing of the
  1992 edition.

\bibitem{MR1131398}
R.~Ba\~nuelos, R.~F. Bass, and K.~Burdzy.
\newblock H\"older domains and the boundary {H}arnack principle.
\newblock {\em Duke Math. J.}, 64(1):195--200, 1991.

\bibitem{MR4482110}
M.~T. Barlow and D.~Karli.
\newblock Some boundary {H}arnack principles with uniform constants.
\newblock {\em Potential Anal.}, 57(3):433--446, 2022.

\bibitem{MR1042338}
R.~F. Bass and K.~Burdzy.
\newblock A probabilistic proof of the boundary {H}arnack principle.
\newblock In {\em Seminar on {S}tochastic {P}rocesses, 1989 ({S}an {D}iego,
  {CA}, 1989)}, volume~18 of {\em Progr. Probab.}, pages 1--16. Birkh\"auser
  Boston, Boston, MA, 1990.

\bibitem{MR1127476}
R.~F. Bass and K.~Burdzy.
\newblock A boundary {H}arnack principle in twisted {H}\"older domains.
\newblock {\em Ann. of Math. (2)}, 134(2):253--276, 1991.

\bibitem{bass1992lifetimes}
R.~F. Bass and K.~Burdzy.
\newblock Lifetimes of conditioned diffusions.
\newblock {\em Probab. Theory Related Fields}, 91:405--443, 1992.

\bibitem{MR3126579}
F.~Comets, C.~Gallesco, S.~Popov, and M.~Vachkovskaia.
\newblock On large deviations for the cover time of two-dimensional torus.
\newblock {\em Electron. J. Probab.}, 18(96):1--18, 2013.

\bibitem{Da77}
B.~E.~J. Dahlberg.
\newblock Estimates of harmonic measure.
\newblock {\em Arch. Rational Mech. Anal.}, 65(3):275--288, 1977.

\bibitem{DS11}
M.~Damron and A.~Sapozhnikov.
\newblock Outlets of 2{D} invasion percolation and multiple-armed incipient
  infinite clusters.
\newblock {\em Probab. Theory Related Fields}, 150(1-2):257--294, 2011.

\bibitem{MR4093736}
D.~De~Silva and O.~Savin.
\newblock A short proof of boundary {H}arnack principle.
\newblock {\em J. Differential Equations}, 269(3):2419--2429, 2020.

\bibitem{MR2123929}
A.~Dembo, Y.~Peres, J.~Rosen, and O.~Zeitouni.
\newblock Cover times for {B}rownian motion and random walks in two dimensions.
\newblock {\em Ann. of Math. (2)}, 160(2):433--464, 2004.

\bibitem{du2022sharp}
H.~Du, Y.~Gao, X.~Li, and Z.~Zhuang.
\newblock Sharp asymptotics for arm probabilities in critical planar
  percolation.
\newblock {\em Comm. Math. Phys.}, 405(182):1--51, 2024.

\bibitem{Du98}
B.~{Duplantier}.
\newblock {Random Walks and Quantum Gravity in Two Dimensions}.
\newblock {\em Phys. Rev. Lett.}, 81:5489--5492, 1998.

\bibitem{duplantier1999two}
B.~Duplantier.
\newblock Two-dimensional copolymers and exact conformal multifractality.
\newblock {\em Phys. Rev. Lett.}, 82(5):880, 1999.

\bibitem{DK88}
B.~Duplantier and K.-H. Kwon.
\newblock Conformal invariance and intersections of random walks.
\newblock {\em Phys. Rev. Lett.}, 61:2514--2517, 1988.

\bibitem{MR1658624}
F.~Ferrari.
\newblock On boundary behavior of harmonic functions in {H}\"older domains.
\newblock {\em J. Fourier Anal. Appl.}, 4(4-5):447--461, 1998.

\bibitem{GLQ2022}
Y.~Gao, X.~Li, and W.~Qian.
\newblock Multiple points on the boundaries of {B}rownian loop-soup clusters.
\newblock {\em arXiv preprint arXiv:2205.11468}, 2022.

\bibitem{GNQ2024a}
Y.~Gao, P.~Nolin, and W.~Qian.
\newblock Percolation of discrete {GFF} in dimension two {I}. {A}rm events in
  the random walk loop soup.
\newblock {\em arXiv preprint arXiv:2409.16230}, 2024.

\bibitem{GPS2013}
C.~Garban, G.~Pete, and O.~Schramm.
\newblock Pivotal, cluster, and interface measures for critical planar
  percolation.
\newblock {\em J. Amer. Math. Soc.}, 26(4):939--1024, 2013.

\bibitem{MR676988}
D.~S. Jerison and C.~E. Kenig.
\newblock Boundary behavior of harmonic functions in nontangentially accessible
  domains.
\newblock {\em Adv. in Math.}, 46(1):80--147, 1982.

\bibitem{MR716504}
D.~S. Jerison and C.~E. Kenig.
\newblock Boundary value problems on {L}ipschitz domains.
\newblock In {\em Studies in partial differential equations}, volume~23 of {\em
  MAA Stud. Math.}, pages 1--68. Math. Assoc. America, Washington, DC, 1982.

\bibitem{Ke1987a}
H.~Kesten.
\newblock Scaling relations for {$2$}{D}-percolation.
\newblock {\em Comm. Math. Phys.}, 109(1):109--156, 1987.

\bibitem{MR2350070}
G.~Kozma.
\newblock The scaling limit of loop-erased random walk in three dimensions.
\newblock {\em Acta Math.}, 199(1):29--152, 2007.

\bibitem{MR998666}
G.~F. Lawler.
\newblock Intersections of random walks with random sets.
\newblock {\em Israel J. Math.}, 65(2):113--132, 1989.

\bibitem{RWcuttimes}
G.~F. Lawler.
\newblock Cut times for simple random walk.
\newblock {\em Electron.~J.~Probab.}, 1(13):1--24, 1996.

\bibitem{Law96a}
G.~F. Lawler.
\newblock The dimension of the frontier of planar {B}rownian motion.
\newblock {\em Electron. Comm. Probab.}, 1(5):29--47, 1996.

\bibitem{Law96}
G.~F. Lawler.
\newblock Hausdorff dimension of cut points for {B}rownian motion.
\newblock {\em Electron. J. Probab.}, 1(2):1--20, 1996.

\bibitem{G98}
G.~F. Lawler.
\newblock Strict concavity of the intersection exponent for {B}rownian motion
  in two and three dimensions.
\newblock {\em Math. Phys. Electron. J.}, 4(5):1--67, 1998.

\bibitem{MR2129588}
G.~F. Lawler.
\newblock {\em Conformally invariant processes in the plane}, volume 114 of
  {\em Mathematical Surveys and Monographs}.
\newblock American Mathematical Society, Providence, RI, 2005.

\bibitem{La2020}
G.~F. Lawler.
\newblock The infinite two-sided loop-erased random walk.
\newblock {\em Electron. J. Probab.}, 25(87):1--42, 2020.

\bibitem{BL90}
G.~F. Lawler and K.~Burdzy.
\newblock Non-intersection exponents for {B}rownian paths. {P}art {I}.
  {E}xistence and an invariance principle.
\newblock {\em Probab. Theory Related Fields}, 84:393--410, 1990.

\bibitem{LP2000}
G.~F. Lawler and E.~E. Puckette.
\newblock The intersection exponent for simple random walk.
\newblock {\em Combin. Probab. Comput.}, 9(5):441--464, 2000.

\bibitem{LSW01a}
G.~F. Lawler, O.~Schramm, and W.~Werner.
\newblock Values of {B}rownian intersection exponents. {I}. {H}alf-plane
  exponents.
\newblock {\em Acta Math.}, 187(2):237--273, 2001.

\bibitem{LSW01b}
G.~F. Lawler, O.~Schramm, and W.~Werner.
\newblock Values of {B}rownian intersection exponents. {II}. {P}lane exponents.
\newblock {\em Acta Math.}, 187(2):275--308, 2001.

\bibitem{LSW02a}
G.~F. Lawler, O.~Schramm, and W.~Werner.
\newblock Analyticity of intersection exponents for planar {B}rownian motion.
\newblock {\em Acta Math.}, 189(2):179--201, 2002.

\bibitem{LSW02}
G.~F. Lawler, O.~Schramm, and W.~Werner.
\newblock Values of {B}rownian intersection exponents. {III}. {T}wo-sided
  exponents.
\newblock {\em Ann. Inst. H. Poincar\'e Probab. Stat.}, 38(1):109--123, 2002.

\bibitem{MR2112127}
G.~F. Lawler, O.~Schramm, and W.~Werner.
\newblock On the scaling limit of planar self-avoiding walk.
\newblock In {\em Fractal geometry and applications: a jubilee of {B}eno\^\i t
  {M}andelbrot, {P}art 2}, volume~72 of {\em Proc. Sympos. Pure Math.}, pages
  339--364. Amer. Math. Soc., Providence, RI, 2004.

\bibitem{LV12}
G.~F. Lawler and B.~Vermesi.
\newblock Fast convergence to an invariant measure for non-intersecting
  3-dimensional {B}rownian paths.
\newblock {\em ALEA Lat. Am. J. Probab. Math. Stat.}, 9(2):717--738, 2012.

\bibitem{MR1796962}
G.~F. Lawler and W.~Werner.
\newblock Universality for conformally invariant intersection exponents.
\newblock {\em J. Eur. Math. Soc.}, 2(4):291--328, 2000.

\bibitem{MR4017129}
X.~Li and D.~Shiraishi.
\newblock One-point function estimates for loop-erased random walk in three
  dimensions.
\newblock {\em Electron. J. Probab.}, 24(111):1--46, 2019.

\bibitem{Ma2009}
R.~Masson.
\newblock The growth exponent for planar loop-erased random walk.
\newblock {\em Electron. J. Probab.}, 14(36):1012--1073, 2009.

\bibitem{BM1}
P.~M\"{o}rters and Y.~Peres.
\newblock {\em Brownian motion}, volume~30 of {\em Cambridge Series in
  Statistical and Probabilistic Mathematics}.
\newblock Cambridge University Press, Cambridge, 2010.
\newblock With an appendix by Oded Schramm and Wendelin Werner.

\bibitem{BM}
S.~C. Port and C.~J. Stone.
\newblock {\em Brownian motion and classical potential theory}.
\newblock Probability and Mathematical Statistics. Academic Press [Harcourt
  Brace Jovanovich, Publishers], New York-London, 1978.

\bibitem{MR3877544}
A.~Sapozhnikov and D.~Shiraishi.
\newblock On {B}rownian motion, simple paths, and loops.
\newblock {\em Probab. Theory Related Fields}, 172(3-4):615--662, 2018.

\bibitem{SS2010}
O.~Schramm and J.~E. Steif.
\newblock Quantitative noise sensitivity and exceptional times for percolation.
\newblock {\em Ann. of Math. (2)}, 171(2):619--672, 2010.

\bibitem{LERW3exp}
D.~Shiraishi.
\newblock Growth exponent for loop-erased random walk in three dimensions.
\newblock {\em Ann.~Probab.}, 46(2):687--774, 2018.

\bibitem{sznitman2013scaling}
A.-S. Sznitman.
\newblock On scaling limits and {B}rownian interlacements.
\newblock {\em Bulletin of the Brazilian Mathematical Society, New Series},
  44:555--592, 2013.

\bibitem{BN2021}
J.~van~den Berg and P.~Nolin.
\newblock Near-critical 2{D} percolation with heavy-tailed impurities, forest
  fires and frozen percolation.
\newblock {\em Probab. Theory Related Fields}, 181(1-3):211--290, 2021.

\bibitem{vermesi2008intersection}
B.~Vermesi.
\newblock Intersection exponents for biased random walks on discrete cylinders.
\newblock {\em arXiv preprint arXiv:0810.0572}, 2008.

\bibitem{MR1905353}
W.~Werner.
\newblock Critical exponents, conformal invariance and planar {B}rownian
  motion.
\newblock In {\em European {C}ongress of {M}athematics, {V}ol. {II}
  ({B}arcelona, 2000)}, volume 202 of {\em Progr. Math.}, pages 87--103.
  Birkh\"auser, Basel, 2001.

\bibitem{MR946350}
R.~Wittmann.
\newblock A uniform boundary {H}arnack inequality on nontangentially accessible
  domains.
\newblock {\em J. Reine Angew. Math.}, 387:69--96, 1988.

\bibitem{MR513884}
J.~M.~G. Wu.
\newblock Comparisons of kernel functions, boundary {H}arnack principle and
  relative {F}atou theorem on {L}ipschitz domains.
\newblock {\em Ann. Inst. Fourier (Grenoble)}, 28(4):147--167, vi, 1978.

\end{thebibliography}

\end{document}